\documentclass{article}

\usepackage{amsmath,amssymb,amsthm,graphicx,color,sidecap,wrapfig}
\usepackage[inner=32mm,outer=32mm,top=28mm,bottom=28mm]{geometry}
\usepackage{framed,color,perpage}

\definecolor{gray}{rgb}{0.5,0.5,0.5}

\newcommand{\ve}{\varepsilon}

\newcommand{\textred}{\textcolor{red}}

\def \div {{\rm div}}

\MakePerPage{footnote}

\numberwithin{equation}{section}

\newtheorem{thm}{Theorem}%[section]
\numberwithin{thm}{section}
\newtheorem{rem}[thm]{Remark}
\newtheorem{prop}[thm]{Proposition}
\newtheorem{lem}[thm]{Lemma}
\newtheorem{cor}[thm]{Corollary}
\newtheorem{conj}[thm]{Conjecture}

\usepackage{authblk}

\title{Emergence of
	traveling waves and their stability  in a free boundary model of cell motility}

\author[1]{Volodymyr Rybalko}
\author[2]{Leonid Berlyand}
\affil[1]{B.Verkin Institute for Low Temperature Physics and Engineering of NASU, 47 Nauky ave, Khariv 61103, e-mail: vrybalko@ilt.kharkov.ua} %v_rybalko@yahoo.com}
\affil[2]{Department of Mathematics, Huck Institutes of Life   Sciences and Materials Research Institute at the  Penn State University,
		University Park, PA, 16802, USA, e-mail: lvb2@psu.edu}

\begin{document}

\maketitle

%
%\title{Emergence of
%	traveling waves and their stability  in a free boundary model of cell motility}
%	%\\Traveling waves and bistability in a free boundary model of cell motility}

	%\begin{abstract}
	% \texttt{We introduce a two-dimensional model of an active gel drop (cytoskeleton gel)
	% that is a free boundary problem with Keller-Segel PDEs. The key ingredients of this model are the Darcy law for overdamped motion and Hele-Shaw type boundary conditions (Young -Laplace equation for pressure and continuity of velocities).
	% We first  show that  radially symmetric steady state solutions become unstable and bifurcate  to traveling wave  solutions. 
	% Next we  establish linear and nonlinear stability properties of the steady states and traveling waves.   We show that  linear stability analysis is inclusive for both steady states and traveling waves. Therefore we  use invariance properties to prove nonlinear stability of steady states.}
	%\end{abstract}

	\begin{abstract}
		%	\texttt
		{We introduce a two-dimensional  Hele-Shaw type free boundary model for motility of eukaryotic cells on substrates.
			% The main features of motility are determined by the properties of the cytoskeleton (active) gel.
			% gel which is the focus of our study.
			%drop (cytoskeleton gel)
			%	that is a free boundary problem with Keller-Segel PDEs. 
			The key ingredients of this model are the Darcy law for overdamped motion of the cytoskeleton  gel (active gel) coupled with  advection-diffusion equation for myosin density leading to elliptic-parabolic Keller-Segel  system. This system is supplemented with Hele-Shaw type boundary conditions: Young-Laplace equation for pressure and continuity of velocities. We first  show that  radially symmetric stationary solutions become unstable and bifurcate  to traveling wave solutions at a critical value of the total myosin mass.  Next we perform linear stability analysis of these traveling wave solutions  and identify the type of bifurcation (sub- or supercritical). Our study sheds light on the mathematics underlying  instability/stability transitions in  this model.  Specifically, we show that these transitions occur via generalized eigenvectors of the linearized operator.  
%\textcolor{red}{These eigenvectors appear due to non  self-adjointness of this operator, which is a signature of  active matter PDE models.}
			
%			due to advection of the myosin diffusion . 
%			% that model cytoskeleton gel.
%			% of the cell.
%			We first  show that  radially symmetric steady state solutions become unstable and bifurcate  to traveling wave solutions at a critical value of the total myosin mass.  Next we perform linear stability analysis and identify the type of bifurcation (sub- or supercritical). Here features specific to active matter result in a non self-adjoint operator with  multiple eigenvalues that asymptotically merge at small velocities with corresponding eigenvectors becoming collinear. \sout{ This spectral analysis reveals how parameters  determine the bifurcation type  and it also\
%			establishes bistability.}}
%		
			
		%and show that stability/instability is determined by the behavior of total m. 		
		%We show that  linear stability analysis is inconclusive for both steady states and traveling waves. Therefore we  use invariance properties to prove nonlinear stability of steady states.}
		%\part{title}
}
	\end{abstract}
	
	\section{Introduction}
	
	%Biology: we concentrate
	%on the myosin generated forces and movements.
	
	%Consider the flow $u$ satisfying the Darcy law (see Callan-Jones, Joanny,  Prost 2008, PRL 100, 258106; Blanch-Mercader, %Casademunt 2013, PRL 110, 078102)

	Motion (motility)  of living cells has been the subject of extensive  studies in biology, soft-matter physics and more recently in mathematics.   Living cells     are  primarily driven by cytoskeleton  gel dynamics. The study of  cytoskeleton  gels  led to a recent development of the so-called  
	``Active gel physics", see \cite{ProJulJoa2015}.
	
	The key element  of this motion is cell polarity (asymmetry, e.g., the cell has a front and  back), which enables cells to carry out specialized functions. Therefore understanding  of  cell motility and polarity   are the fundamental issues in cell biology.  Also, motion of specific cells such as keratocytes  in the cornea   is of medical relevance as they are involved, e.g., in wound healing after eye surgery or injuries.    Moreover keratocytes are perfect for experiments and modeling since they are naturally found on 
	flat surfaces, which allows capturing 
	the main features of their motion by spatially two dimensional models. The typical
	modes of motion of keratocytes in  cornea as well as in fishscales are rest (no movement at all) or steady motion with
	fixed shape, speed, and direction \cite{Ker_etal2008}, \cite{BarLeeAllTheMog2015}.
		%, Z. Pincus, G. Allen, E. Barnhart, G. Marriott, A. Mogilner, and J. Theriot, Mechanism
		%of shape determination in motile cells, Nature, 453(7194):475{480, 2008., 
		 That is why it is important  to study the
	stationary solutions  and traveling waves that describe resting cells  and 
	steadily moving cells respectively.
	
	The two leading  mechanisms of cell motion are  protrusion generated by polymerization of actin filaments (more precisely, filamentous 
	actin or F-actin) and  contraction  due to myosin motors \cite{Ker_etal2008}.
	%, Z. Pincus, G. Allen, E. Barnhart, G. Marriott, A. Mogilner, and J. Theriot, Mechanism
	%of shape determination in motile cells, Nature, 453(7194):475{480, 2008.,
	The goal of this work is to study  the contraction-driven cell motion, since it dominates motility initiation \cite{RecPutTru2015}. 
	% when   polymerization is negligible since it is balanced by depolymerization (see, e.g.,   \cite{BlaCas2013} and \cite{EtcMeuVoi2017}  for the complementary works on polymerization without myosin contraction).
	 To this end we  introduce and  investigate  a 2D model with free boundary that generalizes 1D free boundary model from 
	\cite{RecPutTru2013}, \cite{RecPutTru2015}. Despite of its simplicity this 1D model captures the  bifurcation of  stationary  solutions to traveling waves, which is the signature  property of cell motility.   While  mathematical  analysis in 2D is obviously much more involved than in 1D, especially in the  free boundary setting, the  results of the bifurcation analysis in 2D  agrees with 1D case \cite{RecPutTru2013} and \cite{RecPutTru2015}, in particular, both models exhibit a supercritical bifurcation.   However,  modeling of the  important phenomenon of cell shape evolution requires consideration beyond 1D and our results  captures breaking of  the shape symmetry, as depicted in Fig. \ref{fig:john}, which is an important biological phenomenon, see, e.g.,  \cite{BarLeeAllTheMog2015} and \cite{ZieAra2015}. Moreover, the main results of this work, in particular the explicit asymptotic formula \eqref{formula_for_lambda_second} for the eigenvalue, that decides on stability,  provide  a new insight for {\it both 1D and 2D models}. 
	Various 2D  free boundary models  of active gels were introduced in, e.g., \cite{BarLeeAllTheMog2015},  \cite{CalJonJoanPro2008}, \cite{BlaCas2013}.   
	The problems in  \cite{CalJonJoanPro2008} and \cite{BlaCas2013}
	model the polymerization  driven cell motion   when   myosin contraction is dominated by polymerization, which naturally  complements present work.  These models extend the classical Hele-Shaw model by adding  fundamental  active matter features  such as the presence of persistent motion  modeled by traveling wave solution.  
		%\textcolor{red}{ Although  model from  \cite{CalJonJoanPro2008} \cite{BlaCas2013}  looks similar to the classical Hele-Shaw model, these two models are  different in some fundamental aspects such as  presence of persistent motion modeled by traveling wave solution. }
	%\textcolor{green}{More recently  a 2D model of the intracellular dynamics with fixed cell shape as a disc was introduced and analyzed \textcolor{red}{numerically analytically} in \cite{EtcMeuVoi2019}.} 
The Keller-Segel system with free boundaries as a model for contraction driven motility was first introduced in \cite{RecPutTru2013}, in 1D setting.
	%and  \cite{BlaCas2013} is a new prototype system in the context of Laplacian growth phenomena.
		%However, the problem
	%differs from viscous fingering in some fundamental
	%aspects and thus it can be viewed also as a new prototype
	%system in the context of Laplacian growth phenomena.
Its 2D counterpart  introduced  and analyzed numerically in   \cite{BarLeeAllTheMog2015} accounts for both polymerization  and myosin contraction. A simplified version of this  model was studied analytically in \cite{BerFuhRyb2018} where the traveling wave solutions were established. 
% It was also observed in \cite{BerFuhRyb2018} that this model reduces to  the Keller-Segel  system in a free boundary setting. 
Note that the Keller-Segel system  in fixed domains  appears in various chemotaxis models and it has been extensively studied in mathematical literature due to the finite time blow-up phenomenon caused by the cross-diffusion term (\cite{TaoWin2012}, p.1903)  in dimensions 2 and higher, see also \cite{CalPerYas2018} for traveling waves in the 1D flux-limited Keller-Segel model. 
	% where the corresponding traveling wave solutions were analyzed in the simplest 1D setting.
 We also mention  closely related  free boundary problems in tumor growth models. The key differences  are that in the latter models the  area  of  domain undergoes significant changes and  there is no persistent motion (see, e.g.,  \cite{Fri2004},
		\cite{PerVau2015}, and  \cite{MelRoc2017}).  
	%}

While in the model  \cite{BarLeeAllTheMog2015} the kinematic condition 
at the free boundary contains curvature, in the present work we assume continuity of velocities of the gel at the  cell edge  following the 1D  model introduced in \cite{RecPutTru2013}. Still the curvature appears in the force balance on the boundary since 
we adapt the Young-Laplace equation for the pressure. This  provides the same   regularizing effect as in  the classical 2D Hele-Shaw model. 

The focus of this work is on  understanding of  transitions from unstable rest to stable motion in the model.  Specifically, we establish existence of traveling wave solutions and perform their stability analysis. To show existence of a family of traveling 
waves we employ bifurcation analysis of the family of radially symmetric stationary solutions, following the idea originally proposed in  \cite{FriRei2001} in the the framework of a tumor growth model and followed in many subsequent works on such models, e.g. \cite{FriHu2008}, \cite{HaoHauHuLiuSomZha2012}. While aforementioned works deal with bifurcation from  radial to non-radial stationary solutions  via eigenvectors, in the present work we establish existence of traveling wave solutions bifurcating via {\it generalized} eigenvectors rather than eigenvectors. 
Similarly to \cite{FriHu2008} we use the Crandall-Rabinowitz bifurcation theorem to justify bifurcation to a family of traveling wave solutions parametrized by their velocity  $V$. However, the functional framework  for application of this theorem  significantly differs from that is used for tumor growth models.

%applications of this theorem requires developing significantly different functional framework which significantly differs from that is used for tumor growth models.

The main mathematical novelty of this work is in the study of 
spectral properties of the operator $\mathcal{A}(V)$  linearized around traveling wave solutions.  The spectrum of $\mathcal{A}(V)$ near zero has rather interesting asymptotic behavior in the limit of small traveling wave velocity due to presence of  non trivial Jordan chains leading to generalized eigenvectors for multiple zero eigenvalue. Specifically, $\mathcal{A}(V)$ has zero eigenvalue  of multiplicity five for  $V=0$ that splits into  zero eigenvalue of multiplicity  four and  simple non zero eigenvalue   $\lambda(V)\not=0$ for $V\not=0$  whose sign determines stability of traveling waves.  
%in particular,  transition of the zero  eigenvalue $\lambda(V)=0$ for  $V=0$  into a small non-zero eigenvalue $\lambda(V)\not=0$ for $V\not=0$  whose sign determines stability of TWs.
The main result of this work is an explicit asymptotic formula  \eqref{sovsemfinal_formula_for_lambda}  for   $\lambda(V)$, which determines stability of traveling waves  in terms of the total  myosin mass and a special eigenvalue $E$ describing movability (see Remark \ref{movability}) of  stationary solutions.

The spectral analysis of $\mathcal{A}(V)$ has two main challenges. First,   neither its coefficients 
%of the differential operator $\mathcal{A}(V)$ (see \eqref{Hact_flow_lin_op}-\eqref{Hmyosin_lin_BC}) 
nor spatial domain $\Omega(V)$
are explicitly known for $V\not=0$,  since  they are expressed via solution pair $\phi=\Phi(x, y, V)$, $\Omega=\Omega(V)$  
of  the free boundary  problem \eqref{tw_Liouvtypeeq}--\eqref{addit_cond_tw} for traveling waves.
The second principal challenge  is  due to non self-adjointness of  the operator $\mathcal{A}(V)$, which is the signature of active matter models.
 
%\textcolor{red}{Non new paragraph}
%Moreover, the spectrum of $\mathcal{A}(V)$ near zero has rather interesting asymptotic behavior in the limit of small traveling wave velocity due to presence of  non trivial Jordan chains leading to generalized eigenfunctions 
  %for multiple zero eigenvalue.
  
 We next briefly describe main steps in the spectral analysis of  $\mathcal{A}(V)$. 
First we  assume  that the perturbations have the 
 natural symmetry of traveling wave solutions  and show 
that at zero velocity the linearized operator $\mathcal{A}(0)$ (restricted to the space of symmetric vectors)  has zero eigenvalue with multiplicity three (rather than multiplicity five in the general non-symmetric case).  For $V \not =0$ the operator %$\mathcal{A}(V)$ 
has zero eigenvalue of multiplicity two with an eigenvector representing infinitesimal shifts of traveling wave solutions and generalized eigenvector obtained by taking derivative of traveling wave solutions in velocity. It has also another small eigenvalue $\lambda(V)$ whose corresponding eigenvector asymptotically merges with the eigenvector representing infinitesimal shifts (this  feature stands in contrast with orthogonality of eignevectors in the self-adjoint case).  Moreover, %due to non self-adjointness of $\mathcal{A}(V)$, 
 finding  the principlal first  term in the asymptotic expansion of the eigenvalue $\lambda(V)$  requires a {\it four term} ansatz for the eigenvector, which has an interesting structure: the first two  terms are the eigenvector and the generalized eigenvector of $\mathcal{A}(V)$ for the zero eigenvalue (see pairs $m_i$, $\rho_i$, $i=1,2$ in \eqref{AnSaTzform}- \eqref{AnSaTzforrho}).  The resulting asymptotic formula \eqref{sovsemfinal_formula_for_lambda} for the eigenvalue  $\lambda(V)$ is remarkably simple, the principal term of $\lambda(V)$ is given in terms of two key physical quantities: movability of stationary solutions and the dependence of the total myosin mass $M (V)$ on the  traveling wave velocity $V$ (see explanation after  the main Theorem \ref{ThmNoSymmetry}).  However, its justification is rather involved and requires passing to the invariant subspace complementary to the generalized eigenspace of zero eigenvalue. In order to describe this 
invariant subspace we study generalized eigenvector 
of  the adjoint operator that  exibits singular behavior (it blows up) as $V\to 0$. Finally,  
we extend the results for symmetric perturbations to general perturbations of the traveling wave solutions.  The key observation here is that the  multiplicity of  zero eigenvalue  changes from two in the symmetric case to four in the general non-symmetric case. The additional eigenvector and generalized eigenvector are  the infinitesimal shifts in the direction orthogonal to motion and infinitesimal rotations respectively.  Thus in general case  there are five eigenvalues  (counted with multiplicity) of 
$\mathcal{A}(V)$ near zero,  but only one of them is nonzero and it determines stability of traveling waves.

	{\bf Acknowledgments}. Volodymyr Rybalko is grateful to PSU Center for Mathematics of Living and Mimetic Matter, and  to PSU Center for Interdisciplinary Mathematics  for support of his  two stays at Penn State. His travel was also supported by NSF grant DMS-1405769.  The work L. Berlyand  was partially supported by NSF grant DMS-2005262. We thank  our colleagues R. Alert, I. Aronson, J. Casademunt, J.-F. Joanny, N. Meunier, A. Mogilner, J. Prost  and L. Truskinovsky for useful discussions and suggestions on the model. We also express our gratitude to the  members of the L. Berlyand's PSU  research team, R. Creese,  M. Potomkin,  and A. Safsten for careful reading and help in  the preparation of the manuscript.   We  gratefully  acknowledge numerical implementation  by A. Safsten  of the asymptotic expansions of traveling wave solutions established in Theorem \ref{biftwtheorem} (Fig.\ref{fig:john}).  The computational work of  A. Safsten  was  partially supported by NSF grant DMS-2005262 and the detailed results are presented in \cite{SafRybBer2021}.
	\section{The model}
	\label{Section_the_model}
	
	%\textcolor{red}
	%{
	We consider a 2D model of motion of  a cell  on a   flat substrate %an active gel drop 
	which occupies a domain $\Omega(t)$ with free boundary. 
	The flow of the acto-myosin network  inside the domain $\Omega(t)$ is described by the velocity field $u$. In the adhesion-dominated regime (overdamped flow) \cite{CalJonJoanPro2008}, \cite{BlaCas2013} $u$  obeys  the Darcy law 
	\begin{equation}
	\label{DarcyLaw}
	-\nabla p=\zeta u \quad \text{in}\ \Omega(t),
	\end{equation}
	where $-p$ stands for the scalar stress ($p$ is the pressure) and $\zeta$  is the constant effective
	adhesion drag coefficient.  The actomyosin network  is modeled by  a compressible fluid (incompressible cytoplasm fluid
	can be squeezed easily into the dorsal direction in the cell \cite{NicNovPulRumBraSleMog2017}).
	%We consider compressible gel 
	%(the actomyosin etwork is a compressible fluid,  incompressible cytoplasm fluid
	%can be squeezed easinly into the dorsal direction in the cell \cite{NicNovPulRumBraSleMog2017}). 
	The  main modeling assumption of this 
	%work
	%  postulating the following 
	is the following constitutive law 
	for the scalar stress $-p$
	\begin{equation}
	\label{constitutiveEQ}
	-p= \mu\div u +k m-p_{\rm h} \quad \text{in}\ \Omega(t)
	\end{equation}
	where $\mu\div u$ is the hydrodynamic stress ($\mu$ being the effective bulk 
	viscosity of the gel), the term 
	$k m$ is the active component of the stress 
	which is proportional to the density 
	%on the right hand side of \eqref{constitutiveEQ} is
	$m=m(x,y,t)>0$ 
	%is the active 
	%component of the stress created by 
	of myosin motors with a constant contractility coefficient $k>0$,  
	$p_{\rm h}$ is the constant %effective
	hydrostatic  pressure (at equilibrium).
	Throughout this work we assume that 
	the effective bulk viscosity $\mu$ and the contractility coefficient $k$ in \eqref{constitutiveEQ} are scaled to $\mu=1$, 
	$k=1$. We prescribe the following condition on the boundary
	\begin{equation}\label{Young_Laplace_eq}
	p+p_{\rm e}=\gamma \kappa \quad \text{on} \ \partial \Omega(t),
	\end{equation}
	known as the Young-Laplace equation.  In \eqref{Young_Laplace_eq} $\kappa$ denotes the curvature, $\gamma>0$ 
	is a constant coefficient and $p_{\rm e}$ is the effective elastic restoring force which describes the mechanism of 
	approximate conservation of the area due to the membrane-cortex tension. The  elastic restoring force $p_{\rm e}$ generalizes  the one-dimensional nonlocal spring condition introduced in \cite{RecPutTru2013}, \cite{RecPutTru2015}, see  more recent work \cite{PutRecTru2018} which also introduces   the cell volume regulating pressure\footnote{The authors are grateful to L.Truskinovsky for bringing 
		\cite{PutRecTru2018} to their attention 
		and helpful discussions on bifurcations 
		during the preparation of the manuscript.},  
 and  we  similarly assume the simple linear dependence of $p_{\rm e}=p_{\rm e}(|\Omega|)$ on the area\footnote{An alternative way to this mean field elasticity approach (used to regularize the minimal model)  could be incorporating the Kelvin-Voigt model which accounts for the elastic
response at long time scales. To this end one can  introduce the intracellular density $\varrho$, whose transport is governed, e.g., 
by $\partial_t \varrho+{\rm div}(\varrho u)=0$ and modify the constitutive law \eqref{constitutiveEQ} by a term $P(\varrho)$ with appropriate linear or nonlinear function $P$. For a discussion of  different approaches of elastic regularazation of the minimal model in 1D case, including also Maxwell model, 
we address  interested reader to  \cite{RecTru2013}.  }: 
	\begin{equation}\label{k_e}
	p_{\rm e}=k_{\rm e}( |\Omega_h|-|\Omega|)/|\Omega_h|,
	\end{equation} 
	where $k_{\rm e}$ is the inverse compressibility  coefficient (characterizing membrane-cortex elastic tension), $|\Omega_h|$ is the  area of the reference configuration $\Omega_h$  in which  $p_e=0$, c.f. vertex models  (e.g., formula (2.2) in \cite{AltGanSal2015}). 

%\textcolor{red}{It is necessary to discuss how your nonlocal constraint can be eventually dropped and replaced by some alternative constitutive structure. See about this in our early  paper ( attached), section III (elastic regularization).}
	%https://www.ncbi.nlm.nih.gov/pmc/articles/PMC5379026/#targetText=In%20vertex%20models%2C%20the%20epithelial,the%20positions%20of%20the%20vertices.
	
	%, when $p=p_h$ in \eqref{}
	%(e.g, a simple reference configuration is a radial cell at rest $m=const, u=0, p_h=km +\gamma \kappa$, $\Omega_h$ is a disk). 
	%The condition  \eqref{Young_Laplace_eq} means that pressure $p$ adjusts rapidly to the area.   
	%$p= \gamma \kappa$, , $p_h=km +\gamma \kappa$, 
	%$m=const$, $\Omega=\Omega_h$ could be a  radial steady state.
	
	%the surface tension coefficient. 
	The evolution of the myosin motors density is described by the advection-diffusion equation
	\begin{equation}
	\label{myosin_equations}
	\partial_t m=\Delta m -\div (u m) \quad \text{in} \ \Omega(t)
	\end{equation}
	and no flux boundary condition in the moving domain
	\begin{equation}
	\label{myosin_bc}
	\partial_\nu m =(u\cdot \nu - V_\nu)m \quad\text{on}\ \partial\Omega(t),
	\end{equation}
	where $\nu$ stands for the outward pointing normal vector and $V_\nu$ is the normal velocity of the domain $\Omega(t)$. 
	Finally, we assume continuity of velocities on the boundary 
	\begin{equation}
	\label{KinematicBC}
	V_\nu=u\cdot\nu,
	\end{equation} 
	so that \eqref{myosin_bc} becomes the homogeneous Neumann condition. Combining    \eqref{DarcyLaw}--\eqref{KinematicBC} yields
	% a closed set of equations that forms 
	a  free boundary model of the cell motility  investigated in this work.  While there are several models of cell motility in literature (both free boundary and phase field models), in this work we perform  analytical study  of stability of stationary and persistently moving states in the model \eqref{DarcyLaw}--\eqref{KinematicBC}. Moreover, the  analysis of the linearized problem can be used to establish local in time  existence and uniqueness of solutions. 
	%  While  proofs of these  results are rather straightforward, they are quite technical and therefore we do not present       
	%}
	
	It is convenient to introduce  the potential for the velocity field $u$  using \eqref{DarcyLaw}:
	\begin{equation}\label{phi_def}
	u= \nabla \phi=- \nabla \frac{1}{\zeta} p, \quad  
	\phi:=-\frac{1}{\zeta}(p+p_h),
	\end{equation}
	and  rewrite problem \eqref{DarcyLaw}--\eqref{KinematicBC} in the form
	\begin{equation}
	\label{actflow_in_terms_of_phi}
	\Delta \phi +m=\zeta \phi %+\textcolor{green}{\text{\sout{
	%$p_{\ast}(|\Omega(t)|)$}}}
	\quad\text{in}\ \Omega(t),
	\end{equation}
	\begin{equation}\label{actin_bc_potential}
	\zeta\phi={p_{\ast}(|\Omega(t)|)}-\gamma \kappa\quad\text{on}\ \partial\Omega(t),
	\end{equation}
	\begin{equation}\label{actin_bc_normal}
	V_\nu=\partial_\nu \phi \quad\text{on}\ \partial\Omega(t),
	\end{equation}
	\begin{equation}
	\label{myosin_equations_I}
	\partial_t m= \Delta m -\div (m \nabla \phi ), \quad \text{in} \ \Omega(t),
	\end{equation}
	\begin{equation}
	\label{myosin_bc_I}
	\partial_\nu m =0 \quad\text{on}\ \partial\Omega(t),
	\end{equation}
	where %$\phi:=-\frac{1}{\zeta}(p+p_e)$ and 
	we  introduced  the notation 
	\begin{equation}\label{change}
	p_{\ast}:=p_{\rm h}+p_{\rm e}=
	p_{\rm h}-k_{\rm e}( |\Omega(t)|-|\Omega_{\rm h}|)/|\Omega_{\rm h}|
	\end{equation}
	for the sum of
	the hydrostatic pressure $p_{\rm h}$ and the effective elastic restoring force $p_{\rm e}$. We  assume that the area  $|\Omega(t)|$ is  such that (e.g., $|\Omega|$ is close to $|\Omega_{\rm h}|$)
	%In what follows  $p_{\ast}$ is considered as a given function of the domain area $|\Omega(t)|$, which is not necessary 
	%linear,
	\begin{equation}
	\label{gelswelling}
	p_{\ast}=p_{\ast}(|\Omega(t)|) >0.
	\end{equation}
	Moreover,  we consider the  coefficient $k_{\rm e}$ to be sufficiently large so that
	%Specifically, $p_{\ast}$ is a smooth positive decreasing function of $|\Omega(t)|$ which has sufficiently large negative derivative to
	it penalizes changes of the area. For instance, it prevents from shrinking of $\Omega$ to a point or from infinite expanding. The precise lower bound  on  $ k_{\rm e}=-|\Omega_{\rm h}|\, p_{\ast}^{\prime}(|\Omega(t)|) $ is given below in \eqref{areashrinkingprevention}.
	% see also Remark \ref{shrinking_expansion}.
	% or exact preservation of area 
	%\begin{equation}
	%\label{areapreservation}
	%p_{\ast}=\frac{1}{|\Omega(t)|} \int_{\Omega(t)} (m-\zeta \phi)dxdy
	%\end{equation}
	
	\begin{rem}\label{evolution}
		In this work we consider  the problem \eqref{actflow_in_terms_of_phi}--\eqref{myosin_bc_I}  as an evolution problem that determines a dynamical system in the phase space of two unknowns $m(x,y, t)$ and  $\Omega (t)$, while the potential $\phi(x,y,t)$ is considered as an additional 
		unknown function defining evolution of the free boundary. Indeed, for given $\Omega (t)$ and $m(x,y,t)$	the function $\phi(x,y,t)$ is obtained as the unique solution of the elliptic problem 
		\eqref{actflow_in_terms_of_phi}--\eqref{actin_bc_potential}, and its normal derivative $\partial_\nu \phi$ 
		defines normal velocity of the boundary $\partial \Omega(t)$ due to \eqref{actin_bc_normal}, see also \eqref{KinematicBC}-\eqref{phi_def}. Problem \eqref{actflow_in_terms_of_phi}--\eqref{myosin_bc_I} is 
		supplied with initial conditions for $m$ and $\Omega$ and it is natural not to include the unknown $\phi$ 
		into the phase space of this evolution problem but rather in the definition of the operator governing the semi-group corresponding to the dynamical system in this phase space that defines the evolution of  $m$ and $\Omega$ (see, e.g. \eqref{operator_form}).
	\end{rem}

	%\begin{rem}\label{young-laplace}
	%To provide an additional insight on the physical nature of effective homeostatic pressure $p_{\ast}$ we observe that introducing the internal pressure 
	%$p_{int}= -\mu\div u -k m$ and the external pressure $p_{ext}=-p_{\ast}$ we can rewrite 
	%%the constituitive law \eqref{constitutiveEQ} and 
	%the Young-Laplace  equation \eqref{Young_Laplace_eq} in the form
	%\begin{equation}\label{young-laplace1}
	%p_{int}-p_{ext}=\gamma \kappa \quad \text{on} \ \partial \Omega(t),
	%\end{equation}
	%so that  the left hand side of \eqref{young-laplace1} is the pressure difference across the interface $\partial \Omega(t)$.  %between  the internal  pressure of the gel %$p_{int}$  and  the external  pressure 	%$p_{ext}=-p_{\ast}$ due to cortex/membrane elasticity.
	%\end{rem}
	
 For technical simplicity we  first assume that solutions of problem \eqref{actflow_in_terms_of_phi}--\eqref{myosin_bc_I} is symmetric with respect to $x$-axis (it suffiices to assume such a symmetry of initial data).
% Specifically we assume this symmetry of the initial data, domain  $\Omega(0)$ and  $m(x,y, t=0)$,  then this symmetry holds for  
%	$\forall t >0$. 
Subsequently we relax this assumption  in Theorem \ref{ThmNoSymmetry} to obtain a complete characterization of   linear stability of the traveling wave solutions.
%} 

	\section{Linear stability analysis of radially symmetric stationary solutions}
	\label{section_lin_an_stst}
	
	In the class of  radially symmetric  stationary  solutions with constant density $m=const$ for a given radius $R>0$  there exists the unique radial  solution of the problem \eqref{actflow_in_terms_of_phi}--\eqref{myosin_bc_I}:
	% $m=const$ implies $\phi=const$ and then find constants.
	%	Problem \eqref{actflow_in_terms_of_phi}--\eqref{myosin_bc_I} possesses a family of radially symmetric  stationary solutions with
	%constant $\phi$ and $m$. For a given radius $R>0$  there exists the unique radial  solution of \eqref{actflow_in_terms_of_phi}--\eqref{myosin_bc_I}:
	%$\phi=\phi_0$  and  $m=m_0$,
	%is obtained from \eqref{actin_bc_potential} and \eqref{actflow_in_terms_of_phi}
	%in the domain $\Omega(t)=B_R$ 
	%by the direct substitution ($B_R$ is the disk with radius $R$): 
	\begin{align}\label{steady_state}
	\Omega=B_R,\quad m_0:=-\gamma/ R+p_{\ast}(\pi R^2),
	\quad
	\zeta \phi_0={p_{\ast}(\pi R^2)}-\gamma/ R.
	\end{align}

	%In our perturbative analysis we assume the symmetry about the $x$-axis. 
To describe evolution of  perturbations of \eqref{steady_state}, %thIn the evolution problem \eqref{actflow_in_terms_of_phi}--\eqref{myosin_bc_I} 
	it is convenient to use the polar coordinate system $(r,\varphi)$, 
%whose origin is moving with the 
%	domain, 
	\begin{equation}\label{def_of_Omega}
	\Omega=\left\{(x=r\cos\varphi, y=r\sin\varphi); 0\leq r<R+\rho(\varphi,t)\right\}.
	\end{equation}
Then	linearizing problem 
	\eqref{actflow_in_terms_of_phi}--\eqref{myosin_bc_I}
	%\label{myosin_equations_I} ...
	around a radially symmetric reference stationary solution \eqref{steady_state},  we get the following problem
	\begin{equation}
	\label{RadSymLinearized1}
	\partial_t \rho%+\dot X_c\cos\varphi
=\partial_r \phi \quad \text{on}\ \partial B_R, 
	\end{equation}
	\begin{equation}
	\label{ur_e_dlyacrtin}
	\Delta \phi + m=\zeta \phi \quad\text{in}\ B_R,
	\end{equation}
	\begin{equation}
	\label{krizna_linearized}
	\phi=\frac{p_{\ast}^\prime(\pi R^2)R}{\zeta}\int_{-\pi}^{\pi}\rho(\varphi)d\varphi+\frac{\gamma}{R^2\zeta}(\partial^2 _{\varphi\varphi}\rho+\rho) \quad \text{on}\ \partial B_R,
	\end{equation}
	\begin{equation}
	\label{ur_e_dlya_m}
	\partial_t {m}=\Delta {m} - m_0 \Delta {\phi} \quad\text{in}\ B_R, \quad \partial_r m=0\quad \text{on}\ \partial B_R.
	\end{equation}
	%where $X_{c}$ is the $x$-coordinate of the center of polar coordinates.
	%the integral term in \eqref{krizna_linearized} appears due to linearization of the term $p_{\ast} (|\Omega|)$ in \eqref{actflow_in_terms_of_phi}, 
	%$\rho^{\prime\prime}$ denotes $\partial_\varphi^2 \rho$.
This  problem %\eqref{RadSymLinearized1}--\eqref{ur_e_dlya_m}  
can be rewritten in the operator form
	\begin{equation} \label{operator_form}
	\frac{d}{dt}U=\mathcal{A}_{\rm ss} U, 
	\end{equation} 
	where $U=(m,\rho)$, and $\mathcal{A}_{\rm ss}$ is the following operator 
%($\rho$ enters \eqref{operator1} via $\phi$)
	\begin{equation}
(\mathcal{A}_{\rm ss}U)_m=\Delta m- m_0 \Delta \phi\quad\text{in}\ B_R, \quad (\mathcal{A}_{\rm ss}U)_\rho=\partial_r\phi
\quad\text{on}\ \partial B_R.
%	\mathcal{A}_{\rm ss}:\left[\begin{array}{c}m\\\rho\end{array}\right]
%	\mapsto \left[\begin{array}{c}\Delta m- m_0 \Delta \phi\\\partial_r\phi-\frac{\cos \varphi}{\pi}\int_{-\pi}^{\pi}\partial_r\phi \cos \tilde{\varphi}\,d\tilde{\varphi}%\end{array}\right],
	\label{operator1}
	\end{equation}
	%\begin{equation}
	%\mathcal{A}:
	%%X_c \mapsto \frac{1}{\pi} \int_{-\pi}^{\pi}\partial_r \phi \cos\varphi d\varphi,\ \ 
	%m\mapsto \Delta m-m_0 \Delta \phi,\ \ 
	%\rho   \mapsto \partial_\nu \phi-\frac{\cos\varphi}{\pi} \int_{-\pi}^{\pi}\partial_\nu \phi \cos\varphi d\varphi,
	%\label{operator1}
	%\end{equation}
	Here $\phi$ solves the time independent problem \eqref{ur_e_dlyacrtin}--\eqref{krizna_linearized} for given $m$ and $\rho$.  
%\textcolor{red}
%{ 
	Formulas in \eqref{operator1} define an unbounded operator in $L^2(B_R)\times L^2(\partial B_R)$ whose domain is $H^2(B_R)\cap \{m\, ; \ \partial_r m=0\ \text{on}\ \partial B_R\}\times H^3_{\rm per}(\partial B_R)$.  Using Fourier analysis (see the proof of Theorem \ref{thm_linearizeddisk}) 
one can show that  this operator has a compact resolvent (alternatively, one can establish this fact following the lines of the proof of  Lemma \ref{VspomogLemma}  which does not use radial symmetry).
% \ref{thm_linearizeddisk}) .
	%}
% IDEA Compact operator is almost fiunitedimensional,
%show that coefficients for higher modes in the 
%the RHS becomes small
% Or directly show  using apriori bounds (multiply eqn and integrate by parts ....)
%that bounded set in L^2

%This operator is considered on pairs $U=(m,\rho)$ such that $m\in H^2(B_R)$, then $\phi\in H^4(B_R)$ due to \eqref{ur_e_dlyacrtin}, and its trace $\phi|_{\partial B_R}\in H^{7/2}(\partial B_R)$. Therefore, by \eqref{krizna_linearized}, $\rho\in H^{11/2}(\partial B_R)$.  %and 	$\partial_r m=0$ on $\partial B_R$, $\rho\in H^4(-\pi,\pi)$  and  $\rho$ is an even $2\pi$-periodic function.\textorange{Why is $\rho\in H^4$?}   
%	To obtain the integral term in \eqref{operator1}, we multiply \eqref{RadSymLinearized1} by $\cos\varphi$ and integrate, taking into account that $\partial_t\rho$ is orthogonal to $\cos\varphi$ due to \eqref{centeredDomAin}. Then $\dot X_c=\frac{1}{\pi}\int_{-\pi}^\pi \partial_r\phi\cos\tilde\varphi\,d\tilde\varphi$ is substituted into \eqref{RadSymLinearized1}. %The integral term in \eqref{operator1} results from solving for $\dot X_c$ in terms of $\partial_r\phi$ by multiplying \eqref{RadSymLinearized1} by $\cos\varphi$ and integrating, and noting via \eqref{centeredDomAin} that $\partial_t\rho$ is orthogonal to $\cos\varphi$. %appears  when the orthogonality condition \eqref{centeredDomAin} is applied to \eqref{RadSymLinearized1}; multiply .
	
Next we study the spectrum of  the operator ${\mathcal A}_{\rm ss}$. Due to the radial symmetry of the problem this amounts to the Fourier analysis. 
Moreover we will consider only perturbations possessing the reflection symmetry with respect to the $x$-axis. That is we consider 
 Fourier modes $m=\hat{m}(r)\cos (n\varphi)$ and $\rho=\hat{\rho}\cos(n\varphi)$ for integer $n\geq 0$. Notice that the operator  ${\mathcal A}_{\rm ss}$ 
 always has zero eigenvalue of multplicity at least two with eigenvectors $(m=0,\rho=\cos\varphi)$  and  $(m=2\pi R p_\ast^\prime (\pi R^2)+\gamma/R^2, \rho=1)$. The first of these eigenvectors represents infinitesimal shifts in the $x$-direction, and the second one is  obtained by
 taking derivative  (in $R$) of the family of stationary solutions \eqref{steady_state}. Next we introduce the eigenvalue $E(R)$ describing movability of  stationary solutions. Namely, we will see that at the critical radius when $E(R)$ crosses zero a family of traveling wave solutions emerges.

Consider the minimization problem 
%Introduce the space of functions $K_1=\{m\in H^1(B_R);\, m=\hat m(r) \cos\varphi\}$  and consider the quadratic form 
		\begin{multline}
		\label{formaquadratic}
		E(R)=-\inf \left\{ E_\zeta (m)\Bigl/\Bigr.\int_{B_R} m^2 dxdy; m\in H^1(B_R),\, m=\hat m(r) \cos\varphi\right\},\\ \text{where}\ E_\zeta (m)=\int_{B_R}\left( |\nabla m|^2 -m_0 m^2
+m_0\zeta|\nabla\phi|^2+m_0\zeta^2 \phi^2\right)dxdy
		\end{multline}
and $\phi$ is the unique solution of the equation $\Delta \phi+m=\zeta\phi$ with the Dirichlet boundary condition 
		$\phi=0$ on $\partial \Omega$.  Minimizing the Rayleigh quotient in \eqref{formaquadratic} yields a minimizer $m$ that satisfies 
 $\Delta m+m_0 m-m_0\zeta\phi=E m$ in $B_{R}$ and $\partial_r m=0$ on $\partial B_R$ (c.f. \eqref{operator1}). In the case when $E\not =0$ one obtains the  $\rho$-component of the
eigenvector by setting $\rho=\partial_r\phi/E$. If $E=0$, then the pair $(m,0)$ is a generalized eigenvector in the Jordan chain generated by infinitesimal shifts.

\begin{rem}\label{movability}
	Because of the  radial symmetry, the spectral analysis of $\mathcal A_{ss}$ is performed via Fourier modes. The Fourier mode with $(m,\rho)=(\hat m (r)\cos\varphi,\hat\rho\cos\varphi)$ is the only mode that corresponds to motion (when the geometrical center of mass of $\Omega$ changes). There are infinitely many eigenvectors within this Fourier mode, each with a different $\hat m(r)$. In particular, we have the zero eigenvalue with its eigenvector corresponding to infinitesimal shifts. Then $E(R)$ is the largest of the remaining eigenvalues. That is why $E(R)$ describes movability of the stationary solutions.
% (If $E(R)<0$, a stationary state resists motion, i.e., it returns to stationary after a perturbation. If $E(R)>0$, a stationary state ``wants to move,'' i.e., it diverges from stationary upon perturbation.)
\end{rem}

\begin{lem}
\label{och_polezn_lemma}  Assume that $m_0\leq \zeta$, then $E(R)=0$ if and only if the solution $\Phi_D(r)$ of
\begin{equation}
		\label{BifCond1}
		\frac{1}{r}(r\Phi^\prime_D(r))^\prime-\frac{1}{r^2}\Phi_D(r) +(m_0-\zeta)\Phi_D(r) =m_0 r\quad 0\leq r<R,
		\quad \Phi_D(0)=\Phi_D(R)=0.
		\end{equation}
satisfies  the additional boundary condition  
\begin{equation}
	\label{BifCond1bis}
\Phi_D^\prime(R) =1.
\end{equation}	
Moreover, in this case $E(R)=0$ is a simple eigenvalue of  the variational problem \eqref{formaquadratic} and, up to multiplication by a constant, $m=(\Phi_D(r)-r)\cos\varphi$.
\end{lem}	
	\begin{proof} Since $m=\hat{m}(r)\cos\varphi$, the solution $\phi$ of  $\Delta \phi+m=\zeta\phi$ in $B_R$, $\phi=0$  on $\partial B_R$ has the representation 
$\phi=\hat\phi(r)\cos\varphi$. Now assuming that $E(R)=0$ we have $\Delta (m-m_0\phi)=0$, therefore $m=m_0(\phi -Cr\cos\varphi)$. Clearly $C\not=0$, 
therefore, changing the normalization if necessary, we can assume that  $C=1$.  Then $\phi=\Phi_D(r)$, and since $\partial_r m=0$ on $\partial B_R$ we obtain  
\eqref{BifCond1bis}. Inversely, if the solution $\Phi_D(r)$ of
		\eqref{BifCond1} satisfies \eqref{BifCond1bis}, setting $m=m_0(\Phi_D(r)-r)\cos\varphi$ we have $\phi=\Phi_D(r)\cos\varphi$, and 
\begin{equation*}
E(R) \geq-\frac{1}{\int_{B_R} m^2 dxdy}E_\zeta(m) =0, \quad \text{thus}\ E(R)=0.
\end{equation*}
Lemma \ref{och_polezn_lemma} is proved.
	\end{proof}

\begin{lem} 
\label{motilityLemma}
Assume that $m_0\leq \zeta$, then $E(R)>0$, $E(R)=0$ and $E(R)<0$ if and only if $\Phi_D^\prime(R)>1$, $\Phi_D^\prime(R)=1$ and $\Phi_D^\prime(R)<1$, correspondingly, where $\Phi_D(r)$ is the solution of  \eqref{BifCond1}.	
\end{lem}

\begin{proof} Assume that  $\Phi_D^\prime(R)>1$ and consider the test function $m:=m_0(\Phi_D(r)-r)\cos\varphi$. Observe that
$\Delta(\Phi_D(r)\cos\varphi)+m=\zeta \Phi_D(r)\cos\varphi $, therefore we have, integrating by parts,
\begin{multline*}
E(R)\int_{B_R} m^2 dxdy \geq-E_\zeta(m) =-\int_{\partial B_R}m \partial_r m d s +\int_{B_R}(\Delta m+m_0 m-\zeta m_0\Phi_D(r)\cos\varphi )m dxdy
\\=\pi R^2 m_0^2 (\Phi_D^\prime(R)-1)>0.
\end{multline*}
The case  $\Phi_D^\prime(R)=1$ is considered in Lemma \ref{och_polezn_lemma}. Finally we prove that $E(R)<0$  if  $\Phi_D^\prime(R)<1$. We argue by contradiction. Assume that $E(R)>0$ and notice that allowing the parameter $\zeta$ in \eqref{formaquadratic}   increase we have a continuous function $E(R,\zeta)$ which becomes negative for sufficiently large $\zeta$. To prove the latter clame observe that otherwise there exists a sequence $\zeta_j\to \infty$ and $m_j=\hat m_j(r)\cos\varphi$ such that $\|m_j\|_{L^2(B_R)}=1$ and $E_{\zeta_j}(m_j)\leq 0$. This gives the a priori bound 
		\begin{equation}\label{estimate_with_m_0}
		\int_{B_R}|\nabla m_j|^2 \,dxdy+m_0\zeta_j\int_{B_R}\left(|\nabla \phi_j|^2+ \zeta_j \phi^2_j\right) \, dxdy\leq m_0,
		\end{equation}
		where $\Delta \phi_j+m_j=\zeta_j \phi_j$ in $B_R$, $\phi_j=0$ on $\partial B_R$. 
Let us show  that $\zeta_j\phi_j - m_j\rightharpoonup 0$ weakly in $L^2(B_R)$. 
Indeed, multiply the equation $\Delta \phi_j+m_j=\zeta_j \phi_j$ by a test function $v\in H^1(B_R)$ and integrate over $B_R$ to get
		\begin{equation}\label{vseravnonulyu}
		\int_{B_R}  \nabla \phi_j\cdot  \nabla v \, dxdy + \int_{B_R}(\zeta_j\phi_j-m_j)v\, dxdy =0,
		\end{equation}
		and pass to the limit in this identity as $j\to \infty$. By \eqref{estimate_with_m_0}) we have $\|\nabla \phi_j\|_{L^2(B_R)}<1/\sqrt{\zeta_j}$ therefore the second term in \eqref{vseravnonulyu} tends to zero and thus the weak convergence  $\zeta_j\phi_j - m_j\rightharpoonup 0$ is established. 
It follows from  \eqref{estimate_with_m_0} that there exists $m^\ast \in H^1(B_R)$ such that, up to a subsequence,  $m_j\to m^\ast$ strongly in $L^2(B_R)$, consequently  
$$
\liminf\limits_{j\to \infty} \|\zeta_j \phi_j\|^2_{L^2(B_R)}-1=\liminf\limits_{j\to \infty} (\|\zeta_j \phi_j\|^2_{L^2(B_R)}-\|m_j\|^2_{L^2(B_R)} )\geq 0.
$$
 Then \eqref{estimate_with_m_0} implies  that $\limsup _{j\to \infty}\int_{B_R}|\nabla m_j|^2 \, dxdy=0$, i.e. $m^\ast\equiv \text{const}$. On the other hand 
$m^\ast$ admits the representation  $m^\ast=\hat m^\ast(r)\cos\varphi$. Therefore $m^\ast=0$ which contradicts the normalization $\|m^\ast\|_{L^2(B_R)}=1$.

Thus, $\min E_{\hat\zeta}(m)/\int_{B_R} m^2 dxdy=0$ for some $\tilde\zeta >\zeta$.
		%Thus $F_{\hat\zeta}[m]\geq F_\zeta[m]$, while the continuity of $F_\zeta[m]$ and the fact that $F_\zeta[m]\to+\infty$ as $\zeta\to+\infty$ are obvious. Moreover, by the compactness of the embedding of $H^1(B_R)$  in $L^2(B_R)$ one sees 
		%that $\min F_\zeta [m]/\int_{B_R} m^2 dxdy$ is continuous in $\zeta$ and tends to $+\infty$ as $\zeta\to+\infty$. It follows that $\min F_{\hat\zeta}[m]]/\int_{B_R} m^2 dxdy=0$ for some $\hat \zeta >\zeta$.
		Then by Lemma \ref{och_polezn_lemma}
		the solution of 
		\begin{equation}
		\label{BifCond1aux}
		\frac{1}{r}(r\tilde\Phi^\prime_D(r))^\prime-\frac{1}{r^2} \tilde\Phi_D(r) +(m_0-\tilde\zeta)\tilde\Phi_D(r) =m_0 r\quad 0\leq r<R,
		\quad\tilde\Phi_D(0)=\tilde\Phi_D(R)=0.
		\end{equation}	
		satisfies
		\begin{equation}
		\label{BifCond1bisaux}
		\tilde\Phi_D^\prime(R) =1.
		\end{equation}	
		But  $-\frac{1}{r}(r(\tilde\Phi_D^\prime(r)-\Phi_D^\prime(r)))^\prime+\frac{1}{r^2}( \tilde\Phi_D(r)-\Phi_D(r)) +(\zeta-m_0)
		(\tilde\Phi_D(r)-\Phi_D(r)) =(\zeta-\tilde\zeta)\tilde\Phi_D>0$  for $0< r<R$, and $\tilde\Phi_D(0)-\Phi_D(0)=\tilde\Phi_D(R)-\Phi_D(R)=0$. By the maximum principle $\tilde\Phi_D(r)-\Phi_D(r)> 0$ for $0< r<R$, therefore  
		$\tilde\Phi_D^\prime(R)\leq \Phi_D^\prime(R)$, i.e. $\Phi_D^\prime(R)\geq 1$, contradiction.
Lemma \ref{motilityLemma} is proved.
\end{proof}

The following result provides sufficient conditions for linear stability of radial stationary solutions.

	\begin{thm} 
		\label{thm_linearizeddisk} Assume that  the myosin density $m_0$ is bounded above by the fourth eigenvalue of the operator $-\Delta$ in $B_R$ with the homogeneous Neumann boundary condition on $\partial B_R$, also assume that $\zeta>m_0$ and $p_{\ast}^\prime(\pi R^2)$ satisfies 
		\begin{equation}
		\label{areashrinkingprevention}
		p_{\ast}^\prime(\pi R^2)<
		-\left(\gamma/R +2m_0\right)/(2\pi R^2). 
		\end{equation}
		%Let $\phi_1$ be the solution of \eqref{BifCond1}.
		%\begin{equation}
		%\label{BifCond1}
		%\frac{1}{r}(r\phi^\prime_1(r))^\prime-\frac{1}{r^2} \phi_1(r) +(m_0-\zeta)\phi_1(r) =m_0 r\quad 0\leq r<R,
		%\quad \phi_1(0)=0, \,\phi_1(R)=0.
		%\end{equation}
		Then $\mathcal{A}_{ss}$ has zero eigenvalue with multiplicity two if  $E(R)\not=0$ or three if $E(R)=0$, and all its eigenvalues other than $0$ or
$E(R)$ have negative real parts.
	\end{thm}
	
%	\textcolor{red}{This theorem shows that  the eigenvalue  $E(R)$ that describes mobility of stationary  solutions   determines stability of the nonlinear problem \eqref{actflow_in_terms_of_phi}--\eqref{myosin_bc_I}.  Indeed the two eigenvectors corresponding to the zero eigenvalues  of $\mathcal{A}_{ss}$ are due  to the  shifts and the derivative of stationary  solutions \eqref{steady_state} in $R$ (or equivalently in $M$) and these eigenvectors generate the slow  manifold for the dynamical system  \eqref{operator_form}.
	%slow manifold: both real and  imaginary parts are zero, center only real is zero
	 % This manifold  is formed by the family of stationary solutions \eqref{steady_state}  and their shifts.
	%While these zero eigenvalues  result in center (slow) manifold for the dynamical system  \eqref{operator_form}, and 
	%Due to this explicit description   of the manifold (the slow manifold in general  makes linear analysis inconclusive) 
	%linear stability analysis inconclusive. Howevet 
	%we can use the 
	% freedom in the choice of origin of the coordinate  system and conservation of the total myosin mass to control  these eigenvectors. Thus  the stability of  stationary solutions of the original nonlinear problem  can be described  in terms of  $E(R)$ only.  This can be proved by generalizing techniques from  \cite{basali} developed for shifts eigenvalues. }

 This theorem shows that the linear stability of  stationary solutions can be described  in terms of the eigenvalue $E(R)$ only.  In order to  determine stability/instability of the original nonlinear problem \eqref{actflow_in_terms_of_phi}--\eqref{myosin_bc_I}, an additional consideration is necessary because of the zero eigenvalue.
	%The conclusion of this theorem can n be directly applied to determine stability/instability of the original nonlinear problem\eqref{actflow_in_terms_of_phi}--\eqref{myosin_bc_I} because of zero eigenvalues  .
	 This eigenvalue has multiplicity two; the  corresponding two eigenvectors % corresponding to the zero eigenvalues  of $\mathcal{A}_{ss}$ 
are due to the  shifts and the derivative of stationary  solutions \eqref{steady_state} in $R$ (or equivalently in $M$). These eigenvectors generate the slow  manifold %for the dynamical system  \eqref{operator_form} 
formed by  stationary solutions \eqref{steady_state}  and their shifts. 
While in general transition from linear to nonlinear stability in presence of zero eigenvalue may be  challenging, here  the conservation of  total myosin mass property together with invariance with respect to shifts  can be used to control the   eigenvectors for the zero eigenvalue.  This can be done, for instance, by generalizing the techniques from  \cite{BazFri2003} developed for the tumor growth problem
where authors deal with 2D slow manifold of shifts of a stationary solution.

	\begin{rem}\label{R_M} Due to \eqref{steady_state} the total myosin mass of a radial stationary solution is $M=\pi R^2 m_0=\pi R^2 p_\ast(\pi R^2)-\pi R\gamma$. Then the condition \eqref{areashrinkingprevention} is equivalent to 
		\begin{equation}
		M'(R)<0,
		\end{equation}
		so $M(R)$ is decreasing. Notice that if $M^\prime(R)\not=0$ then (locally) stationary solutions can be reparametrized by $M$.
	\end{rem}

	\begin{proof} Let $\lambda$ be an eigenvalue corresponding to an eigenvector $m=\hat m_n(r)\cos n\varphi$, $\rho=\hat\rho_n\cos n\varphi$ with $n\geq 2$. %In this case \eqref{ur_e_dlyacrtin} takes form $\Delta \phi =\zeta \phi-m$, 
Multiply the equation
		$ %\begin{equation*}
		\lambda m=\Delta m +m_0 m-m_0 \zeta \phi$  %\end{equation*}   
  by the complex conjugate $\overline {m}$  of $m$ and integrate over $B_R$,
		\begin{equation}
		\label{Vychisleniya_0}
		\lambda \int_{B_R} |m|^2dxdy=-\int_{B_R}|\nabla m|^2dxdy +m_0 \int_{B_R} |m|^2dxdy
		-m_0\zeta \int_{B_R} \phi \overline {m}dxdy.
		\end{equation}
		Now multiply the equation $\overline{m}=\zeta\overline{\phi}-\Delta \overline{\phi}$   by $m_0\zeta\phi$ and integrate over $B_R$, then we obtain the following representation for the last term in \eqref{Vychisleniya_0}:
		\begin{equation}\label{eq:phi_eq_multiplied_and_integrated}
		m_0\zeta \int_{B_R} \phi \overline {m}dxdy=m_0\zeta^2 \int_{B_R} |\phi|^2 dxdy+m_0\zeta\int_{B_R}|\nabla \phi|^2dxdy-m_0\zeta\int_{\partial B_R} \phi {\partial_r \overline{\phi}} d s.
		%=
		%\zeta \int |\phi|^2 +\int|\nabla \phi_n|^2+\lambda \frac{R^2\zeta}{\gamma(n^2-1)}\int |\phi_n|^2ds.
		\end{equation}
		Since $\partial_r \overline{\phi}=\overline{\lambda}\overline{\rho}$ and (by virtue of \eqref{krizna_linearized})
		$\overline \rho(\varphi)=\frac{R^2\zeta}{\gamma (1-n^2)}\overline\phi(R,\varphi)$ equality \eqref{Vychisleniya_0} rewrites as
		\begin{equation}\label{eq:SS_high_mode_eval_sign}
		\begin{split}
		\lambda \int_{B_R} |m|^2dxdy +\overline{\lambda} \frac{m_0 R^2\zeta^2}{\gamma(n^2-1)}\int_{\partial B_R} |\phi|^2d\sigma=&-\int_{B_R}|\nabla m|^2dxdy+m_0 \int_{B_R} |m|^2dxdy\\ 
		&- m_0\zeta \int_{B_R} |\nabla \phi|^2dxdy-m_0\zeta^2 \int_{B_R} |\phi|^2dxdy.
		\end{split}
		\end{equation} 
Notice that, thanks to the assumption that $m_0$ is bounded by the fourth eigenvalue of $-\Delta$ in $B_R$ with homogeneous Neumann boundary condition, we have
% By Proposition \ref{lem:ineq_for_m}, we have
		\begin{equation}
		\int_{B_R}|\nabla m|^2dxdy-m_0 \int_{B_R} |m|^2dxdy\geq 0.
		\label{zvezda}
		\end{equation}
%		since (i) for $n\geq 2$ the function $\hat m_n(r)\cos n\varphi$ is orthogonal to the first and second eigenfunctions
%		of the operator $-\Delta$  in $B_R$ with the Neumann condition on $\partial B_R$ and (ii) $m_0$ is bounded by the third eigenvalue.
		Therefore the right hand side of \eqref{eq:SS_high_mode_eval_sign} is negative, so the real part of $\lambda$ is also negative.
		
Next observe that all the eigenvalues whose corresponding eigenvectors have the form $m=\hat m(r)\cos \varphi$, $\rho=\hat\rho\cos \varphi$ 	
are described by the Courant minimax principle, 
\begin{equation*}
		%\label{formaquadratic}
		\lambda_j=-\sup_{{\rm codim} S=j-1}\inf_{m\in S}\frac{E_\zeta (m)}{\int_{B_R} m^2 dxdy},\quad\text{where $S$ is a subspace of} \ \left\{ m\in H^1(B_R),\, m=\hat m(r) \cos\varphi\right\}.
		\end{equation*}
Since $E_\zeta (m)<\int_{B_R} |\nabla m|^2 dxdy-m_0\int_{B_R} m^2 dxdy$ for $m\not=0$, the $j$-th eigenvalue $\lambda_j$  is bounded by 
the $j$-th eigenvalue
of the restriction of the operator $\Delta+m_0$ in $B_R$ with the homogeneous Neumann boundary condition to the space of functions of the form  
$m=\hat m(r)\cos \varphi$. By the assumption the second eigenvalue of the latter operator is non positive,  therefore $\lambda_2< 0$ (while $\lambda_1=E(R)$).

Consider finally an eigenvalue $\lambda$ corresponding to a radially symmetric eigenvector. 
%
%
%		Consider now the case $n=0$ which corresponds to radially symmetric eigenfunctions. Taking the derivative of steady states 
%		with respect to the parameter $R$ we obtain an eigenvector corresponding to zero eigenvalue. Let us show that other radially symmetric eigenvectors correspond to eigenvalues with negative real parts. %It is convenient to change the unknown $\tilde\phi:= \phi+2\pi R p_{\ast}^\prime(\pi R^2)\rho/\zeta$, then \eqref{eig3} rewrites as $-\Delta \tilde \phi+\zeta \tilde \phi =m$, \textorange{This is already the form of the desired equation}
		We have on $\partial B_R$
		\begin{equation}
		\phi=( \gamma/ R^2+2\pi Rp_{\ast}^\prime(\pi R^2))\rho /\zeta, \quad 
		\lambda\rho=
		\partial_r \phi= \frac{\lambda\zeta}{\gamma/R^2+2\pi R p_{\ast}^\prime(\pi R^2)}\phi.
		\label{zvezdavkvadrate}
		\end{equation}
		Multiply the equations $\lambda m = \Delta m+m_0 m -m_0\zeta \phi$ and 
		$-\Delta \overline{\phi}+\zeta \overline{\phi} =\overline{m}$ by $\overline{m}-\langle\overline{m}\rangle$
		and $\phi-\langle \phi\rangle$, where $\langle \overline{m}\rangle$, $\langle \phi \rangle$ denote mean values 
		of $\overline{m}$, $\phi$ (over $B_R$), and integrate over $B_R$. Using integration by parts we obtain equations analogous to \eqref{Vychisleniya_0}-\eqref{eq:phi_eq_multiplied_and_integrated} with $m-\langle m\rangle$ and $\phi-\langle\phi\rangle$ instead of $m$ and $\phi$. Notice that
		\begin{equation}
		\int_{B_R}\phi(\bar m-\langle\bar m\rangle)\,dx\,dy=\int_{B_R}(\phi-\langle\phi\rangle)\bar m\,dx\,dy,
		\end{equation} this leads to
		\begin{equation}
		\label{DlinnayaOtsenkaNoPonyatnaya}
		\begin{aligned}
		\lambda \int_{B_R} |m-\langle{m}\rangle|^2dxdy =&-\int_{B_R}|\nabla m|^2dxdy+m_0 \int_{B_R} |m-\langle{m}\rangle|^2dxdy
		\\&- m_0\zeta \int_{B_R} \left( |\nabla \phi|^2+\zeta |\phi-\langle \phi\rangle|^2\right) dxdy
		+m_0\zeta\int_{\partial B_R} (\phi-\langle \phi\rangle) \partial_r \overline{\phi} d s.
		\end{aligned} 
		\end{equation}
		Assume that $\lambda\not=0$. Then 
		we can evaluate $\langle \phi\rangle$ in terms of $\rho$, integrating the equations 
		$\lambda m = \Delta m+m_0 m -m_0\zeta \phi$ and 
		$-\Delta \overline{\phi}+\zeta \overline{\phi} =m$ over $B_R$ and eliminating $\langle m \rangle$:
		\begin{equation}
		\label{RAVENSTVO_o}
		\langle \phi\rangle= \frac{1}{\pi R^2\zeta}(1-m_0/\lambda)
		\int_{\partial B_R}\partial_r\phi d s =  \frac{2}{ R\zeta}(\lambda-m_0)\rho.
		\end{equation}
		Now we
		use \eqref{zvezdavkvadrate} and \eqref{RAVENSTVO_o} to rewrite the last term in \eqref{DlinnayaOtsenkaNoPonyatnaya}
		as 
	\begin{equation}
		\label{RAV_ooo}
		m_0\zeta\int_{\partial B_R} (\phi-\langle \phi\rangle) \partial_r \overline{\phi} d s=
		-4\pi m_0 |\lambda|^2 |\rho|^2+2\pi m_0\overline{\lambda}|\rho|^2\left(\gamma /R+2m_0+2\pi R^2 p_{\ast}^\prime (\pi R^2)\right).
	\end{equation} 
	Substitute \eqref{RAV_ooo}  into  \eqref{DlinnayaOtsenkaNoPonyatnaya}, as the result we get 
\begin{equation}
		\label{DlinnayaOkonchatel'nayaOtsenka}
		\begin{aligned}
		\lambda \int_{B_R} &|m-\langle{m}\rangle|^2dxdy+4\pi m_0 |\lambda|^2 |\rho|^2-2\pi m_0\overline{\lambda}|\rho|^2\left(\gamma /R+2m_0+2\pi R^2 p_{\ast}^\prime (\pi R^2)\right) =
\\
&
-\int_{B_R}|\nabla m|^2dxdy+m_0 \int_{B_R} |m-\langle{m}\rangle|^2dxdy
		- m_0\zeta \int_{B_R} \left( |\nabla \phi|^2+\zeta |\phi-\langle \phi\rangle|^2\right) dxdy.
		\end{aligned} 
		\end{equation}
Thanks to the radial symmetry of $m$ the function $m-\langle{m}\rangle$ is orthogonal (with respect to the standard inner product in  $L^2(B_R)$) to three first eigenfunctions of the operator $\Delta$ with the homogeneous Neumann boundary condition. Therefore 
$\int_{B_R}|\nabla m|^2dxdy-m_0 \int_{B_R} |m-\langle{m}\rangle|^2dxdy\geq 0$. Thus \eqref{DlinnayaOkonchatel'nayaOtsenka} implies that 
the real part of $\lambda$ is negative. To complete the proof it remains only to consider the case $\lambda=0$. There is no other eigenvector than  $(m=2\pi R p_\ast^\prime (\pi R^2)+\gamma/R^2, \rho=1)$ in this case, as follows from \eqref{DlinnayaOtsenkaNoPonyatnaya}, and there is no other generalized eigenvectors,
 otherwise 
$0=\int_{B_R}(2\pi R p_\ast^\prime (\pi R^2)+\gamma/R^2)dxdy+m_0\int_{\partial B_R} \partial_\nu \phi ds= \pi R(2\pi R^2p_\ast^\prime (\pi R^2)+\gamma/R+2m_0)$.
Theorem \ref{thm_linearizeddisk} is proved. 
\end{proof}
In the proof of  Theorem \ref{thm_linearizeddisk} we have used the following simple 
\begin{prop}
	The eigenfunctions corresponding to the second and  the third (if counted with multiplicity) eigenvalues of $-\Delta$ in $B_R$  with the homogeneous Neumann condition on $\partial B_R$  have the form
	\begin{equation}\label{eq:second_eigenvalue_of_neumann_laplacian}
	v_2(r,\varphi)=\hat v_2(r)\cos(\varphi+\varphi_0)
	\end{equation}
\end{prop}
\begin{proof} It suffices to show  that $v_2(r,\varphi)$ is not a radially symmetric Fourier mode. Let $\lambda_2$ denote the corresponding eigenvalue. Assume by contradiction that $v_2(r,\varphi)=\hat v_2(r)$, then by straightforward  differentiation of the eigenvalue equation one checks that $\hat v_2^\prime(r)\cos\varphi$ is an eigenfunction of $-\Delta$ in $B_R$  with the homogeneous Dirichlet condition on $\partial B_R$ corresponding to the eigenvalue $\lambda_2$. Since each eigenvalue 
of  $-\Delta$ in $B_R$ with the homogeneous Dirichlet condition on $\partial B_R$ is strictly larger than that  of $-\Delta$ in $B_R$  with the homogeneous Neumann condition on $\partial B_R$, $\lambda_2$ must be the first eigenvalue of the former operator.  However the first eigenfunction is sign preserving, a contradiction.  
\end{proof}

	%Let $\Omega$ be given by $r<R+\rho(t,\varphi)$. In polar coordinates the normal velocity is given by 
	%\begin{equation}
	%	\label{normalvelocityinpolar}
	%	V_\nu=\frac{\partial_t \rho (R+\rho)}{\sqrt{(\rho^\prime)^2+(R+\rho)^2}}
	%\end{equation}
	%(note that the  normal is $\nu_x=(\rho^\prime \sin\varphi+(R+\rho)\cos\varphi)/\sqrt{(\rho^\prime)^2+(R+\rho)^2}$;
	%$\nu_y=(-\rho^\prime \cos\varphi+(R+\rho)\sin\varphi)/\sqrt{(\rho^\prime)^2+(R+\rho)^2}$).

	\section{Bifurcation of traveling waves from the family  of stationary solutions}
	\label{section_bifurcation}
	
In this Section we show that at the critical radius $R=R_0$ such that  $E(R_0)=0$ 
%zero eigenvalue corresponding to eigenvector described in Lemma \ref{och_polezn_lemma}  leads to a 
radially symmetric stationary solutions  \eqref{steady_state} bifurcate to a family of traveling wave solutions.
It is interesting to observe that in a neighborhood of $R_0$ the geometric multiplicity of the zero eigenvalue 
of the linearized operator around stationary solutions is two as well as at $R=R_0$, and the bifurcation takes place via the additional generalized eigenvector 
appearing at $R_0$.% from the family of  parametrized by $R$. 
%	This bifurcation is determined by three parameters: the size of the cell $R$,  and adhesion strength $\zeta$ which are independent parameters and  the myosin density $m_0$.  Due to zero force balance in the steady state, surface tension (determined by  curvature $R^{-1}$),  myosin contraction (determined by myosin density $m_0$),  and the pressure $p_{\ast} (\pi R^2)$  balance each other, which provides the dependence between $m_0$ and $R$ given by the second equation in \eqref{steady_state}.  
%	%add myosin that contracts cell but hydristatuic pressure oppposes
%	It is convenient to choose $R$ as the bifurcation parameter in the bifurcation conditions \eqref{BifCond1}-\eqref{BifCond1bis}. }
	
	Consider the  ansatz  of  a traveling wave solution moving with velocity $V>0$ in $x$-direction	
\begin{equation}
	\label{tw_ansatz}
	m=m(x-Vt,y),\ \phi=\phi(x-Vt,y),\ \Omega(t)=\Omega+(Vt,0)
\end{equation}
and substitute it to \eqref{actflow_in_terms_of_phi}--\eqref{myosin_bc_I}  to derive stationary free boundary problem for the unknowns $\phi$ and  $\Omega$
	%\begin{equation}
	%\label{myosin_equations}
	%\frac{1}{\theta} \Delta m^\theta -\div (\nabla \phi  m)+V\partial_x m=0, \quad \text{in} \ \Omega,
	%\end{equation}
	%also $V_\nu= V\nu_x$. Then solutions are given by 
	%\begin{equation}
	%\label{porousmyosin}
	%m=\left((\theta-1)(\phi-Vx)+\Lambda\right)^{1/(\theta-1)}
	%\end{equation} 
	%for $1<\theta\leq 2$, for $\theta=1$ we have $m=\Lambda e^{\phi-Vx}$. 
	%
	%(we consider $\theta=1$)
	%In this case  for traveling wave solutions we have the problem
	\begin{equation}
	\label{tw_Liouvtypeeq}
	\Delta \phi +\frac{M }{\int_\Omega e^{\phi-Vx}dxdy}e^{\phi-Vx} =\zeta \phi %\textcolor{green}{+p_{\ast}(|\Omega|)}
	\quad \text{in}\ \Omega, \quad \partial_\nu(\phi-Vx)=0\quad \text{on}\ \partial\Omega,
	\end{equation}
	\begin{equation}
	\label{addit_cond_tw}
	\quad \zeta\phi=p_{\ast}(|\Omega|)-\gamma \kappa\quad \text{on}\ \partial\Omega. 
	\end{equation}
	Indeed, \eqref{myosin_equations_I} yields $-V\partial_x m= \Delta m -\div (m \nabla \phi )$ in $\Omega$  while 
	$\partial_\nu \phi=V \nu_x$ on $\partial \Omega$, then, taking into account the boundary condition $\partial_\nu m=0$, we see
	that 
	\begin{equation}
	\label{myosin_density}
	m=\Lambda e^{\phi-Vx},\quad \text{where}\  \Lambda:= \frac {M}{\int_\Omega e^{\phi-Vx}dxdy}.
	\end{equation}
Here unknown positive constant $M$  represents the total mass of myosin, $M=\int_\Omega m\, dx dy$. 

 For radial stationary solutions 
we have the following %one to one 
correspondence between the parameters $M$ and the radius $R$ of the domain
%, assuming that \eqref{areashrinkingprevention} holds, 
\begin{equation}
\label{parameters}
M(R)=\pi R^2    p_{\ast}(\pi R^2)-\pi\gamma R.
\end{equation}
It is convenient  to keep parameter $R$ in the bifurcation analysis presented below, although the domain $\Omega$ is no longer a disk. Then dependence on
$R$ will appear implicitly in the parametrization of the boundary $\partial\Omega$ (as the radius of the reference disk) and explicitly in $M(R)$ given by \eqref{parameters}. We use also the notation $\tilde\Lambda$ for the density of radial stationary solutions,
$$
\tilde\Lambda(R):=M(R)/(\pi R^2)= p_{\ast}(\pi R^2)-\gamma /R,
$$
reserving $m_0$ for the density at $R=R_0$.
%\textcolor{red}{It is just new notation for  $m_0$ !!!!!!Change, or $m_0$ at the bifurcation point, $\tilde \Lambda(R)$ for other  ????}

We rely on Theorem 1.7 from \cite{CraRab1971} to get the following result on bifurcation of traveling wave solutions. 

	\begin{thm}%\textcolor{red}{(bifurcation of traveling waves)}
		\label{biftwtheorem}
		Let   $R_0$ be  the critical radius such that the solution of \eqref{BifCond1} satisfies 
		\eqref{BifCond1bis} with $R=R_0$ and 
		$m_0=\tilde\Lambda(R_0)= p_{\ast}(\pi R_0^2)-\gamma/ R_0$. Assume also that $m_0<\zeta$, $m_0$ is bounded by the fourth eigenvalue of the $-\Delta$ in $B_R$ with the homogeneous Neumann boundary condition,  
		$p_{\ast}^\prime(\pi R^2_0)\not= -(2m_0-\gamma/R_0)/(2\pi R^2_0)$ and 
		\begin{equation}
		\label{transversality_expl}
		F^\prime(R_0)\not =0,
%\frac{d}{dR}\Biggl(\frac{\zeta I_1\left(R\sqrt{\zeta-\tilde\Lambda(R)}\right)}{(\zeta-\tilde\Lambda(R))^{3/2}I_1^{\prime}\left(R\sqrt{\zeta-\tilde\Lambda(R)}\right)}-\frac{R\tilde\Lambda(R)}{\zeta-\tilde\Lambda(R)}\Biggr)\Biggl.\Biggr|_{R=R_0}\not=0,
		\end{equation}
where
\begin{equation}
		\label{FunktsiyaF}
		F(R):=\frac{\zeta I_1\left(R\sqrt{\zeta-\tilde\Lambda(R)}\right)}{(\zeta-\tilde\Lambda(R))^{3/2}
I_1^{\prime}\left(R\sqrt{\zeta-\tilde\Lambda(R)}\right)}
		-\frac{R\tilde\Lambda(R)}{\zeta-\tilde\Lambda(R)},
		\end{equation}
 $I_1$ is the 1st modified Bessel function of the first kind. Then stationary solutions \eqref{steady_state} at $R=R_0$ bifurcate to a family of  traveling wave solutions, i.e. solutions of  \eqref{tw_Liouvtypeeq}--\eqref{addit_cond_tw}  
		parametrized by the velocity $V$. Moreover  for small $V$, $|V|\leq \overline V$ (for some $\overline V>0$), these solutions (both the function $\phi$ and the domain $\Omega$)  are smooth and  depend smoothly on the parameter $V$.  When $V=0$ the solution of  \eqref{tw_Liouvtypeeq}--\eqref{addit_cond_tw} is stationary and  radial,
		$\Omega=B_{R_0}$, $m=\zeta \phi=m_0= p_{\ast}(\pi R_0^2)-\gamma/ R_0$.
		% (the bifurcation point where the family of radial steady states intersects with the family of traveling waves).
	\end{thm}

	\begin{proof} 
		As above we consider $\Omega$ in polar coordinates, $\Omega =\{0 \leq r< R +\rho(\varphi)\}$. Since $\zeta > \tilde\Lambda(R_0)$, for sufficienly small 
$\rho$, $V$ and $R$ sufficiently close to $R_0$ there is a unique solution  
		$\Phi=\Phi(x,y; V,R, \rho)$ of \eqref{tw_Liouvtypeeq}.  It depends on three parameters:
		the  scalar parameter $V$ (the prescribed velocity), the radius $R$ via the parametrization of the domain and 
		$\Lambda=\Lambda(R)$, and the functional parameter
		$\rho$ that describes the shape of the domain $\Omega$ or, more precisely, its deviation from the disk $B_R$. 
                  As above we assume the symmetry of the domain with respect to the $x$-axis and therefore its  shape is described by an even function $\rho$.
		
		The condition \eqref{addit_cond_tw} on the unknown boundary, described by $\rho(\varphi)$,   rewrites as 
		\begin{equation}
		\label{TW_equationontheboundary}
		p_{\ast}(|\Omega|)-\gamma\frac{(R+\rho)^2 + 2(\rho^\prime)^2-\rho^{\prime\prime}(R+\rho)}{((R+\rho)^2+(\rho^\prime)^2)^{3/2}}=\zeta\Phi ((R+\rho (\varphi))\cos\varphi, (R+\rho (\varphi))\sin\varphi, V, R, \rho).
		\end{equation}
%		As before,
 To get rid of infinitesimal shifts we will require that
\begin{equation}
	\label{centeredDomAinTH}
	\int_{-\pi}^{\pi}\rho (\varphi)  \cos\varphi d\varphi =0.
	\end{equation}
%\eqref{centeredDomAin}. 
Then introducing the function $\mathcal{F}$ which maps from $%\mathcal{X}
X=C^{2}_{\rm per,even}(-\pi,\pi)\times \mathbb{R}\times \mathbb{R}$ to 
	$
%\mathcal{Y}
Y=C^{0}_{\rm per,even}(-\pi,\pi)\times \mathbb{R}$: 
		%$0<\alpha<1$: 
		%(if $\rho$ is only $C^2$ , then ): % and depending on the parameter $R$,
	%	\textorange{Why the Holder spaces?}
		\begin{equation}
		\label{operator_equation}
		\mathcal{F}((\rho, V), R):=\left(\gamma\frac{(R+\rho)^2 + 2(\rho^\prime)^2-\rho^{\prime\prime}(R+\rho)}
		{\zeta((R+\rho)^2+(\rho^\prime)^2)^{3/2}}+\Phi-
			\frac{p_{\ast}(|\Omega|)}{\zeta}, \int_{-\pi}^\pi \rho\cos\varphi d\varphi\right),
		\end{equation}
		we rewrite problem \eqref{tw_Liouvtypeeq}--\eqref{addit_cond_tw} in the form
		\begin{equation}
		\label{FuncAnTWequation}
		\mathcal{F}((\rho,V),R)=0. 
		\end{equation}
		Next we apply the Crandall-Rabinowitz bifurcation theorem \cite{CraRab1971} (Theorem 1.7),
		%@article{CRANDALL1971321,
		%title = "Bifurcation from simple eigenvalues",
		%journal = "Journal of Functional Analysis",
		%volume = "8",
		%number = "2",
		%pages = "321 - 340",
		%year = "1971",
		%issn = "0022-1236",
		%doi = "https://doi.org/10.1016/0022-1236(71)90015-2",
		%url = "http://www.sciencedirect.com/science/article/pii/0022123671900152",
		%author = "Michael G Crandall and Paul H Rabinowitz"}
		which guarantees bifurcation of new smooth branch of solutions provided that
		% 
		%Consider Banach space $\mathcal{X}=C^{2,\alpha}(\mathbb S^1)\times \mathbb R$, $\mathcal{Y}=C^{0,\alpha}(\mathbb S^1)\times \mathbb R$ of pairs $X=(\rho,V)$ and function $\mathcal{B}:\mathcal{X}\times \mathbb{R}\to \mathcal{Y}$ defined in \eqref{operator_equation}, which guarantees  According to Crandall-Rabinowitz theorem, we need to verify the following four conditions:
		\begin{itemize}
			\item[(i)] $\mathcal{F}(0, R)=0$ for all $R$ in a neighborhood of $R_0$;
			\item[(ii)] there exist continuous  $\partial_{(\rho,V)}\mathcal{F}$, $\partial_R\mathcal{F}$, and $\partial^2_{(\rho,V),R}\mathcal{F}$ in a neighborhood of $(\rho,V)=0$, $R=R_0$;  
			\item[(iii)] the null space $\mathrm{Null}(\partial_{(\rho,V)}\mathcal{F})$ at $(\rho,V)=0$, $R=R_0$ has dimension one and 
$\mathrm{Range}(\partial_{(\rho,V)}\mathcal{F})$ at $(\rho,V)=0$, $R=R_0$ has co-dimension one.
			\item[(iv)]  $\partial^2_{(\rho,V),R}\mathcal{F}(\rho,V)\notin  \mathrm{Range}(\partial_{(\rho,V)}\mathcal{F})$ at $(\rho,V)=0$, $R=R_0$ 
			for all 
			$(\rho,V)\in \mathrm{Null}(\partial_{(\rho,V)}\mathcal{F})$.
		\end{itemize}
		
It is easy to see that condition (i) is satisfied, and  condition (ii) can be verified as in \cite{BerFuhRyb2018}.  To verify (iii)  we begin with calculating $\mathcal{L}:=\partial_{(\rho,V)}\mathcal{F}$ at $(\rho,V)=0$.
		% 	Then for any complement $\mathcal{Z}$ of  $\mathrm{Null}(D_X\mathcal{B})$, there exists a neighborhood $U\subset \mathcal{X}\times \mathbb R$ of $(0,R_{\mathrm{crit}})$ and continuous functions $\Phi:(R_{\mathrm{crit}}-a,R_{\mathrm{crit}}+a)\to \mathbb R$, $\Psi:(R_{\mathrm{crit}}-a,R_{\mathrm{crit}}+a)\to \mathcal Z$ such that $\Phi(0)=0$, $\Psi(0)=0$ and  
		% 	\begin{equation}
		% 	\mathcal{B}^{-1}(0)\cap U =\left\{(s X_0+ s \Psi(s),\Phi(s)):|s-R_{\mathrm{crit}}|<a\right\}\cup\left\{(0,R):(0,R)\in U\right\}.
		% 	\end{equation}
		% 	\end{theorem}
		Linearizing  \eqref{operator_equation} around $(\rho,V)=0$  we get
		\begin{multline}\label{linearization}
		%\begin{split}
		\mathcal{L}: (\rho, V)\mapsto %&
		\left(-\frac{\gamma}{R^2\zeta} (\rho^{\prime\prime}+\rho)+V\partial_V \Phi(R\cos\varphi,R\sin\varphi; 0,R,0) \right.\\
		%&
\left.+\langle \partial_\rho \Phi,  \rho \rangle\bigl|_{(\rho,V)=0}\bigr.-
			R\frac{p^\prime_{\ast}(\pi R^2)}{\zeta}\int_{-\pi}^{\pi} \rho d\varphi,
		\int_{-\pi}^{\pi} \rho(\varphi) \cos \varphi d\varphi
		\right).
		%\end{split}
		\end{multline}
		Here $\langle \partial_\rho \Phi,  \rho \rangle\bigl|_{ (\rho,V)=0}\bigr.$ denotes the Gateaux derivative of  $\Phi$ at $(\rho,V)=0$.  
	We have 
$$
\langle \partial_\rho \Phi,  \rho \rangle\bigl|_{ (\rho,V)=0}\bigr.=-\frac{ m_0  }{\zeta \pi R }\int_{-\pi}^{\pi}\rho\,d\varphi
$$
		%-\frac{R}{\zeta}p_{\ast}^\prime(\pi R^2)\int_{-\pi}^\pi\rho d\varphi$
		and $\partial_V \Phi(R\cos\varphi,R\sin\varphi; 0,R,0)=\tilde\Phi_V^0(R,R)\cos\varphi$, where $\tilde \Phi_V^0(r,R)$ solves
		\begin{equation}
\begin{aligned}
		\label{bicond_po_drugomu}
		&\frac{1}{r}\left(r(\tilde \Phi_V^0)^\prime(r,R)\right)^\prime-\frac{1}{r^2} \tilde \Phi_V^0(r,R) +(\tilde\Lambda(R)-\zeta)\tilde \Phi_V^0(r,R) =\tilde\Lambda(R) r\quad 0\leq r<R,\\
		& \tilde \Phi_V^0(0,R)=0, \,(\tilde \Phi_V^0)^\prime (r,R)\bigl|_{r=R}\bigr.=1.
\end{aligned}
		\end{equation}
The operator $\mathcal{L}$ has a bounded inverse when $\tilde \Phi_V^0(R,R)\not =0$ as can be verified by the Fourier analysis, 
while in the case $\tilde \Phi_V^0(R,R)=0$ (for $R=R_0$, when the operator $\mathcal{A}_{ss}$ has a generalized eigenvector corresponding to the zero eigenvalue, 
%with non-constant density $m$,
see Lemma \ref{och_polezn_lemma}) the null space of the operator $\mathcal{L}$  is one-dimensional (it is ${\rm span}\{(\rho=0,V=1)\}$) and its range consists of all the pairs $(f,C)$ such that 
		$\int_{-\pi}^{\pi}f(\varphi)\cos\varphi d\varphi=0$. Thus, condition (iii) is satisfied.

		It remains to verify  the transversality condition (iv). % In fact, we show that it reduces to the ``non-tangential crossing" condition \eqref{VseVterminahF}. 
We check if  
		$\partial_R \mathcal{L} (0,1)\bigl|_{R=R_0}\bigr.$ does not belong to the range of the opeartor $\mathcal{L}$, where
		\begin{equation}\label{eq:derivative_of_L}
		\partial_R \mathcal{L}\bigl|_{R=R_0}: (\rho, V)\bigr.\mapsto 
		\left(\frac{2\gamma}{R^3_0\zeta} (\rho^{\prime\prime}+\rho)+V\frac{d}{d R}\tilde \Phi_V^0\Bigl.\Bigr|_{R=R_0}\cos\varphi
		+C_\ast (R_0)\int_{-\pi}^{\pi}\rho d\varphi
		,0\right)
		\end{equation}
		with $C_\ast (R_0)=-\frac{1}{\zeta}\left(p_{\ast}^\prime(\pi R_0^2)+2\pi R_0^2 p_{\ast}^{\prime\prime}(\pi R_0^2)+3M/(\pi R^2_0)^2\right)$. 
		Since the range of $\mathcal L$ (described above) is all $(f,C)$ such that $f$ is orthogonal to $\cos\varphi$, we must have a nonzero coefficient in front  of $\cos\varphi$ in \eqref{eq:derivative_of_L} to satisfy  condition (iv). Thus this (transversality) condition can be equivalently restated as
\begin{equation}
		\label{transversalitycondition}
		\Bigl.\frac{d}{d R}
		\tilde \Phi_V^0(R,R)\Bigr|_{R=R_0}\not =0.
\end{equation} 
	In order to check \eqref{transversalitycondition}  we change variable in \eqref{bicond_po_drugomu} by introducing $\psi(r,R):=\tilde \Phi_V^0(Rr,R)$, 
		this leads to the problem in the unit disk:
		\begin{equation*}
\begin{aligned}
		&\frac{1}{r}(r\psi^\prime(r,R))^\prime-\frac{1}{r^2} \psi(r,R) +R^2(\tilde\Lambda(R)-\zeta)\psi(r,R) =R^3\tilde\Lambda(R) r\quad 0\leq r<1,\\
		&\psi(0,R)=0, \,\psi^\prime(r,R)\bigl|_{r=1}\bigr.=R.
\end{aligned}
		\end{equation*}
		The solution of this problem is given by 
		$$
		\psi(r,R)=-\frac{R\tilde\Lambda(R)}{\zeta-\tilde\Lambda(R)}r+\frac{\zeta I_1\left(R\sqrt{\zeta-\tilde\Lambda(R)}r\right)}
		{(\zeta-\tilde\Lambda(R))^{3/2}I_1^{\prime}\left(R\sqrt{\zeta-\tilde\Lambda(R)}\right)},
		$$
		so that condition \eqref{transversalitycondition} writes as \eqref{transversality_expl}.
	\end{proof}
	
 \begin{rem}% In terms roduce the following  function 
%		\begin{equation}
%		\label{FunktsiyaF}
%		F(R):=\frac{\zeta I_1\left(R\sqrt{\zeta-\tilde\Lambda(R)}\right)}{(\zeta-\tilde\Lambda(R))^{3/2}
%I_1^{\prime}\left(R\sqrt{\zeta-\tilde\Lambda(R)}\right)}
%		-\frac{R\tilde\Lambda(R)}{\zeta-\tilde\Lambda(R)}.
%		\end{equation}
		The condition \eqref{BifCond1bis} that selects the critical radius $R$ in \eqref{BifCond1}
		% solvability condition for  \eqref{BifCond1}--\eqref{BifCond1bis} 
		(which is also the necessary bifurcation condition, c.f. Theorem \ref{thm_linearizeddisk}, item (ii)) 
		and the transversality condition \eqref{transversality_expl}  write in terms of the function $F$ given by \eqref{FunktsiyaF} as follows %in terms of $F$ 
		\begin{equation}
		\label{VseVterminahF}
		F(R_0)=0, \  F^\prime(R_0)\not =0.
		\end{equation}
	\end{rem}
\begin{rem} \label{eigenvalue_stationary}
Lemma \ref{och_polezn_lemma}  and comparison of the problems \eqref{bicond_po_drugomu} and \eqref{BifCond1}--\eqref{BifCond1bis}   show that  the necessary bifurcation condition $F(R_0)=0$ on the critical radius $R_0$ is satisfied iff $E(R_0)=0$.  Moreover, we show below (see  \eqref{formula_for_E}) that the second  condition  in \eqref{VseVterminahF} (the transversality condition) is satisfied iff  $E^{\prime}(R_0)\not=0$.	
\end{rem} 

\begin{rem}\label{rem:boundary}
		If $\Phi=\Phi(x,y,V)$ and $\Omega(V)=\{0<r<R_0+\rho_{\rm tw}(\varphi, V)\}$  are the solutions of \eqref{tw_Liouvtypeeq}--\eqref{addit_cond_tw}, %for sufficiently small  $V$, 
then  for small velocities $V$ the first term in the  asymptotic expansion  of the function $\rho_{\rm tw}$   is of order $V^2$, i.e.  $\|\rho_{\rm tw}\|_{C^2}\leq CV^2$,
% The same  observation holds for $\Lambda(V)$,  
also $\Lambda(V)=\Lambda(0) + O(V^2)$, $\Lambda^\prime(0)=0$. These properties  follow from Theorem 1.18  in \cite{CraRab1971}.
%   (formula (1.19)), in the proof of Theorem \eqref{biftwtheorem}  the bifurcation occurs via the eigenvector $(V=1, \rho=0)$. 
Combining this with elliptic estimates one can improve bounds for $\rho_{\rm tw}$ to
\begin{equation}
\label{boundsforrho}
 \|\rho_{\rm tw}\|_{C^j}\leq CV^2 \quad \forall j\in\mathbb{Z}_+
\end{equation}
with $C$ depending only on $j$, and derive the following expansion for $\Phi$,
\begin{equation}
\label{expansion_of_Phi}
\Phi(x,y,V)=m_0/\zeta+V\Phi_V^0(x,y)+V^2\tilde\Phi(x,y,V), 
\end{equation}
where 

\begin{equation}\Phi_V^0=\tilde \Phi_V^0(r,R_0)\cos\varphi
\end{equation}\label{myosin_derivative_zero}	
	 is the unique solution of (cf. \eqref{bicond_po_drugomu})
\begin{equation}
\label{def_of_Phi_V}
\Delta \Phi_V^0+m_0(\Phi_V^0-x)=\zeta\Phi_V^0 
\end{equation}
in $B_R$, which satisfies
\begin{equation}
\label{def_of_Phi_Vadditional}
 \Phi_V^0=0\quad\text{on}\ \partial B_R,\quad \text{and additionally} \quad \partial_r \Phi_V^0=\cos\varphi\quad\text{on}\ \partial B_R,
\end{equation}
and functions  $\tilde \Phi $ are uniformly (in $V$)  bounded in $C^{j}(\overline\Omega(V))$ $\forall j\in\mathbb{Z}_+$. Note that $\Phi_V^0$ %\emph{ 
extends as the solution of \eqref{def_of_Phi_V} to the entire space $\mathbb{R}^2$,  being a product 
of  radially symmetric function and $\cos\varphi$.
\end{rem}

%		\textcolor{brown}{In Theorem 1.7 of \cite{CraRab1971}, $F$ is the nonlinear operator \eqref{operator_equation}, parameter $\alpha$ in C-R is equivalent to the parameter $V$ (velocity), and one can replace $\alpha$ with $V$ after a reparameterization., $t$ can represent either radius or total Myosin, $\varphi(\alpha)=M(V)$. $V$ from C-R is the space $\mathbb R\times H^2_{\rm per}(0,2\pi)$, that is the space of $(V,\rho)$ values. $x_0$ is $(V,\rho)=(1,0)$. (1.8) in C-R paper shows that the steady states and traveling waves (that is, all the solutions to the homogeneous problem $\mathcal B(V,\rho;R)=0$ for the operator $\mathcal B$ in \eqref{operator_equation}) are comprised of the set $\{(M,(V,\rho)):V=0,\rho=0,M\in\text{ some interval}\}$ (SSs), and the set $\{(M(V),(V,\rho(V))):\rho(V)=V(m_0\pi R_0^2,0)+O(V^2)\}$ (TWs). In the C-R paper, there is an additional parameter $\alpha$}
%%	For sufficiently small $V$ the expansion   
%
%
%\begin{rem}
%	content...
%\end{rem}

	The following lemma relates $F'(R)$, which appears in the bifurcation condition \eqref{VseVterminahF}, to $\Phi_V^0$ 
(the unique solution of \eqref{def_of_Phi_V}--\eqref{def_of_Phi_Vadditional}) and the derivative of the eigenvalue $E(R)$.
	
	\begin{lem} 
\label{poleznyeformul'ki}
If $F(R)=0$, then the function $F^\prime(R)$ can be written in terms of  $\Phi_V^0$  as
		\begin{equation}\label{eq:transversality_condition_relation}
		\pi \zeta R F^\prime(R)=
		\pi R (\zeta+\tilde\Lambda-(\tilde\Lambda R)^2)-(\gamma /R^2+2\pi R p_\ast^\prime(\pi R^2))\int_{B_R}|\nabla (\Phi_V^0-x)|^2\,dx\,dy.
		\end{equation}
		Also 
		\begin{equation}
		E^\prime(R)=-\frac{\pi \zeta R F^\prime(R)}{\tilde\Lambda \int_{B_R}(\Phi_V^0-x)^2dxdy}.
		\label{formula_for_E}
		\end{equation}
	\end{lem}
	\begin{proof} We have $\Phi_V^0=\tilde\Phi_V^0(r, R)\cos\varphi$,  and $\partial_R\tilde\Phi_V^0(r, R)$ satisfies the equation 
		\begin{equation}
		\label{ur_e_forsms}
		\frac{1}{r}(r\partial_R (\tilde\Phi_V^0)^\prime)^\prime-\frac{1}{r^2} \partial_R \tilde\Phi_V^0+(\tilde\Lambda(R)-\zeta)\partial_R\tilde\Phi_V^0(r, R) =-\tilde \Lambda^\prime(R)(\tilde\Phi_V^0- r) \quad 0\leq r<R,
		\end{equation}
		with boundary conditions
		\begin{equation}
		\partial_R\tilde\Phi_V^0(0,R)=0, \quad (\partial_R\tilde \Phi_V^0)^\prime(R, R)=-(\tilde\Phi_V^0)^{\prime\prime}(R,R)=1/R-\tilde \Lambda R.
		\end{equation}
		Multiply \eqref{ur_e_forsms} by $r\tilde\phi:=-r(\zeta-\tilde\Lambda)\tilde\Phi_V^0-\tilde\Lambda r^2$ 
and integrate the result to get, using integration by parts,
		$$
		R\left(\tilde\Lambda R(\tilde\Lambda R- 1/R) +\zeta \partial_R\tilde\Phi_V^0(R,R)\right)=-\tilde \Lambda\tilde\Lambda^\prime \int_0^R
		\left(\tilde\Phi_V^0-r\right)^2rdr
		+\zeta \Lambda^\prime \int_0^R \left(\tilde\Phi_V^0-r\right)\tilde\Phi_V^0 r dr,
		$$
		or, since $\tilde\Phi_V^0(r, R)-r$ satisfies 
		$\frac{1}{r}(r((\tilde\Phi_V^0)^\prime-1))^\prime-\frac{1}{r^2} (\tilde\Phi_V^0-r) =-\tilde\Lambda(\tilde\Phi_V^0-r)+\zeta\tilde\Phi_V^0$ in $(0,R)$,
		\begin{multline*}
		R\zeta  (\partial_R \tilde\Phi_V^0(R, R)+1)=R\zeta-\tilde\Lambda R(\tilde\Lambda R^2-1)-\tilde\Lambda^\prime\int_0^R
\left( \left((\tilde\Phi_V^0)^\prime-1\right)^2+
		\left(\tilde\Phi_V^0(r, R)-r\right)^2/r^2\right)rdr
\\
=R (\zeta+\tilde\Lambda-(\tilde\Lambda R)^2)-\frac{\gamma /R^2+2\pi R p_\ast^\prime(\pi R^2)}{\pi}\int_{B_R}|\nabla (\Phi_V^0-x)|^2\,dx\,dy.
		\end{multline*}
Owing to the fact that $(\tilde\Phi_V^0)^\prime(R,R)=1$ the left hand side of the above relation equals $\zeta R F^\prime(R)$. Thus 
\eqref{eq:transversality_condition_relation} is proved.

		Now we calculate the derivative of the eigenvalue $E=E(\tilde R)$ at $\tilde R=R$. Recall that $E(R)=0$ and 
(for $\tilde R$ in a small neighborhood of $R$) $E(\tilde R)$ is
a simple eigenvalue of the equation 
\begin{equation}
\label{ekvivaleigenvalprob}
E m=\Delta m+\tilde\Lambda (\tilde R)m-\zeta\tilde\Lambda(\tilde R)\phi(m) \quad \text{in}\ B_{\tilde R} 
\end{equation} 
with the boundary condition $\partial_r m=0$ on $\partial B_{\tilde R}$, where $\phi(m)$ is the unique solution of the equation 
$\Delta \phi+m=\zeta\phi$ in  $B_{\tilde R}$ 
subject to the boundary condition $\phi=0$  on  $\partial B_{\tilde R}$. Since the problem smoothly depends on the parameter $\tilde R$, 
the eigenvalue $E(\tilde R)$ is a smooth function  of the parameter and one can choose a smooth family of eigenfunctions $m(x,y,\tilde R)=\hat m(r,\tilde R)\cos\varphi$ such that 
$m(x,y,R)=m_0(\Phi_V^0-x)$. Therefore we can differentiate \eqref{ekvivaleigenvalprob} in $\tilde R$  to find that $\partial_{\tilde R} m$ at $\tilde R=R$ satisfies 
\begin{equation}
\label{ekvivaleigenvalprob1}
E^\prime m=\tilde\Lambda^\prime m-\zeta\tilde\Lambda^\prime\phi(m)+\Delta\partial_{\tilde R} m+\tilde\Lambda\partial_{\tilde R} m
-\zeta\tilde\Lambda\partial_{\tilde R}\phi(m) \quad \text{in}\ B_R.
\end{equation} 
Also, differentiating the equality $\partial_r\hat m(\tilde R, \tilde R)=0$ in $\tilde R$ at $\tilde R=R$ we find, 
\begin{equation}
\label{analogichno}
\partial_r\partial_{\tilde R}\hat m=-\partial^2_{rr}\hat m=
-\tilde \Lambda \partial^2_{rr}\tilde\Phi_V^0=\frac{\tilde \Lambda}{R}\partial_r\tilde\Phi_V^0-\tilde \Lambda^2 R=\frac{\tilde \Lambda}{R}-\tilde \Lambda^2 R.
\end{equation}
Now multiply \eqref{ekvivaleigenvalprob1}  by $\Phi_V^0-x$ (the pair $\tilde m= \Phi_V^0-x$, $\tilde\rho=0$ is an element of the null space of the adjoint operator) and integrate over $B_{R}$.  
We have
\begin{equation}
\label{VyChIsLeNiYa}
\begin{aligned}
E^\prime\tilde\Lambda \int_{B_R}(\Phi_V^0-x)^2dxdy&=
-\tilde\Lambda^\prime \int_{B_R} \Delta  \Phi_V^0 (\Phi_V^0-x) dx dy-\int_{\partial B_R} \partial_r\partial_{\tilde R} m\, x ds\\
&+\int_{B_R} \partial_{\tilde R} m \left(\Delta  \Phi_V^0+\tilde \Lambda(\Phi_V^0-x)\right) dx dy-\zeta \tilde \Lambda \int_{B_R}\partial_{\tilde R} \phi (\Phi_V^0-x) dx dy.
\end{aligned}
\end{equation}
Since $\Delta  \Phi_V^0+\tilde \Lambda(\Phi_V^0-x)=\zeta\Phi_V^0$ in $B_R$, $\Phi_V^0=0$ on $\partial B_R$ and $\Delta \partial_{\tilde R} \phi+\partial_{\tilde R}m=\zeta\partial_{\tilde R}\phi$ in  $B_{R}$, the last line in \eqref{VyChIsLeNiYa} rewrites as $\zeta \int_{\partial B_R}\partial_{\tilde R}\phi \partial_r \Phi_V^0\,ds$, or equivalently, $
\zeta \int_{\partial B_R}\partial_{\tilde R}\phi \, x/R\,ds$. Thus 
\begin{equation}
\label{NuzadolbaloUzhe}
E^\prime\tilde\Lambda \int_{B_R}(\Phi_V^0-x)^2dxdy=\tilde\Lambda^\prime \int_{B_R} |\nabla(\Phi_V^0-x)|^2 dx dy+
\int_{\partial B_R} (\zeta \partial_{\tilde R}\phi /R -\partial_r\partial_{\tilde R} m)\, x\,ds.
\end{equation}
Similarly to \eqref{analogichno} one can calculate $\partial_{\tilde R}\phi=-\partial_r \Phi_V^0=-x/R$, so that 
$\zeta \partial_{\tilde R}\phi -R \partial_r\partial_{\tilde R} m=(\tilde\Lambda^2 R^2-\zeta-\tilde \Lambda)\cos \varphi$. Substituting this into \eqref{NuzadolbaloUzhe}
and calculating $\tilde\Lambda^\prime (R)= \gamma/R^2 +2\pi R p_{\ast}^\prime(\pi R^2)$ completes the proof of  Lemma \ref{poleznyeformul'ki}.
%
%\begin{equation}
%%\label{VyChIsLeNiYa}
%\begin{aligned}
%		&=-\int_{\partial B_R} \partial_r (\partial_R m-\Lambda\partial_R\phi) x ds- \int_{B_R} (\partial_R m-\Lambda\partial_R\phi)\tilde\phi dx dy+\partial_R\Lambda\int_{B_R} |\nabla  (\partial_V \Phi-x)|^2 dx dy\\
%		&=-\int_{\partial B_R} \partial_r \partial_R m x ds-\int_{\partial B_R} \partial_R\phi \partial_r \tilde\phi ds +\partial_R\Lambda\int_{B_R} |\nabla  (\partial_V \Phi-x)|^2 dx dy
%		\\
%		&=\Lambda\int_{\partial B_R} x \partial^2_{rr}\partial_V \Phi ds +\int_{\partial B_R} \partial_r\partial_V \Phi \partial_r \tilde\phi ds+\partial_R\Lambda\int_{B_R} |\nabla  (\partial_V \Phi-x)|^2 dx dy\\
%		&=
%		\pi \Lambda R^2(\Lambda R-1/R)-\pi R\zeta+\partial_R\Lambda
%		\int_{B_R} |\nabla  (\partial_V \Phi-x)|^2 dx dy.
%		\end{aligned}
%		\end{equation}
	\end{proof}

	Finally, we demonstrate qualitative agreement of our analytical results with  experimental results from \cite{VerSviBor1999}. First observe that Theorem \ref{biftwtheorem}  establishes existence  of a smooth family of traveling waves in the  model \eqref{actflow_in_terms_of_phi}--\eqref{myosin_bc_I}. Then one can  obtain asymptotic expansions of  traveling waves  (solutions of \eqref{tw_Liouvtypeeq}--\eqref{addit_cond_tw}) in small velocities $V$, similarly to Appendix in \cite{BerFuhRyb2018}.  The plots of these expansions up to $V^3$ show that  as the velocity increases, the cell shape  becomes asymmetric, with flattening of its front and the  myosin  accumulates  at the rear. This myosin accumulation is consistent with the 1D results in \cite{RecPutTru2013} and \cite{RecPutTru2015}, and the 2D results in \cite{ZieAra2015} (see Fig.3 in \cite{ZieAra2015}).  
	%These features are seen in Fig.1, which was obtained by plotting our analytical results (rather than solving the original PDEs numerically). Namely  the pictures in Fig. 1 were obtained numerically  by plotting  asymptotic expansions in small velocities $V$, similarly to Appendix in \cite{BerFuhRyb2018}, by substituting ansatz    
	 %as depicted in Fig.1.    by computing numerically the shape  and the distribution of myosin in the cell for traveling wave solutions with small velocities $V$.   Solutions are obtained via asymptotic expansions in small velocities $V$, similarly to Appendix in \cite{BerFuhRyb2018}, by substituting ansatz  
%	\begin{equation}\label{eq:phi_expansion}
%	\phi=\phi_0+V \phi_1+V^2 \phi_2+...,
%	\end{equation}
%	\begin{equation}\label{eq:domain_expansion}
%	\Omega=\left\{0\leq r \leq R_0+ V \rho_1(\varphi) + V^2 \rho_2(\varphi)+...\right\},
%	\end{equation}
%	\begin{equation}\label{eq:lambda_expansion}
%	\Lambda=\Lambda_0+V \Lambda_1+V^2\Lambda_2...
	%\end{equation} into \eqref{tw_Liouvtypeeq}-\eqref{addit_cond_tw}. Results are depicted in Figure~\ref{fig:john}.  
	\begin{figure}[h]
		\begin{center}
			\includegraphics[width=0.85\textwidth]{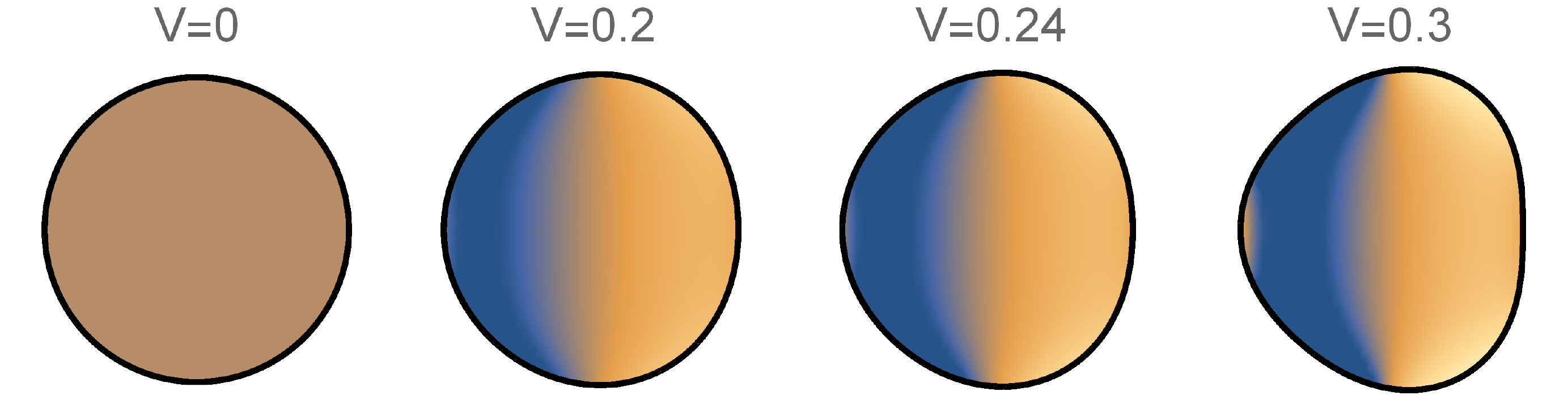}
			\caption{
				Approximate shape and myonsin distribution of traveling wave solutions for $m_0=0.62$, $R=3.6$, $\gamma=3.5$, and $k_e=5.0$ with $V=0,\,0.2,\, 0.24,\,0.3$ bifurcated from the radial steady state. The direction of motion is to the right, and blue color corresponds to higher myosin density.  
				\label{fig:john}
			}
		\end{center}
	\end{figure}

	\section{Asymptotic expansions of  eigenvectors of the linearized problem}
	\label{section_lin_stab_tw_formal}

In this and the next Sections we study the spectrum of the linearized operator around  traveling wave solutions, 
i.e. those established in Theorem \ref{biftwtheorem}. 
We begin with formal asymptotic  
expansions for small velocities $V$. These expansions  will be justified in Section \ref{section_lin_stab_tw_justify}.

It is convenient to pass from the polar coordinates to the parametrization of 
domains via the signed distance $\rho$ from the reference domain $\Omega$. More precisely, given a  solution $\phi=\Phi(x,y,V)$, $\Omega=\Omega(V)$ 
of problem \eqref{tw_Liouvtypeeq}-\eqref{addit_cond_tw}, we describe perturbations $\partial \tilde \Omega$ of the boundary   by the function $\rho$ such that 
$\partial\tilde \Omega=\{(x(s),y(s))+\rho(s)\nu(s);\ (x(s),y(s))\in \partial\Omega(V) \}$, where $s$ is the arc length parametrization of  $\partial\Omega(V)$ and
 $\nu(s)$ denotes the outward pointing unit normal to   $\partial\Omega$. Then the linearized problem around the traveling wave solution writes as 
\begin{equation}
	\label{Hact_flow_lin_op}
	\Delta \phi + m=\zeta \phi \quad\text{in}\ \Omega(V),%\Omega_{\rm tw},
	\end{equation}
	\begin{equation}
	\label{Hact_flow_BC_dir_op}
	\zeta(\phi+V\nu_x \rho) =p_{\ast}^\prime(| \Omega(V)|)\int_{\partial \Omega(V)}
	\rho(s)ds
	+\gamma(\rho^{\prime\prime}+\kappa^2 \rho)\quad\text{on}\   \partial\Omega(V),% \partial\Omega_{\rm tw},
	\end{equation}
	\begin{equation}
	\label{Hact_flow_sec_BC}
	\partial_t \rho=(\mathcal{A}(V)(m,\rho))_\rho:=\frac{\partial \phi}{\partial \nu}+\frac{\partial^2 \Phi}{\partial \nu^2}\rho-\left(\frac{\partial \Phi}{\partial \tau}+V \nu_y\right)\rho^\prime\quad\text{on}\ \partial \Omega(V),%\Omega_{\rm tw},
	\end{equation}
\begin{equation}
	\label{Hmyosin_lin_eq}
	\partial_t m=(\mathcal{A}(V)(m,\rho)))_m:=\Delta {m}+V\partial_x m -\div( \Lambda e^{\Phi-Vx} \nabla {\phi})
	-\div(m\nabla \Phi) \quad\text{in}\   \Omega(V),%\Omega_{\rm tw}.
	\end{equation}
\begin{equation}
	\label{Hmyosin_lin_BC}
	\partial_\nu m+\Lambda e^{\Phi-Vx}\left(\frac{\partial^2 \Phi}{\partial \nu^2}\rho -\Bigl(\frac{\partial \Phi}{\partial \tau}+V \nu_y\Bigr)\rho^\prime\right)=0
	\quad \text{on}\ \partial  \Omega(V).%\Omega_{\rm tw}.
\end{equation}
Here and in what follows $\rho^\prime$, $\rho^{\prime\prime}$ denote derivatives of $\rho$ with respect to the arc length $s$, $\kappa$ is the curvature of $\partial\Omega(V)$, and $\partial/\partial\tau$ denotes the tangential derivative on $\partial\Omega(V)$. The linearized operator 
$\mathcal{A}(V)$ appearing in \eqref{Hact_flow_lin_op}-\eqref{Hmyosin_lin_BC} is well defined on smooth $m\in L^2(\Omega(V))$, $\rho\in L^2(\partial\Omega(V))$ 
such that \eqref{Hmyosin_lin_BC} holds. It can be extended to the closed operator in $L^2(\Omega(V))\times L^2(\partial\Omega(V))$ whose domain is 
the set of pairs $(m,\rho)$ from $H^2(\Omega(V))\times H^3(\partial\Omega(V))$ satisfying \eqref{Hmyosin_lin_BC}.

Since traveling wave solutions bifurcate from radial stationary solutions the spectrum of the operator $\mathcal{A}(V)$ for small $|V|$ is close to the 
spectrum of the operator $\mathcal{A}_{ss}$ representing linearization around the radial stationary solution at the critical radius $R=R_0$ (heareafter we  write 
simply $R$ for  brevity). The latter oprerator has zero eigenvalue with multiplicity three while other eigenvalues are bounded away from zero. Therefore in order to study stability of traveling wave solutions it suffices to investigate what happens with zero eigenvalue for small $V\not=0$. Observe that $\forall V$ the operator  
$\mathcal{A}(V)$ has zero eigenvalue with multiplicity at least two. It has the eigenvector 
\begin{equation}
\label{tw_shifts1}
m_1=-\Lambda(V)\partial_x\Phi(x,y,V), \quad \rho_1=\nu_x
\end{equation}
corresponding to the infinitesimal shifts along the $x$-axis, and the generalized eigenvector
\begin{equation}
\label{TwderivativeinV1}
m_2=\partial_V\left(\Lambda(V)e^{\Phi(x,y,V)-Vx}\right), \quad \rho_2=\partial_{\tilde V}\tilde \rho\bigl|_{\tilde V=V}\bigr.,
\end{equation}
 that satisfies $\mathcal{A}(V)(m_2,\rho_2)=(m_1,\rho_1)$, where $\tilde \rho$ describes the boundary of the traveling wave with velocity $\tilde V$ via 
the signed distance to $\partial \Omega(V)$. The zero eigenvalue has multiplicity three for $V=0$, as follows from Lemma \eqref{och_polezn_lemma}.
It will be shown that one eigenvalue $\lambda=\lambda(V)$ becomes nonzero for $V\not =0$ and its asymptotic behavior as $V\to 0$ 
is studied below. The main difficulty in this analysis comes from the fact that the eigenvector corresponding to $\lambda(V)$ merges asymptotically as $V\to 0$ 
with the eigenvector $(m_1,\rho_1)$. Moreover, the next term in the expansion of this eigenvector is proportional to $(m_2,\rho_2)$. That is why the asymptotic problem for $\lambda(V)$ is a kind of singularly perturbed problem.

% According to Theorem \ref{thm_linearizeddisk} the operator $\mathcal{A}_{ss}$ has zero eigenvalue of multiplicity three while other eigenvalues 
%are sparated out of zero (this statement is proved in Theorem \ref{thm_linearizeddisk} assuming that 

We seek the eigenvalue $\lambda$ and the eigenvector  $(m, \rho)$ in the form 
		\begin{equation}\label{eq:lambda_ansatz}
		\lambda(V)= \hat \lambda V^2 +\dots,
		\end{equation}
		\begin{equation}
		\label{AnSaTzform}
		m=m_1+\hat \lambda V^2 m_2+V^3 m_3+V^4 m_4 +V^5 m_5+\dots
		\end{equation}
		\begin{equation}
		\rho=\rho_1+\hat \lambda V^2\rho_2+V^3 \rho_3+V^4 \rho_4 +V^5 \rho_5+\dots
		\label{AnSaTzforrho}
		\end{equation}
		with unknown $\hat \lambda$, $m_k$, $\rho_k$ ($k=3,4,5$) which do not depend on $V$, and will be found via perturbation expansion in $V$. In contrast, $m_1$, $m_2$, $\rho_1$, and $\rho_2$ are expressed in terms of the traveling wave solution via \eqref{tw_shifts1} and \eqref{TwderivativeinV1}, and 
{\it do depend} on $V$. Observe that $\lambda(V)$ is an even function of $V$ due to the symmetry with respect to the change $V\mapsto -V$. 
Therefore, the ansatz \eqref{eq:lambda_ansatz} starts with a quadratic term. It is convenient to consider $\phi$ now as independent unknown, seeking this function  in the form 
		\begin{equation}
		\phi =-\partial_x  \Phi +\hat \lambda V^2 \partial_V \Phi+V^3 \phi_3+V^4 \phi_4 +V^5 \phi_5+\dots\, .
		\label{AnSaTzforphi}
		\end{equation}
	%	where the traveling wave solution $\Phi=\Phi(x,y,V)$ solves \eqref{tw_Liouvtypeeq}-\eqref{addit_cond_tw}, and $\Phi|_{V=0}=\text{const.}$, which will be used in the expansions near $V=0$.
		
		Substitute these expansions into the equation $\mathcal{A}(V)(m,\rho)=\lambda(V) (m,\rho)$ and collect terms of the order $V^3$, replacing 
		$\Omega(V)$ by the disk $B_R$ (it approximates  $\Omega(V)$ to the order $V^2$, see Remark \ref{rem:boundary}).  This leads to the following problem for $m_3$, $\phi_3$ and  $\rho_3$,
		\begin{equation}\label{eq:phi_equation_third_order}
		\begin{split}
		&\Delta \phi_3+ m_3=\zeta \phi_3 \quad  \text{in}\ B_R,\\
		&\partial_r \phi_3 =0,\quad\zeta \phi_3=p_\ast^\prime(\pi R^2)R\int_{-\pi}^{\pi}\rho_3 d\varphi+\frac{\gamma}{R^2}(\partial^2_{\varphi\varphi}\rho_3+\rho_3)\quad \text{on}\ \partial B_R,
		\end{split}
		\end{equation}
		\begin{equation}\label{eq:myosin_equation_third_order}
		\Delta m_3-m_0 \Delta \phi_3=0\quad \text{in}\ B_R,\quad \partial_r m_3=0\quad\text{on}\ \partial B_R.
		\end{equation}
		%Note that formally the above problem should be considered in $\Omega_{\rm tw}$. 
		%However, postulating the ansatz \eqref{AnSaTzform}-\eqref{AnSaTzforphi} we assume $m_3$, $\phi_3$ being independent of $V$ which 
		%forces us to replace $\Omega_{\rm tw}$ by the disk $B_R$ (it approximates  $\Omega_{\rm tw}$ to the order $V^2$). 
		Thus, up to the eigenvector corresponding to the infinitisimal shifts of the disk $B_R$,  
		\begin{equation}\label{eq:third_order_solutions}
		\rho_3=\alpha,\ \zeta \phi_3=\alpha(\gamma /R^2+2\pi R p_\ast^\prime(\pi R^2)), \ m_3=\zeta \phi_3.
		\end{equation} 
The parameter $\alpha$ in the solution to the homogeneous problem \eqref{eq:phi_equation_third_order}--\eqref{eq:myosin_equation_third_order} will be determined by considering higher order terms of the expansions.
		
%		In order define $\hat \lambda$ in \eqref{eq:lambda_ansatz}, we need to resolve the fourth and fifth order in the expansion. 
%To this end, 
Next we collect terms of the order $V^4$ arriving at the following problem for $m_4$, $\phi_4$ and  $\rho_4$,
		\begin{equation} \label{4first2D}
		\Delta \phi_4+m_4=\zeta\phi_4\quad \text{in}\ B_R,
		\end{equation}
		\begin{equation}\label{5first2D}
		\Delta m_4-m_0 \Delta \phi_4=\hat\lambda^2 m_0 (\Phi_V^0-x)
		+\alpha (\gamma /R^2+2\pi R p_\ast^\prime(\pi R^2))\Delta \Phi_V^0
		\quad
		\text{in}\ B_R,
		\end{equation}
		\begin{equation}\label{6first2D}
		\zeta\left(\phi_4+\alpha \partial_r\Phi_V^0\right) =p_{\ast}^\prime(\pi R^2)R \int_{-\pi}^{\pi}
			\rho_4(\varphi)d\varphi
		+\frac{\gamma}{R^2}(\partial^2_{\varphi\varphi}\rho_4+\rho_4) \quad\text{on}\ \partial B_R,
		\end{equation}
		\begin{equation}\label{7first2D} 
		 \partial_r m_4+\alpha m_0 \partial^2_{r r}\Phi_V^0=0\quad \text{on}\ \partial B_R,
		\end{equation}
		\begin{equation}\label{8first2D}
		\begin{aligned}
		%	\hat\lambda^2\partial_V \rho_{\rm tw}
		\frac{\partial \phi_4}{\partial r}+
		\alpha\partial^2_{r r}\Phi_V^0=0
%+\frac{\alpha}{R}\partial^2_{Vr}({\Phi-Vx})|_{V=0} 
\quad \text{on}\ \partial B_R.
		%-V\frac{\rho^\prime\sin\varphi+\rho\cos\varphi}{\sqrt{(\rho_{\rm tw}^\prime)^2+(R_0+\rho_{\rm tw})^2}}\\
		%&+V\frac{(\rho^\prime_{\rm tw}\sin\varphi+(\rho_{\rm tw}+R_0)\cos\varphi)(\rho_{\rm tw}^\prime\rho^\prime +(R_0+
		%\rho_{\rm tw})\rho)}
		%{((\rho_{\rm tw}^\prime)^2+(R_0+\rho_{\rm tw})^2)^{3/2}}
		\end{aligned}
		%\phi^\prime=-\Phi^{\prime\prime}_V(\pm L)(\beta \pm\alpha)\ \text{at} \ x=\pm L.
		\end{equation}
To determine  solvability of \eqref{4first2D}--\eqref{8first2D} observe that after adding  $\alpha m_0 \partial_{r}\Phi_V^0$ to $m_4$ and 
$\alpha \partial_{r}\Phi_V^0$ to $\phi_4$, the problem is transformed to the form ${\mathcal A}_{ss}(m_4+\alpha m_0 \partial_{r}\Phi_V^0,\rho_4)=
(f(r)\cos\varphi,\varrho\cos\varphi)$ with some function $f(r)$ and a constant $\varrho$. Conditions for solvability of the  latter problem are provided by the  Fredholm alternative, i.e. one has to satisfy orthogonality of  $(f,\varrho)$ to solutions of the adjoint homogeneous problem
	\begin{equation}\label{eqn1adj2D} 
		\Delta \tilde m +\tilde \phi=0\quad\text{in}\ B_R,\quad \partial_r \tilde m=0\quad \text{on}\ \partial B_R,
		\end{equation}
		\begin{equation}\label{eqn2adj2D}
		\Delta \tilde \phi-\zeta\tilde \phi -m_0 \Delta \tilde m=0\quad \text{in}\ B_R,
		\end{equation}
		with boundary conditions
		\begin{equation} \label{eqn3adj2D}
		\tilde\rho- m_0 \tilde m+\tilde \phi=0\ \text{on}\ \partial B_R, 
		\end{equation}
		\begin{equation} 
		\label{eqn4adj2D}
-\frac{\gamma}{R}(\partial^2_{\varphi\varphi}\partial_r\tilde \phi +\partial_r\tilde \phi)
         -\frac{R}{\zeta} p_\ast^\prime(\pi R^2)\int_{-\pi}^{\pi}\partial_r\tilde \phi\, d\varphi=0
\quad \text{on}\ \partial B_R.
		\end{equation}
	This problem has a nontrivial solution $\tilde m=\Phi_V^0-x$  and $\tilde\rho=m_0\tilde m-\tilde\phi$ with 
$\tilde \phi=-(\zeta-m_0)\Phi_V^0-m_0 x$ (note that actually $\tilde\rho=0$). 
In order to identify the unknown coefficient $\alpha$, multiply \eqref{5first2D} by $\tilde m$ and integrate,
		\begin{equation}
		\label{1relation to determine_alpha}
		\begin{split}
		\int_{B_R} (\Delta m_4-m_0 \Delta \phi_4)\tilde m dxdy&=\hat\lambda^2 m_0\int_{B_R}  ( \Phi_V^0-x)^2dxdy\\
		&+\alpha (\gamma /R^2+2\pi R p_\ast^\prime(\pi R^2))\int_{B_R}(\Phi_V^0-x) \Delta (\Phi_V^0-x) dxdy.
		\end{split}
		\end{equation}
The left hand side of \eqref{1relation to determine_alpha} rewrites as follows, using integration by parts,
		\begin{equation}
		\label{2relation to determine_alpha}
		\int_{B_R} (\Delta m_4-m_0 \Delta \phi_4)\tilde m dxdy=\int_{B_R}(\Delta \phi_4 -(\zeta-m_0)\phi_4)\tilde \phi dxdy=
		\int_{\partial B_R} (\partial_r\phi_4 \tilde\phi-\partial_r\tilde \phi \phi_4)ds.
		\end{equation}
	The functions appearing in the right hand side are $\partial_r\phi_4=\alpha(1/R-m_0 R)\cos\varphi $, $\phi_4=-\alpha \cos\varphi$, $\tilde\phi=-m_0 R\cos\varphi$, 
		$\partial_r \tilde\phi=-\zeta\cos\varphi$, therefore
		\begin{equation}
		\label{3relation to determine_alpha}
		\int_{B_R} (\Delta m_4-m_0 \Delta \phi_4)\tilde m dxdy=-m_0 \pi R^2 \alpha(1/R-m_0 R) -\alpha \zeta \pi R=\alpha\pi R ((m_0 R)^2-m_0-\zeta).
		\end{equation}
	Thus we get the following relation between $\hat\lambda$ and $\alpha$,
		\begin{equation}
		\label{4relation to determine_alpha}
		\begin{split}
		&\alpha\left\{(-\gamma /R^2-2\pi R p_\ast^\prime(\pi R^2))\int_{B_R}|\nabla (\Phi_V^0-x)|^2\,dx\,dy-
		\pi R ((m_0 R)^2-m_0-\zeta)\right\}\\
		&\hspace{3.1in}=-\hat\lambda^2 m_0\int_{B_R}  (\Phi_V^0-x)^2dxdy.
		\end{split}
		\end{equation}
%	\textcolor{red}{The coefficient in front of $\alpha$ in \eqref{4relation to determine_alpha} is related to the sign of $F'(R)$ in the transversality condition \eqref{VseVterminahF} as shown by \eqref{eq:transversality_condition_relation}.} %A change in the sign corresponds to a singularity in the dependence of $M''(0)$ on $k_e$.}
Besides  the solution $(\tilde m=\Phi_V^0-x,\tilde\rho=0)$
the  problem \eqref{eqn1adj2D}--\eqref{eqn4adj2D} has exactly one linearly independent solution $(\tilde m=1, \tilde \rho= m_0)$. 
%=\Lambda(V)e^{\Phi-Vx}|_{V=0})$. 
 Since $(f(r)\cos\varphi, \varrho \cos\varphi)$
is orthogonal to $(1,m_0)$ there is a solution of  \eqref{4first2D}--\eqref{8first2D}. Moreover, if we require additionally that $m_4$ has zero mean value  
%Note that defined up to shifts, and eigenvector $m_3$, $\rho_3$. Up to these eigenvectors  $\rho_4=0$,  the 
and $\rho_4$ is orthogonal to $\cos\varphi$ then $\rho_4=0$ and both $m_4$ and $\phi_4$ are represented in the form of  products of  radially symmetric functions 
and $\cos\varphi$. These radially symmetric factors  solve a system of second order ordinary differential equatiuons (with bounded coefficients, except at $r=0$) and therefore extend as solutions of \eqref{4first2D}--\eqref{5first2D} %\eqref{7first2D}
 to the whole $\mathbb{R}^2$. Thus the ansatz of the first four terms is well defined in $\Omega(V)$. However it  is not in the domain of the operator $\mathcal{A}(V)$ as the boundary condition \eqref{Hmyosin_lin_BC} is satisfied only approximately (with discrepancy of the order $O(V^5)$). That is why we introduce a correcting term $\tilde m_5$ such that 
\begin{equation}
\label{corrector_m_BC}
\partial_\nu \tilde m_5=-\frac{1}{V}\partial_\nu m_4-\frac{\Lambda(V)}{V^2}e^{\Phi-Vx}\partial^2_{\nu\nu}\Phi \rho_3
%=\frac{1}{V}\left(\frac{\partial_\varphi \rho_{\rm tw}}{R\sqrt {(\partial_\varphi\rho_{\rm tw})^2+(R+\rho_{\rm tw})^2}}\partial_\varphi m_4-\frac{R+\rho_{\rm tw}}{\sqrt {(\partial_\varphi\rho_{\rm tw})^2+(R+\rho_{\rm tw})^2}}\partial_r m_4
%\right)-\frac{\Lambda(V)}{V^2}e^{\Phi-Vx}\partial^2_\nu\Phi \rho_3  
\quad\text{on}\ \partial\Omega(V).
\end{equation}
In view of \eqref{7first2D}, bounds \eqref{boundsforrho} and the expansion \eqref{expansion_of_Phi} (see Remark \ref{rem:boundary}) one can show that the right hand side of \eqref{corrector_m_BC} defines functions
%normal derivatives $\partial_\nu \tilde m_5$ 
uniformly bounded in $C^j(\partial \Omega(V))$ $\forall j\in \mathbb{Z}_+$. Therefore we can define 
$\tilde m_5$ in $\Omega(V)$, e.g., by solving the equation $\Delta \tilde m_5=\tilde m_5$ with the boundary condition \eqref{corrector_m_BC} and set
\begin{equation}
\label{ansatzCorrected}
\tilde W:=(-\Lambda(V)\partial_xe^{\Phi-Vx},\nu_x)+\hat \lambda V^2\partial_{\tilde V} (\Lambda(\tilde V)e^{\Phi-\tilde Vx},\tilde\rho_{\rm tw}(s,\tilde V))\bigl|_{\tilde V=V}+V^3(m_3,\alpha)+V^4(m_4+V\tilde m_5,0),
\end{equation}
where $\tilde\rho_{\rm tw}$ stands for the parametrisation of $\partial \Omega(\tilde V)$ via the signed distance to $\partial \Omega(V)$.   The (corrected) four term ansatz given by \eqref{ansatzCorrected} is in domain of the operator $\mathcal{A}(V)$ and 
%where we have used
%$$
%\partial_\nu=\frac{R+\rho_{\rm tw}}{\sqrt {(\partial_\varphi\rho_{\rm tw})^2+(R+\rho_{\rm tw})^2}}\partial_r -\frac{\partial_\varphi \rho_{\rm tw}}{(R+\rho_{\rm tw})\sqrt {(\partial_\varphi\rho_{\rm tw})^2+(R+\rho_{\rm tw})^2}}\partial_\varphi 
%$$ 
%E.g. we can use $\tilde m_5$  solving 
%\begin{equation}
%	\label{corrector_m_EQN}
% \tilde m_5=\Delta {\tilde m_5}+V\partial_x \tilde m_5% -\div( \tilde \Lambda e^{\Phi-Vx} \nabla {\phi})
%	-\div(\tilde m_5\nabla \Phi) \quad\text{in}\ \Omega(V).
%	\end{equation}
introducing the unique solution $\tilde\phi_5$
of 
$$
\Delta \tilde\phi_5+\tilde m_5=\zeta \tilde \phi_5 \quad\text{in}\ \Omega(V)
$$
with the boundary condition
$$
\zeta(\tilde \phi_5+\frac{1}{V}\phi_4 +\frac{1}{V^2}\phi_3+\frac{\nu_x}{V} \rho_3) =
\frac{1}{V^2}\left(p_{\ast}^\prime(|\Omega|)\int_{\partial \Omega(V)}
	\rho_3 ds
	+\gamma\kappa^2 \rho_3\right)\quad\text{on}\ \partial\Omega(V),
$$
we can calculate the components of  $\mathcal {A}(V) W -\hat \lambda V^2 W$:
\begin{multline}
\label{mcomponent}
\frac{1}{V^5}(\mathcal {A} (V)\tilde W -\hat \lambda V^2 \tilde W)_m=%\tilde m_5-{\rm div}(\Lambda(V)e^{\Phi-Vx}\nabla \tilde\phi_5)
\Delta {\tilde m_5}+V\partial_x \tilde m_5 -\div( \Lambda e^{\Phi-Vx} \nabla {\tilde \phi_5})
	-\div(\tilde m_5\nabla \Phi)
 \\ +\frac{1}{V}\left\{\Delta { m_4}
         +V\partial_x  m_4 -\div(  \Lambda e^{\Phi-Vx} \nabla {\phi_4})	-\div(m_4\nabla \Phi) \right\}
\\
	+\frac{1}{V^2}\left\{\Delta {m_3}+V\partial_x m_3-\div( \Lambda e^{\Phi-Vx} \nabla {\phi_3})
	-\div(m_3\nabla \Phi) \right\}
\\
-\frac{\hat\lambda}{V}\left\{\hat \lambda\partial _V(\Lambda(V)e^{\Phi-Vx}) +V m_3+V^2 m_4+V^3 \tilde m_5\right\},
\end{multline}
\begin{equation}
\label{rhocomponent}
\frac{1}{V^5}(\mathcal {A}(V) \tilde W -\hat \lambda V^2 \tilde W)_\rho=\partial_\nu \tilde \phi_5+
\frac{1}{V}\partial_\nu\phi_4 +\frac{1}{V^2}\partial_\nu\phi_3 +\frac{1}{V^2}\partial^2_{\nu\nu}\Phi \rho_3
-\frac{\hat\lambda}{V}\left\{\hat\lambda\partial _{\tilde V}\tilde \rho_{\rm tw}\bigl|_{\tilde V=V}\bigr.+V\rho_3\right\}.
\end{equation}
Thanks to \eqref{eq:phi_equation_third_order}, \eqref{6first2D} we have $\|\tilde\phi_5\|_{C^j(\partial\Omega(V))}\leq C_j$ $\forall j\in \mathbb{Z}_+$, 
also one can verify that normal derivatives $\partial_\nu \tilde \phi$ are uniformly bounded in $C^j(\partial \Omega(V))$ $\forall j\in \mathbb{Z}_+$, similarly to 
$\partial_\nu \tilde m_5$. Furthermore, since $m_3$, $\phi_3$ are constants the third line of  
\eqref{mcomponent} simplifies to $-m_3\Delta \Phi / V^2$  and substituting $\Delta m_4$ from \eqref{5first2D} we obtain after rearranging terms,  
\begin{equation}
\label{mcomponentbis}
\begin{aligned}
\frac{1}{V^5}(\mathcal {A} (V)\tilde W -\hat \lambda V^2\tilde W)_m=\Delta {\tilde m_5}+V\partial_x \tilde m_5 -\div( \Lambda e^{\Phi-Vx} \nabla {\tilde \phi_5})
	-\div(\tilde m_5\nabla \Phi)
 \\ -\frac{\hat\lambda^2}{V}\left\{\partial _V(\Lambda(V)e^{\Phi-Vx}) - m_0 (\Phi_V^0-x)\right\}
\\+\frac{1}{V}\left\{
         V\partial_x  m_4 -\div(  (\Lambda e^{\Phi-Vx}-m_0) \nabla {\phi_4})	-\div(m_4\nabla \Phi) \right\}
\\
-\hat\lambda ( m_3+V m_4+V^2 \tilde m_5)	+\frac{m_3}{V^2}
	 (V\Delta \Phi_V^0-\Delta \Phi)
\end{aligned}
\end{equation}
\begin{equation}
\label{rhocomponentbis}
\frac{1}{V^5}(\mathcal {A}(V) \tilde W -\hat \lambda V^2\tilde W)_\rho=\partial_\nu \tilde \phi_5+
\frac{1}{V}\partial_\nu\phi_4 +\frac{1}{V^2}\partial^2_{\nu\nu}\Phi \rho_3
-\frac{\hat\lambda}{V}(\hat\lambda\partial _V\tilde \rho_{\rm tw}+V\rho_3).
\end{equation}
Thus $\|(\mathcal {A} (V)\tilde W -\hat \lambda V^2\tilde W)_m\|_{L^2(\Omega(V))}
+ \|(\mathcal {A} (V)\tilde W -\hat \lambda V^2 \tilde W)_{\rho}\|_{L^2(\partial\Omega(V))}\leq C|V|^5$.

Next assuming that there exists the next term $V^5(m_5,\rho_5)$ of the asymptotic expansion, we have, neglecting higher order terms, 
$\mathcal{A}(V)(m_5,\rho_5)=-\frac{1}{V^5}(\mathcal {A} (V)W -\hat \lambda V^2 W)$. Since the null space of the adjoint operator $\mathcal{A}^\ast (V)$ contains 
$(1,\Lambda(V) e^{\Phi-Vx})$ (see \eqref{eq:adjoint_eigenvalue}),  we will require that 
\begin{equation}
\label{umovarozvyaznosti}
\int_{\Omega(V)}(\mathcal {A} (V)\tilde W -\hat \lambda V^2 \tilde W)_m\, dx dy
+\Lambda(V)\int_{\partial \Omega(V)}(\mathcal {A}(V) \tilde W -\hat \lambda V^2 \tilde W)_\rho e^{\Phi-Vx}\, ds=0.
\end{equation}
We will see that this condition yields an asymptotic formula for $\hat\lambda$.

Let $I(\hat\lambda,V)$ denote the left hand side of \eqref{umovarozvyaznosti}, then with the help of integration by parts and 
\eqref{corrector_m_BC}, \eqref{rhocomponentbis} we get
\begin{equation}
\label{posledniiboi}
\begin{aligned}
I(\hat\lambda)&= \frac{1}{V}\int_{\partial \Omega(V)}(\Lambda(V)e^{\Phi-Vx}\partial_\nu\phi_4 -\partial_\nu m_4)\, ds  -\int_{\partial \Omega(V)}\hat\lambda\rho_3\Lambda(V)e^{\Phi-Vx}\,ds
\\
&+\int_{\Omega(V)} \left(V\partial_x \tilde m_5-\div(\tilde m_5 \nabla \Phi)\right)\, dxdy
+\frac{1}{V}\int_{\Omega(V)} (\hat\lambda^2 m_0(\Phi_V^0-x)+m_3 \Delta \Phi_V^0) \,dxdy
\\
&+\frac{1}{V}\int_{\Omega(V)} \left\{
         V\partial_x  m_4 -\div(  (\Lambda e^{\Phi-Vx}-m_0) \nabla {\phi_4})	-\div(m_4\nabla \Phi) \right\}\,dxdy\\
&- \hat \lambda \int_{\Omega(V)} ( m_3+V m_4+V^2 \tilde m_5)\,dxdy 
\\
&-\frac{\hat\lambda^2}{V}\left\{\int_{\partial \Omega(V)} \partial _V(\Lambda(V)e^{\Phi-Vx})\,ds+\int_{\partial \Omega(V)} \Lambda(V)e^{\Phi-Vx}\partial _V\tilde \rho_{\rm tw}\, ds\right\}.
\end{aligned}
\end{equation}
This formula is further simplified by observing that the first term in the second line of \eqref{posledniiboi} is zero thanks to the fact that $\partial_\nu\Phi=\nu_x$ on 
$\partial \Omega(V)$. Also, using \eqref{5first2D} in the second term, then collecting all the terms with the prefactor $\frac{1}{V}$ except the last line, we see  that these terms cancel each other. Finally, notice that the last line of \eqref{posledniiboi} equals $-\frac{\hat\lambda^2}{V} M^\prime(V)$.  Thus
\begin{equation}
\label{pochtiformula}
I(\hat\lambda,V)=-\hat\lambda\alpha\pi R(2 m_0+ \gamma/R+2\pi R^2 p^\prime_\ast (\pi R^2)+O(V))-\frac{\hat\lambda^2}{V} M^\prime(V),
\end{equation}
and substituting $\alpha$ from \eqref{4relation to determine_alpha} we see that the equation $I(\hat\lambda,V)=0$ has nonzero solution 
\begin{equation}
\label{sovsempochtiformula}
\hat\lambda (V)=-\frac{(\gamma /R^2+2\pi R p_\ast^\prime(\pi R^2))\int_{B_R}|\nabla (\Phi_V^0-x)|^2\,dxdy+
			\pi R ((m_0 R)^2-m_0-\zeta)}{\pi m_0 R (2m_0+\gamma/R+2\pi R^2p_{\ast}^\prime(\pi R^2) +O(V))\int_{B_R}  (\Phi_V^0-x)^2\,dxdy} \frac{ M^\prime(V)}{V}
\end{equation}
for sufficiently small $V$, provided that $M^\prime(V)\not=0$ when $V\not=0$. The first factor in \eqref{sovsempochtiformula}
simplifies, by virtue of \eqref{eq:transversality_condition_relation}--\eqref{formula_for_E}, to  $-\frac{d E}{d M}\bigl|_{M=M_0}\bigr.$, where 
$M_0$ is the total myosin mass of the stationary solution (with $V=0$). 
 In the nondegenerate case, $M^{\prime\prime}(0)\not=0$, the solution 
$\hat \lambda(V)$ has nonzero finite limit
\begin{equation}
\label{formulka}
\hat \lambda=-\frac{d E}{d M}\Bigl|_{M=M_0}\Bigr.\frac{d^2 M}{d V^2}\Bigl|_{V=0}\Bigr..
\end{equation}
If  $M^{\prime\prime}(0)=0$ but $M^\prime(V)\not =0$ for small $V\not=0$ there still exists a nonzero solution $\hat\lambda(V)$ of the equation
$I(\hat\lambda,V)=0$ and we can repeat the above construction observing that in this case $m_3$, $\rho_3$, $m_4$ and $\tilde m_5$ contain  the small 
factor $\alpha =O\left( (M^\prime(V)/V)^2\right)$.
%, this leads to a better bound for the norm of $\mathcal {A} (V)\tilde W -\hat \lambda(V) V^2\tilde W$. 
%Namely, 

We summarize the results   of 
the above asymptotic analysis in the following lemma  that provides the construction of the approximation  for the eigenvalue $\lambda(V)$ and the corresponding eigenvector.
It also plays an important role in the justification of the asymptotic formula for the eigenvalue $\lambda(V)$ in Section \ref {section_lin_stab_tw_justify}.

\begin{lem}
\label{TestFunctionLem} Assume that  $M^\prime(V)\not =0$ for small $V\not=0$, then there exists  $\hat\lambda(V)$ and $\tilde W=(\tilde m,\tilde\rho)$ in the domain of $\mathcal{A}(V)$ 
such that 
\begin{equation}\label{order2}
\tilde W=(-\Lambda(V)\partial_x e^{\Phi-Vx},\nu_x) +O(V^2),  
\end{equation}
\begin{equation}
\label{OTSENKA}
\|(\mathcal {A} (V)\tilde W -\hat \lambda(V) V^2\tilde W)_m\|_{L^2(\Omega(V))}
+ \|(\mathcal {A} (V)\tilde W -\hat \lambda V^2 \tilde W)_{\rho}\|_{L^2(\partial\Omega(V))}\leq C |M^\prime(V)|^2 |V|^3
\end{equation}
and 
\begin{equation}
\label{orthogonality}
\langle \mathcal {A} (V)\tilde W -\hat \lambda(V) V^2\tilde W, W_1^\ast \rangle_{L^2}=0,
%\int_{\Omega(V)} \tilde m \, dx dy +\Lambda(V)\int_{\partial \Omega(V)} \tilde \rho e^{\Phi-Vx}\,ds=0,
\end{equation}
where $\langle \,\cdot\,,\,\cdot\,\rangle_{L^2}$  denotes the pairing defined by \eqref{pairingL2}, and  $W_1^\ast=(1,\Lambda(V) e^{\Phi-Vx})$. 
Moreover, $\hat\lambda(V)\not =0$  for $V\not=0$ and is given by the asymptotic formula
\begin{equation}
\label{Final_formula}
\hat \lambda(V)=-\frac{d E}{d M}\Bigl|_{M=M_0}\Bigr.\frac{M^\prime(V)}{ V}(1+O(V)) \quad \text{as}\ V\to 0.
\end{equation}
\end{lem}   

%\begin{rem}The inequality \eqref{OTSENKA} estimates the discrepancy in  finding the eigenvalue of the form $\hat \lambda(V)V^2$. The equality \eqref{orthogonality} provides orthogonality of the approximate eigenvector  $(\tilde m, \tilde \rho)$ to the eigenvector of the adjoint operator  that corresponds to zero eigenvalue. Finally \eqref{Final_formula} provides an explicit representation of $\hat \lambda(V)V^2$ via total myosin mass $M$ and  the eigenvalues $E(M)$ of the  stationary solutions.   %In the proof of this Lemma we actually use the four term expansion of $\tilde W$  given by \eqref{ansatzCorrected}, however  in the subsequent considerations  the  % leading term  \eqref{order2} is sufficient.   
%	%For this Lemma the expansion of $\tilde W$ up to $V^2$ is sufficient but the 
%%	
%%	(more precisely, $\tilde W$ is given by \eqref{ansatzCorrected}) satisfies
%\end{rem}
%%%%%%%%%%%%%%%%%%%%PROOF

%\end{proof}
%
%DEFINITION OF ADJOINT
As usual in the spectral analysis of non self-adjoint boundary value problems, the adjoint operator  plays an important role. To define the adjoint operator introduce, for given 
smooth functions $\tilde m$ and $\tilde \rho$ defined on $\Omega(V)$ and $\partial \Omega(V)$, respectively, the auxiliary 
function $\tilde \phi$ as the unique solution of the problem
\begin{equation}
	\label{pressure_adj_tw}
	\Delta\tilde \phi -\zeta\tilde\phi -{\rm div}(\Lambda e^{\Phi-Vx}\nabla \tilde m)=0\quad \text{in}\ \Omega(V), \quad \tilde\phi=\Lambda(V) e^{\Phi-Vx}\tilde m-\tilde\rho\quad\text{on}\ \partial \Omega(V).
\end{equation}
Then one derives, via integration by parts, that the adjoint to  $\mathcal{A}(V)$ with respect to the pairing
\begin{equation}
\label{pairingL2}
\langle (m,\rho), (\tilde m, \tilde \rho)\rangle_{L^2} =\int_{\Omega(V)} m\tilde m\, dx dy +\int_{\partial \Omega(V)} \rho\tilde \rho \, ds 
\end{equation}
 is the  
following operator $\mathcal{A}^\ast(V)$:
\begin{equation}
	\label{myosin_adj_tw}
	\left({\mathcal A}^{\ast}(\tilde m, \tilde \rho)\right)_{\tilde m}=\Delta \tilde m +\nabla\Phi\cdot \nabla \tilde m -V\partial_x \tilde m+\tilde \phi
	\quad\text{in}\ \Omega(V),\quad \text{with} \quad \partial_\nu \tilde m=0\quad\text{on}\ \partial\Omega(V),%\quad \text{or}\quad =\frac{1}{\Lambda  e^{\Phi-Vx}}(\Delta\tilde \phi -\zeta\tilde\phi) +\tilde \phi
\end{equation}
\begin{multline}
	\label{rho_adj_tw}
\left({\mathcal A}^{\ast}(\tilde m, \tilde \rho)\right)_{\tilde\rho}=
	\partial^2_\nu\Phi ( \tilde \rho-\tilde m\Lambda e^{\Phi-Vx})+\partial_\tau\left((\partial_\tau \Phi+V\nu_y)( \tilde \rho-\tilde m\Lambda e^{\Phi-Vx})\right)
	\\
	+(V\nu_x -\frac{\gamma\kappa^2}{\zeta})\partial_\nu\tilde\phi 
	-\frac{\gamma}{\zeta}
	\partial^2_{\tau\tau}\partial_\nu\tilde\phi -\frac{p_\ast^\prime(|\Omega(V)|)}{\zeta}\int_{\partial \Omega(V)}\partial_\nu\tilde\phi\, ds
\quad\text{on} \ \partial\Omega(V).
\end{multline}

Observe that the definition of $\mathcal{A}^\ast(V)$ admits an important simplification. Namely, one can express the action of the operator $\mathcal{A}^\ast(V)$
in terms of the only function $\tilde \phi$: in view of \eqref{pressure_adj_tw} we have 
	\begin{equation}
	\label{myosin_adjBIS_tw}
	\left({\mathcal A}^{\ast}(V)(\tilde m, \tilde \rho)\right)_{\tilde m} =\frac{1}{\Lambda(V)  e^{\Phi-Vx}}(\Delta\tilde \phi -\zeta\tilde\phi) +\tilde \phi,
	\end{equation}
	and, since $\tilde\phi=\Lambda(V) e^{\Phi-Vx}\tilde m-\tilde\rho$ on $\partial \Omega(V)$,  \eqref{rho_adj_tw} rewrites as 
	\begin{multline}
	\label{rho_adjBIS_tw}
	\left({\mathcal A}^{\ast}(\tilde m, \tilde \rho)\right)_{\tilde\rho}=
	-\partial^2_\nu\Phi \tilde\phi-\partial_\tau((\partial_\tau \Phi+V\nu_y)\tilde \phi)
	\\
+(V\nu_x -\frac{\gamma\kappa^2}{\zeta})\partial_\nu\tilde\phi 
	-\frac{\gamma}{\zeta}
	\partial^2_{\tau\tau}\partial_\nu\tilde\phi -\frac{p_\ast^\prime(|\Omega(V)|)}{\zeta}\int_{\partial \Omega(V)}\partial_\nu\tilde\phi\, ds.
	\end{multline}
	Since ${\rm div}(\Lambda e^{\Phi-Vx}\nabla \tilde m)=\Delta\tilde \phi -\zeta\tilde\phi $  in $\Omega(V)$ and $\partial_\nu \tilde m=0$ on 
$\partial\Omega(V)$ the following additional condition 
	\begin{equation}
	\label{ad_cond}
	\int_{\partial \Omega(V)} \partial_\nu\tilde\phi\, ds=\zeta\int_{\Omega(V)}\tilde\phi\,dxdy
	\end{equation}
must be satisfied by $\tilde \phi$. % (as a consequence of \eqref{pressure_adj_tw}),  
Then one can reconstruct $\tilde m$, up to an additive constant, by solving  \eqref{pressure_adj_tw}. 

The following equivalent form of \eqref{rho_adjBIS_tw} is obtained by using the equation and the boundary conditions from \eqref{tw_Liouvtypeeq}-\eqref{addit_cond_tw}  on the boundary,
%will be used in the construction of generalized eigenvector of $\mathcal{A}^\ast(V)$,
\begin{multline}
	\label{rho_adjBIS_BIS_tw}
	\left({\mathcal A}^{\ast}(V)(\tilde m, \tilde \rho)\right)_{\tilde\rho}=
	(\Lambda(V)e^{\Phi-Vx}-\zeta\Phi+\kappa V \nu_x )\tilde\phi+\frac{\gamma \partial_\tau \kappa}{\zeta}\partial_\tau \tilde\phi -V\partial_\tau(\nu_y\tilde \phi)
	\\
+(V\nu_x -\frac{\gamma\kappa^2}{\zeta})\partial_\nu\tilde\phi 
	-\frac{\gamma}{\zeta}
	\partial^2_{\tau\tau}\partial_\nu\tilde\phi -\frac{p_\ast^\prime(|\Omega(V)|)}{\zeta}\int_{\partial \Omega(V)}\partial_\nu\tilde\phi\, ds.
	\end{multline}

While the generalized eigenspace of  $\mathcal{A}(V)$ corresponding to the zero eigenvalue is explicitly given in terms of the solutions  $\phi=\Phi(x,y,V)$, 
$\Omega=\Omega(V)$ 
of the free boundary problem \eqref{tw_Liouvtypeeq}--\eqref{addit_cond_tw}, for the operator $\mathcal{A}^\ast(V)$ we know explicitly only the eigenvector 
\begin{equation}\label{eq:adjoint_eigenvalue}
W_1^\ast=(1,\Lambda(V)e^{\Phi-Vx})
\end{equation}
(which is related to the conservation of the total myosin  mass) while the corresponding generalized eigenvector for $V\not= 0$ has more complicated structure and exhibits singular behavior.  Namely, 
we will show that  if  $\mathcal{A}^\ast(V)(\tilde m,\tilde\rho)=(1,\Lambda(V)e^{\Phi-Vx})$ then $(\tilde m,\tilde\rho)$ blows up as $1/V$ as $V\to 0$. After 
normalizing it is natural to consider the problem in the form $\mathcal{A}^\ast(V)(\tilde m,\tilde\rho)=Vk(1,\Lambda(V)e^{\Phi-Vx})$ with bounded 
$(\tilde m,\tilde\rho)$.

%Observe that for $V=0$, one obtains two eigenvectors for $\mathcal A^\ast_{\rm tw}$ explicitly: 
%\begin{itemize}
%	\item[(i)] $\tilde m=\partial_V \Phi-x$  and $\tilde\rho=m_0\tilde m-\tilde\phi$ (with $\tilde \phi=-(\zeta-m_0)\partial_V\Phi-m_0 x$, $\partial_\nu \tilde \phi=-\zeta \nu_x$);
%	\item[(ii)]$\tilde m=1$  and $\tilde\rho=m_0 $ (with $\tilde \phi=0$).
%\end{itemize}
%For $V\not =0$ the only eigenvector is  $\tilde m=1$  and $\tilde\rho=\tilde\Lambda e^{\Phi-Vx}$. 

We postulate  the ansatz 
\begin{equation}\label{eq:adjoint_operator_generalized_eigenvector_ansatz}
\tilde m=\Phi_V^0-x+V\tilde m_1+\dots,\quad \tilde\rho=V\tilde \rho_1+\dots,
\end{equation}
and substitute it in the equation $\mathcal{A}^\ast(V)(\tilde m,\tilde\rho)=Vk(1,\Lambda(V)e^{\Phi-Vx})$ with unknown for the moment constant $k$. Note that 
the first term $(\Phi_V^0-x,0)$ of the proposed ansatz is an eigenvector of $\mathcal{A}^\ast(0)$ (with $\tilde \phi=-(\zeta-m_0)\Phi_V^0-m_0 x$), while collecting terms of the order $V$ yields (as above we replace 
$\Omega(V)$ by the disk $B_R$ which approximates $\Omega(V)$ to the order $V$) 
\begin{equation}\label{eq:adjoint_operator_generalized_eigenvector_PDE_1}
\Delta\tilde m_1%(\Delta\tilde \phi_1 -\zeta\tilde\phi_1)
+\tilde \phi_1=k-|\nabla (\Phi_V^0-x)|^2 \quad \text{in}\ B_R,\quad \partial_r \tilde m_1=0\quad\text{on}\ \partial B_R,
\end{equation}
%+(\Phi_V^0-x)\Delta (\Phi_V^0-x)
\begin{equation}\label{eq:adjoint_operator_generalized_eigenvector_PDE_2}
\Delta\tilde \phi_1 -\zeta\tilde\phi_1={m_0}\Delta \tilde m_1  +m_0{\rm div}\left((\Phi_V^0-x)\nabla (\Phi_V^0-x)\right)\quad \text{in}\ B_R,
\end{equation}
\begin{equation}\label{eq:adjoint_operator_generalized_eigenvector_BC_1}
\tilde \phi_1=m_0\tilde m_1+m_0 R^2 \cos^2\varphi-\tilde \rho_1 \quad \text{on}\ \partial B_R,
\end{equation}
\begin{equation}\label{eq:adjoint_operator_generalized_eigenvector_BC_2}
-\frac{\gamma}{R^2 \zeta}(\partial_r \tilde\phi_1 +\partial^2_{\varphi\varphi} \partial_r \tilde\phi_1)-\frac{p_\ast^\prime(\pi R^2)}{\zeta}
\int_{-\pi}^\pi \partial_r \tilde\phi_1 R\,d\varphi =m_0(k-\partial^2_{rr} \Phi_V^0  R \cos\varphi-\cos (2\varphi))+\zeta \cos^2 \varphi
\quad \text{on}\ \partial B_R.
\end{equation}
Introducing a solution $f$ of $\Delta f={\rm div}\left((\Phi_V^0-x)\nabla (\Phi_V^0-x)\right)$ in $B_R$, $\partial_r f=0$ on $\partial B_R$, we can rewrite problem 
 \eqref{eq:adjoint_operator_generalized_eigenvector_PDE_1}-\eqref{eq:adjoint_operator_generalized_eigenvector_BC_2} in the operator form:
\begin{multline*}
\mathcal{A}^\ast_{ss}(\tilde m_1+f,\tilde \rho_1-m_0 f -m_0 R^2 \cos^2\varphi)\\
=(k+(\Phi_V^0-x)\Delta (\Phi_V^0-x), 
m_0(k-\partial^2_{rr} \Phi_V^0  R \cos\varphi-\cos (2\varphi))+\zeta \cos^2 \varphi),
\end{multline*}
and since the null space of $\mathcal{A}^\ast_{ss}$ is nonzero, we can use solvability conditions to identify $k$. Indeed, the operator $\mathcal{A}_{ss}$ has the eigenvector  
 $(\gamma/R^2+2\pi R p_\ast^\prime(\pi R^2), 1)$ corresponding to the zero eigenvalue, and we necessarily have
\begin{multline*}
(\gamma/R^2+2\pi R p_\ast^\prime(\pi R^2))\left\{\pi R^2 k -\int_{B_R}  |\nabla  (\Phi_V^0-x) |^2\, dxdy\right\}\\+\int_{-\pi}^{\pi}
\left\{k m_0+(m_0-m_0^2 R^2) \cos^2\varphi+\zeta \cos^2 \varphi\right\} R\,d\varphi=0.
\end{multline*}
After rearranging terms and using \eqref{formula_for_E} this yields
\begin{equation}
\label{karavno}
k=k_0:=m_0 \frac{d E}{d M}\Bigl|_{M=M_0}\Bigr.\int_{B_R} (\Phi_V^0-x)^2dxdy.
\end{equation}
Solving  \eqref{eq:adjoint_operator_generalized_eigenvector_BC_2} we find
\begin{equation}
\label{priblizhBC}
\partial_r\tilde\phi_1(R,\varphi)=A\cos(2\varphi)+B =2A\cos^2\varphi+(B-A),
\end{equation}
where 
\begin{equation}
\label{defAB}
A=\frac{R^2\zeta}{6\gamma}(\zeta-m_0-m_0^2 R^2), \quad B=\frac{\zeta}{2}(2k_0 m_0+m_0-m_0^2 R^2+\zeta)/\left\{-\gamma/R^2-2\pi R p_\ast^\prime(\pi R^2))\right\}. 
\end{equation}
Also, eliminating $\tilde m_1$ from \eqref{eq:adjoint_operator_generalized_eigenvector_PDE_1}-\eqref{eq:adjoint_operator_generalized_eigenvector_PDE_2} we have that $\tilde \phi_1$ satisfies 
\begin{equation}
\label{EscheOdno Uravnen}
\frac{1}{m_0}\left(\Delta\tilde \phi_1 -\zeta\tilde\phi_1\right)+\tilde \phi_1 = k_0 +(\Phi_V^0-x)\Delta (\Phi_V^0-x)
\end{equation}
in $B_R$. The unique solution of this equation with boundary condition \eqref{priblizhBC} is represented as the sum of a radially symmetric function and the product of another radially symmetric function with $\cos(2\varphi)$, therefore it extends as a solution of  \eqref{EscheOdno Uravnen} to the entire  $\mathbb{R}^2$. 
Thus the  function 
\begin{equation}
\label{Priblizhennaya_tilde_phi}
\tilde \phi=-(\zeta-m_0)\Phi_V^0-m_0 x+V\tilde \phi_1
\end{equation}
is well defined on $\Omega(V)$. One can define $\tilde m_1$  by solving \eqref{eq:adjoint_operator_generalized_eigenvector_PDE_1} and then $\tilde \rho_1$ by
\eqref{eq:adjoint_operator_generalized_eigenvector_BC_1}, completing the construction of the
ansatz \eqref{eq:adjoint_operator_generalized_eigenvector_ansatz}. The properties of $\tilde \phi$ needed for the justification of 
\eqref{eq:adjoint_operator_generalized_eigenvector_ansatz} are collected in

\begin{lem}
\label{muchitelnayalemma} The function $\tilde \phi$ given by \eqref{Priblizhennaya_tilde_phi} satisfies for small $V$ 
\begin{equation}
\label{PervoeProp}
\left\|\Delta \tilde \phi -\zeta \tilde\phi +\Lambda(V)e^{\Phi-Vx}\tilde\phi-k_0V\Lambda(V)e^{\Phi-Vx}\right\|_{C^j(\overline\Omega(V))}=O(V^2) \quad \forall j\in\mathbb{Z}_{+},
\end{equation}
\begin{equation}
\label{VtoroeProp}
\left\|\partial_\nu\tilde \phi-\left(-\zeta \nu_x+2AV\nu_x^2+V(B-A)\right)\right\|_{C^j(\partial\Omega(V))}=O(V^2)\quad \forall j\in\mathbb{Z}_{+} ,
\end{equation}
and 
\begin{equation}
\label{TreteProp}
\int_{\partial\Omega(V)} \partial_\nu \tilde \phi  \, ds -\zeta \int_{\Omega(V)}\tilde \phi \, dx dy=O(V^2).
\end{equation}
\end{lem}

\begin{proof} Bound \eqref{PervoeProp} follows from the construction of $\tilde \phi$ and asymptotic representation \eqref{expansion_of_Phi} for $\Phi-Vx$ in cojunction with the formula $\Lambda(V)=\Lambda(0)+O(V^2)$ (see Remark  \ref{rem:boundary}). To verify \eqref{VtoroeProp} one passes to polar 
coordinates and uses \eqref{priblizhBC} together with the bound \eqref{boundsforrho}. Finally, \eqref{TreteProp} follows from the construction of $\tilde\phi$
(recall that $\int_{\partial B_R} \partial_r \tilde \phi  \, ds -\zeta \int_{B_R}\tilde \phi \, dx dy=0$) and \eqref{boundsforrho}.
\end{proof}

\section{Small velocity asymptotic formulas for eigenvalues of the operator  linearized  around traveling wave solutions}
\label{section_lin_stab_tw_justify}

In this Section we justify asymptotic expansions constructed in Section \ref{section_lin_stab_tw_formal}. We begin with the generalized eigenvector of the adjoint 
operator $\mathcal{A}^\ast(V)$. Recall that $\mathcal{A}^\ast(V)$ has the eigenvector $W^\ast_1=(1,\Lambda(V)e^{\Phi-Vx})$ that 
is related to the total mayosin mass conservation property in  problem \eqref{actflow_in_terms_of_phi}--\eqref{myosin_bc_I}.

\begin{lem} \label{generLemma} The operator $\mathcal{A}^\ast(V)$ has a generalized eigenvector 
$W^\ast_2=(m^\ast_2,\rho^\ast_2)$,  $\mathcal{A}^\ast(V) W^\ast_2 
=W^\ast_1$, whose first component expands when $V\to 0$  as follows,
 \begin{equation}\label{singular}
	m_2^\ast=\frac {1}{k_0V+V^2k_1(V)}( \Phi_{V}^0-x+V\tilde m_1)+V f
	\end{equation}
with bounded $k_1(V)$ and $f(x,y,V)$ (in $C^j(\Omega(V))$ $\forall j\in \mathbb{Z}_+$),
while $\|\rho^\ast_2\|_{C^j(\Omega(V))}=O(1)$  $\forall j\in \mathbb{Z}_+$. The constant $k_0$ in \eqref{singular}
is given by \eqref{karavno}, $\Phi_{V}^0=\lim_{V\to 0}\partial_V\Phi$, and $m_1$ is a smooth function independent of  $V$ (and defined on $\mathbb{R}^2$).
\end{lem}

% To obtain spectral properties of the linearized operator, we need to study the a

%The difference between the weakly nonlinear regime and the linear regime in Section \ref{section_lin_stab_tw} can be explained as follows. 

%\bigskip		
%\begin{center}
%Adjoint operator and its spectrum
%\end{center}
%
%\bigskip

\begin{proof}
	 Consider the problem of finding generalized eigenvector in the form $\mathcal{A}^\ast(V)(\tilde m,\tilde \rho)=k(1,\Lambda(V) e^{\Phi -Vx})$ (with constant $k\not=0$), then
\begin{equation}
\label{myosin_adjBIS_tw_generalized}
\Delta\tilde \phi -\zeta\tilde\phi +\Lambda  e^{\Phi-Vx}\tilde \phi=k \Lambda  e^{\Phi-Vx}\quad \text{in}\ \Omega(V).
\end{equation}
Allowing $k=0$, which corresponds to the case of  zero eigenvalue, we can choose constant $k$ to satisfy condition \eqref{ad_cond}. 
Then the problem  
of finding generalized eigenvector reduces to the equation
\begin{equation}
\label{rho_adjBIS_tw_generalized}
-\partial^2_{\nu\nu}\Phi \tilde\phi-\partial_\tau((\partial_\tau \Phi+V\nu_y)\tilde \phi)
+(V\nu_x -\frac{\gamma\kappa^2}{\zeta})\partial_\nu\tilde\phi 
	-\frac{\gamma}{\zeta}
	\partial^2_{\tau\tau}\partial_\nu\tilde\phi -\frac{p_\ast^\prime(|\Omega(V)|)}{\zeta}\int_{\partial \Omega(V)}\partial_\nu\tilde\phi\, ds=k\Lambda e^{\Phi-Vx}
\end{equation}
for the only unknow function $\partial_\nu\tilde\phi$. Indeed, observe that  $k$ and  $\tilde\phi$  are given in terms of $\partial_\nu \tilde\phi$ by 
\begin{equation}
\label{kandphi_via_partial_nu_phi}
k=\frac{\int_{\partial \Omega}\partial_\nu \tilde \phi (1+\psi_1) \, ds} {\Lambda \int_ {\Omega} e^{\Phi-Vx}\psi_1\, dx dy}, \quad \tilde \phi =\tilde\psi+k\psi_2,
\end{equation}
where $\tilde \psi$, $\tilde \psi_1$ and $\tilde \psi_2$ are solutions of the equations 
\begin{equation}
\Delta\tilde \psi -\zeta\tilde\psi +\Lambda  e^{\Phi-Vx}\tilde \psi=0 \quad \text{in}\ \Omega(V)
\end{equation}
and
\begin{equation}
\Delta\psi_1 -\zeta\psi_1 +\Lambda  e^{\Phi-Vx} \psi_1=\zeta \quad \text{in}\ \Omega(V), \quad \Delta \psi_2 -\zeta\psi_2 +\Lambda  e^{\Phi-Vx} \psi_2=\Lambda  e^{\Phi-Vx}\quad \text{in}\ \Omega(V),
\end{equation}
subject to the Neumann boundary conditions 
$$
\partial_\nu \tilde \psi = v \quad\text{on}\  \partial\Omega(V),
$$
with $v=\partial_\nu\tilde\phi$, and $\partial_\nu \psi_i =0$, $i=1,2$, on $\partial\Omega(V)$. Thus \eqref{rho_adjBIS_tw_generalized} writes as
\begin{equation}
\label{RHO_adjBIS_tw_generalized}
-\frac{\gamma}{\zeta}
	\partial^2_{\tau\tau}v+(V\nu_x -\frac{\gamma\kappa^2}{\zeta})v
	-\partial^2_{\nu\nu}\Phi \tilde\psi-\partial_\tau((\partial_\tau \Phi+V\nu_y)\tilde \psi)+
\int_{\partial \Omega} Q(s,\tilde s, V)v(\tilde s)\, d\tilde s=0,
\end{equation}
%Q_2(x)
%\int_{\partial \Omega_{\rm tw}}v(1+\psi_1)\, ds
%-\frac{p_\ast^\prime(|\Omega_{\rm tw}|)}{\zeta}\int_{\partial \Omega_{\rm tw}}\partial_\nu\tilde\phi\, ds-k(\Lambda e^{\phi-Vx}+\partial^2_\nu\Phi \psi_2-((\partial_\tau \Phi+V\nu_y)\psi_2)^\prime=0
%\end{equation}
where $Q$ is a smooth function and $Q=-p_\ast^\prime(\pi R^2)/\zeta-m_0/(\pi R^2 \zeta)+O(V)$ as $|V|\to 0$. Observe that \eqref{RHO_adjBIS_tw_generalized} is
a small perturbation of the equation 
$$
-\frac{\gamma}{\zeta}
	\partial^2_{\tau\tau}v-\frac{\gamma }{\zeta R^2}v-
\int_{\partial B_R} (p_\ast^\prime(\pi R^2)/\zeta+m_0/(\pi R^2 \zeta)) v(\tilde s)\, d\tilde s=0.
$$
In the case $p_\ast^\prime(\pi R^2)/\zeta+m_0/(\pi R^2 \zeta)+\gamma /(2\pi R^3\zeta)\not =0$ the latter equation has the only (even) 
solution $\cos \frac{\pi} {R} s$. On the other 
hand, since the multiplicity of zero eigenvalue of the operator ${\mathcal A}(V)$ is at least two ($\forall V$), and the same holds for ${\mathcal A}^\ast(V)$,  the equation \eqref{RHO_adjBIS_tw_generalized} always has at least one solution. 

Rewriting \eqref{RHO_adjBIS_tw_generalized} in the operator form $\tilde{\mathcal{L}}(V)v=0$ in $L^2(\partial\Omega(V))$, we have operator 
 $\tilde{\mathcal{L}}(V)$ with simple isolated eigenvalue $\lambda=0$. Then one can show that for some $\delta>0$ the norms $\| (\lambda -\tilde{\mathcal{L}}(V))^{-1}\|$  are uniformly bounded for complex $\lambda$ with $|\lambda|=\delta$ and sufficiently small $|V|$. Therefore, if $\tilde v$ is an approximation of the eigenfunction, we have
$$
\Pi_0\tilde v-\tilde v =\frac{1}{2\pi{\rm i}}\oint_{|\lambda|=\delta}(\lambda- \tilde{\mathcal{L}}(V))^{-1}\tilde{\mathcal{L}}(V)\tilde v \frac{ d\lambda}{\lambda}
$$
(to see this one takes integral of the identity $\frac{1}{\lambda}\tilde v= (\lambda -\tilde{\mathcal{L}}(V))^{-1}\tilde v-\frac{1}{\lambda}(\lambda -\tilde{\mathcal{L}}(V))^{-1}\tilde{\mathcal{L}}(V)\tilde v$), where $\Pi_0$ denotes the spectral projector on the null space of $\tilde{\mathcal{L}}(V)$. Thus
 $$
\|\Pi_0\tilde v-\tilde v \|_{L^2(\partial\Omega(V))}\leq C\|\tilde{\mathcal{L}}(V)\tilde v\|_{L^2(\partial\Omega(V))}
$$
and in a standard way, via bootstrapping, this bound yields  $\|\Pi_0\tilde v-\tilde v \|_{H^2(\partial\Omega(V))}\leq C\|\tilde{\mathcal{L}}(V)\tilde v\|_{L^2(\partial\Omega(V))}$.

Now consider $\tilde v:=-\zeta \nu_x+2AV\nu_x^2+V(B-A)$ (see \eqref{VtoroeProp}). Introducing the pair $(w,k(w))$ that solves 
\begin{equation*}
\Delta w-\zeta w+\Lambda  e^{\Phi-Vx}w=k(w) \Lambda  e^{\Phi-Vx}\quad \text{in}\ \Omega(V),\quad \partial_\nu w=\tilde v \ \text{on}\ \partial \Omega(V),
\end{equation*}
with the additional condition 
$
\int_{\partial \Omega(V)} \partial_\nu w\, ds=\zeta\int_{\Omega(V)} w \,dxdy
$, 
we get by virtue of  Lemma \ref{muchitelnayalemma} that $\|w-\tilde\phi\|_{C^j(\overline\Omega(V))}=O(V^2)$ $\forall j\in\mathbb{Z}_+$, $k(w)=k_0 V+O(V^2)$, 
where  $\tilde\phi$, $k_0$ are given by \eqref{Priblizhennaya_tilde_phi} and  \eqref{karavno}. Direct calculations show that   
$$
\tilde{\mathcal{L}}(V)\tilde v=O(V^2).
$$
Indeed, observe that $\kappa=\frac{1}{R}+O(V^2)$, $w=-m_0 R\nu_x+O(V)$ and $\partial_\tau (\nu_y w)= m_0 (1-2\nu_x^2)+O(V)$ on $\partial\Omega(V)$,
$$
\partial_{\tau\tau} \nu_x= -\nu_y\partial_\tau \kappa-\kappa^2\nu_x=-\frac{1}{R^2}\nu_x +O(V^2),
$$
$$ 
\partial_{\tau\tau} \nu_x^2= -4\kappa^2 \nu_x
+2\kappa^2-2\nu_x\nu_y \partial_\tau \kappa=\frac{2}{R^2}(1-2\nu_x^2) +O(V^2),
$$ 
also $-\partial^2_{\nu\nu}\Phi-\partial^2_{\tau\tau}\Phi=\kappa V \nu_x+\Lambda(V)e^{\Phi-Vx}-\zeta\Phi=\frac{1}{R} V\nu_x-m_0 V R\nu_x+O(V^2)$ (cf .\eqref{rho_adjBIS_BIS_tw}), then
%\begin{multline*}
%\tilde{\mathcal{L}}\tilde v=\left\{(m_0^2 R^2-m_0)\nu_x^2-m_0(1-2\nu_x^2)-\zeta\nu_x^2+6 A\gamma\nu_x^2/(\zeta R^2)-4\gamma A/(\zeta R^2)-V(B-A)\gamma/(\zeta R^2)\right\}V
%\\-2\pi R V\frac{p_\ast^\prime(\pi R^2)}{\zeta}B-m_0k_0V+O(V^2)
%\end{multline*}
\begin{multline}
\label{vychisleniya}
\tilde{\mathcal{L}}\tilde v=\left\{(m_0^2 R^2-m_0)+2m_0 -\zeta+\frac{6\gamma A}{\zeta R^2}\right\}V\nu_x^2
 \\
-\left\{m_0+\frac{3\gamma A}{\zeta R^2}+\frac{\gamma B}{\zeta R^2}-2\pi R \frac{p_\ast^\prime(\pi R^2)}{\zeta}B-m_0k_0
\right\}
V+O(V^2).
\end{multline}
Both the coefficient in front of $V\nu_x^2$ and the coefficient in front of $V$ in  \eqref{vychisleniya} vanish by virtue 
of formulas  \eqref{defAB} for constants $A$ and $B$. Thus we have $\| v-\tilde v \|_{H^2(\partial\Omega(V))}\leq C V^2$ for a properly normalized solution 
$v$ of \eqref{RHO_adjBIS_tw_generalized}. Finally, retrieving first  the number $k$ and the auxiliary function $\tilde \phi$ via \eqref{kandphi_via_partial_nu_phi}
for $v$ and $\tilde v$, then reconstructing $W_2^\ast$ and its approximation corresponding to $\tilde v$ one comletes the proof of Lemma \ref{generLemma} (details are left to the reader).
\end{proof}

Now we prove the key asymptotic formula 
\begin{equation}
\label{sovsemfinal_formula_for_lambda}
 \lambda(V)=-\frac{d E}{d M}\Bigl|_{M=M_0}\Bigr. V {M^\prime(V)}(1+O(V)) \quad \text{as}\ V\to 0, 
\end{equation}
where $M_0$ is the  critical value of the total myosin mass   of the stationary solution (corresponding to the critical  radius $R=R_0$).
%  when this solution bifurcates to traveling waves.   

\begin{thm}
\label{GlavnayaTh}
 Assume that  conditions of Theorem \ref{biftwtheorem} are satisfied and also that $M^\prime(V)\not=0$ for sufficiently small $|V|\not=0$. Then  the spectrum of the linearized operator $\mathcal{A}(V)$ (around the traveling wave solution) has the following structure near zero:
$\mathcal{A}(V)$ has  a small eigenvalue $\lambda(V)$ given by the asymptotic formula  
  \eqref{sovsemfinal_formula_for_lambda} in addition to the zero eigenvalue with multiplicity two whose eigenvector is given by 
\eqref{tw_shifts1}
 and the generalized eigenvector is given by
\eqref{TwderivativeinV1}.
\end{thm}
\begin{rem} %The condition $M^\prime(V)\not=0$ means  that  different velocities   small nonzero velocities $V$  of the traveling waves  correspond to the different  values of the total myosin mass $M$. 
In generic case  (for almost all  values of the parameters $p_{\rm h}$, $k_{\rm e}$, $\zeta$, and $\gamma$) 
$M^{\prime \prime}(0) \not=0$.  Then   $M^\prime(V)\not=0$  is satisfied for small $V\not=0$ 
and formula  \eqref{sovsemfinal_formula_for_lambda} can be simplified as follows 
\begin{equation}
\label{formula_for_lambda_second}
\lambda(V)=-V^2\frac{d E}{d M}\Bigl|_{M=M_0}\Bigr.  M^{\prime \prime}(0)+O(V^3) \quad \text{as}\ V\to 0.
\end{equation}\label{two_formulas}	
	  \end{rem}	
\begin{rem} In Theorem \ref{GlavnayaTh} we  tacitly assume that operator $\mathcal{A}(V)$ is restricted to the subspace of  vectors that are symmetric with respect to the $x$-axis, while the general case without any symmetry restrictions on eigenvectors and generalized eigenvectors is considered in Section \ref{SectionNoSymmetry}, see Theorem~\ref{ThmNoSymmetry}.
\end{rem}
\begin{proof}
  Let $W_2^\ast$ be a generalized eigenvector of $\mathcal{A}^\ast
(V)$ corresponding to the eigenvector $W_1^\ast=(1,\Lambda
e^{\Phi-Vx})$, $\mathcal{A}^\ast (V)W_2^\ast=W_1^\ast$. The space
$L^2(\Omega(V))\times L^2(\partial \Omega(V))$ decomposes into the
direct sum of invariant subspaces
\begin{equation}
\label{invarsubspace} \mathcal{I}(V)=\{W\in L^2(\Omega(V))\times
L^2(\partial \Omega(V)); \langle W,W_1^\ast\rangle=\langle
W,W_1^\ast\rangle=0\}\oplus {\rm span}\{W_1, W_2\}
\end{equation}
of the operator $\mathcal{A}(V)$, where $W_1$, $W_2$ denote the
pair of the eigenvector of $\mathcal{A}(V)$ corresponding to the
zero eigenvalue and a generalized eigenvector. This induces also
the decomposition of the domain $D(\mathcal{A}
(V))=H^2(\Omega(V))\times H^{3}(\partial\Omega(V))$ into the sum
$D(\mathcal{A} (V))=D(\mathcal{A}(V))\cap \mathcal{I}(V)\oplus
{\rm span}\{W_1, W_2\}$.

%\textcolor{red}
%{
%Let us redefine $\mathcal{A}(V)$ on ${\rm span}\{W_1, W_2\}$ by
%the identity operator, the resulting operator will be denoted by
%$\tilde {\mathcal{A}} (V)$.
%}

Fix a sufficiently small $\delta>0$
such that $\mathcal{A}_{ss}$ does not have eigenvalues $\lambda$
with $0<|\lambda|\leq 2 \delta$. Then we claim that for sufficiently
small $V$ the operator $(\lambda- {\mathcal{A}} (V))^{-1}$
exists and is uniformly bounded on $\delta/2\leq |\lambda|\leq
2\delta$. Indeed, assume by contradiction that for a sequence
$V_j\to 0$ $\exists W_j\in D(\mathcal{A} (V_j))\cap
\mathcal{I}(V_j)$, $W_j=(m_j,\rho_j)$, with
$\|m_j\|_{L^2(\Omega(V_j))}^2+\|\rho_j\|_{L^2(\partial\Omega(V_j))}^2=1$,
such that norms of $U_j=(\lambda_j- {\mathcal{A}} (V_j))W_j$
in $L^2(\Omega(V_j))\times L^2(\partial \Omega(V_j))$ tend to zero as
$j\to \infty$. We use the following  lemma which provides a priori estimates 
implying  that norms
$\|m_j\|_{H^2(\Omega(V_j))}$ and
$\|\rho_j\|_{H^{3}_{\partial\Omega(V_j)}}$ are uniformly
bounded.
% The next Lemma establishes a priori estimates  for $\rho$ and $m$.
\begin{lem} 
\label{VspomogLemma}
There exists $K=K(\overline V)>0$ such that for all $V$ with $|V|<\overline V$  every pair
$(m,\rho)$ solving
\begin{equation}
	\label{apriori_Hact_flow_lin_op}
	\Delta \phi + m=\zeta \phi \quad\text{in}\ \Omega(V),
	\end{equation}
	\begin{equation}
	\label{apriori_Hact_flow_BC_dir_op}
	\zeta(\phi+V\nu_x \rho) =p_{\ast}^\prime(|\Omega(V)|)\int_{\partial\Omega(V)}
	\rho(s)ds
	+\gamma(\rho^{\prime\prime}+\kappa^2 \rho)\quad\text{on}\ \partial\Omega(V),
	\end{equation}
	\begin{equation}
	\label{apriori_Hact_flow_sec_BC}
	\varrho+K\rho=\frac{\partial \phi}{\partial \nu}+\frac{\partial^2 \Phi}{\partial \nu^2}\rho-\left(\frac{\partial \Phi}{\partial \tau}+V \nu_y\right)\rho^\prime\quad\text{on}\ \partial\Omega(V),
	\end{equation}
\begin{equation}
	\label{apriori_Hmyosin_lin_eq}
	f+K m=\Delta {m}+V\partial_x m -\div( \tilde \Lambda e^{\Phi-Vx} \nabla {\phi})
	-\div(m\nabla \Phi) \quad\text{in}\ \Omega(V).
	\end{equation}
\begin{equation}
	\label{apriori_Hmyosin_lin_BC}
	\partial_\nu m+\tilde\Lambda e^{\Phi-Vx}\left(\frac{\partial^2 \Phi}{\partial \nu^2}\rho -\Bigl(\frac{\partial \Phi}{\partial \tau}+V \nu_y\Bigr)\rho^\prime\right)=0
	\quad \text{on}\ \partial \Omega(V).
\end{equation}
satisfies the bound
\begin{equation}
\label{apriori_est}
\|\rho\|_{H^{3}(\partial\Omega(V))}+\|m\|_{H^2(\Omega(V))}\leq C(\|\varrho\|_{L^2(\Omega(V))}+\|f\|_{L^2(\Omega(V))})
\end{equation}
\end{lem}
\begin{proof} Without loss of generality we can assume that $\rho$ and $m$ are sufficiently smooth. Also, for brevity we suppress  hereafter the 
dependence of the domain $\Omega$ on $V$.

The key a priori bound is obtained  multiplying the equation \eqref{apriori_Hact_flow_lin_op} by the harmonic extension $\mathcal{H}(\rho)$ of $\rho$ from $\partial \Omega$ into $\Omega$ 
($\Delta  \mathcal{H}(\rho)=0$ in $\Omega$, and $\mathcal{H}(\rho)=\rho$ on $\partial \Omega$) which yields, after integrating by parts  twice and eliminating 
$\phi$, $\partial_\nu \phi$ from the integrals over the boundary  with the help  
of \eqref{apriori_Hact_flow_BC_dir_op} and \eqref{apriori_Hact_flow_sec_BC},
\begin{multline}
\label{Raz}
\frac{K}{2}\int_{\partial \Omega} \rho^2 \, ds-\frac{\gamma}{\zeta}\int_{\partial \Omega} \rho^{\prime\prime} \partial_\nu \mathcal{H}(\rho)\, ds=-\frac{K}{2}\int_{\partial \Omega} \rho^2 \, ds+\int (\zeta \phi -m)\mathcal{H}(\rho)\, dx dy \\
+ \int_{\partial \Omega} \left(\Bigl(\frac{\gamma\kappa^2}{\zeta}-V\nu_x\Bigr)\rho\partial_\nu  \mathcal{H}(\rho)
+\frac{\partial^2 \Phi}{\partial \nu^2}\rho^2-\left(\frac{\partial \Phi}{\partial \tau}+V \nu_y\right)\rho^\prime\rho-\rho\varrho \right)ds.
\end{multline}
Next observe that the left hand side of \eqref{Raz} represents (square of) a norm in $H^{3/2}(\partial \Omega)$ when $K>0$ is big enough. Actually, the second term solely defines a seminorm in $H^{3/2}(\partial \Omega_{\rm tw})$ if $\kappa\geq 0$. Indeed, using the Frenet-Serret formulas $\partial_\tau \nu_x=\kappa \tau_x=-\kappa \nu_y$,
$\partial_\tau \nu_y=\kappa \tau_y=\kappa \nu_x$ and the fact that $\Delta \mathcal{H}(\rho)=0$ we find
\begin{equation}
\begin{aligned}
\label{identity1}
-\rho^{\prime\prime} \partial_\nu \mathcal{H}(\rho)=-\partial_\tau(\partial_\tau  \mathcal{H}(\rho)) \partial_\nu \mathcal{H}(\rho)=&\kappa(\partial_\nu \mathcal{H}(\rho))^2\\
&+
\partial_\nu \mathcal{H}(\rho)\left(\nu_x^2\partial_{xx}^2 \mathcal{H}(\rho)+2\nu_x\nu_y\partial_{xy}^2 \mathcal{H}(\rho)+\nu_y^2\partial_{yy}^2 \mathcal{H}(\rho)\right),
\end{aligned}
\end{equation}
\begin{equation}
\begin{aligned}
\label{identity2}
\partial_\tau  \mathcal{H}(\rho)) \partial_\tau(\partial_\nu \mathcal{H}(\rho))=&\kappa(\partial_\tau \mathcal{H}(\rho))^2+\nabla  \mathcal{H}(\rho)\cdot\partial_\nu \nabla  \mathcal{H}(\rho)
\\
&-\partial_\nu \mathcal{H}(\rho)\left(\nu_x^2\partial_{xx}^2 \mathcal{H}(\rho)+2\nu_x\nu_y\partial_{xy}^2 \mathcal{H}(\rho)+\nu_y^2\partial_{yy}^2 \mathcal{H}(\rho)\right).
\end{aligned}
\end{equation}
Then taking the half-sum of these identities and integrating over $\partial\Omega$ we obtain, using integration by parts  and the fact that  $\Delta \mathcal{H}(\rho)=0$,
\begin{equation}
\label{identity3}
\begin{aligned}
-\int_{\partial\Omega}\rho^{\prime\prime} \partial_\nu \mathcal{H}(\rho) \, ds=&\frac{1}{2}\int_{\partial\Omega}\kappa |\nabla  \mathcal{H}(\rho)|^2\, ds
+\frac{1}{2}\int_{\partial\Omega}\nabla  \mathcal{H}(\rho)\cdot\partial_\nu \nabla  \mathcal{H}(\rho)\, ds\\ &=
\frac{1}{2}\int_{\partial\Omega}\kappa |\nabla  \mathcal{H}(\rho)|^2\, ds
+\frac{1}{2}\int_{\Omega}|\nabla^2  \mathcal{H}(\rho)|^2 \, dx dy.
\end{aligned}
\end{equation}
Thus \eqref{Raz} yields the following bound
\begin{multline}
\label{Dva}
\|\rho\|_{H^{3/2}(\partial \Omega)}^2\leq -\theta K\|\rho\|_{L^2(\partial\Omega)}^2\\
+C_1\left( \|\varrho\|_{L^2(\partial \Omega)}^2+\left(\frac{1}\ve+1\right)\|\rho\|_{L^2(\partial\Omega)}^2 +\ve \|\phi\|_{L^2(\Omega)}^2+
\|m\|_{L^2( \Omega)}^2\right)
\end{multline}
where $\theta>0$ is independent of $K$, while $\varepsilon>0$ is an arbitrary number (and $C_1$ does not depend on $\ve$).

To derive a bound for $L^2$-norm of $\phi$ represent this function as $\phi =\frac{\gamma}{\zeta}\mathcal{H}(\rho^{\prime\prime})+G$, where $G$ is the solution of
\begin{equation}
\label{Dlinnoe_no_prostoe}
\Delta G=\zeta G+\gamma \mathcal{H}(\rho^{\prime\prime})-m
\end{equation}
\begin{equation}
\label{Dlinnoe_no_prostoe_BC}
\zeta (G +V\nu_x \rho) =p_{\ast}^\prime(|\Omega_{\rm tw}|)\int_{\partial \Omega(V)}
	\rho(s)ds
	+\gamma \kappa^2 \rho\quad\text{on}\ \partial \Omega(V).
\end{equation}
Assume for a moment that a bound for $\|\mathcal{H}(\rho^{\prime\prime})\|_{L^2}$ is known, then by elliptic estimates we have
\begin{equation}
\label{otsenkaAPRIoR}
\|G\|_{L^2(\Omega)}\leq C(\|\rho\|_{H^1(\partial \Omega)}+\|\mathcal{H}(\rho^{\prime\prime})\|_{L^2(\Omega)}+\|m\|_{L^2(\Omega)}).
\end{equation}
%A bound for $G$ in terms of $\|\rho\|_{H^{1/2}}$, $\|m\|_{L^2}$ and $\|H(\rho^{\prime\prime})\|_{L^2}$ is immediate. 
We proceed with derivation of a bound for
$\|\mathcal{H}(\rho^{\prime\prime})\|_{L^2(\Omega)}$. To this end consider the solution of the Dirichlet problem $\Delta g=\mathcal{H}(\rho^{\prime\prime})$ in $\Omega$, $g=0$ on $\partial \Omega$, along with the functions $\mathcal{H}(\partial_\nu g)$, $\mathcal{H}(\rho^{\prime})$ and its harmonic conjugate $\mathcal{H}^\ast(\rho^{\prime})$ (such that $\partial_\nu \mathcal{H}^\ast(\rho^{\prime})=-\partial_\tau \mathcal{H}(\rho^{\prime})=-\rho^{\prime\prime}$). We have 
\begin{multline}
\label{pokostyam}
\int_{\Omega} |\mathcal{H}(\rho^{\prime\prime})|^2\, dxdy=\int_{\Omega} \mathcal{H}(\rho^{\prime\prime})\Delta g\, dxdy\\= \int_{\partial \Omega} \rho^{\prime\prime}\partial_\nu g\, ds=- \int_{\partial \Omega}  \partial_\nu \mathcal{H}^\ast(\rho^{\prime}) \mathcal{H}(\partial_\nu g)\, ds=-
\int_{\Omega}\nabla \mathcal{H}^\ast(\rho^{\prime})\cdot \nabla \mathcal{H}(\partial_\nu g)\,dxdy,
\end{multline}
while by elliptic estimates
$$
\int_{\Omega} |\nabla \mathcal{H}(\partial_\nu g)|^2\, dxdy\leq C_2 \| \partial_\nu g\|_{H^{1/2}(\partial \Omega)}\leq C_3 \|g\|_{H^2(\Omega)}\leq C_4 \int_{\Omega} |\mathcal{H}(\rho^{\prime\prime})|^2\, dxdy,
$$
$$
\int_{\Omega} |\nabla \mathcal{H}^\ast(\rho^{\prime})|^2\, dxdy=\int_{\Omega} |\nabla \mathcal{H}(\rho^{\prime})|^2\, dxdy\leq  C_2 \| \rho^\prime\|_{H^{1/2}(\partial \Omega)}
\leq  C_5  \| \rho\|_{H^{3/2}(\partial \Omega)}.
$$
Thus $\|\mathcal{H}(\rho^{\prime\prime})\|_{L^2(\Omega)}\leq C \| \rho\|_{H^{3/2}(\partial \Omega)}$,  and in view of  \eqref{otsenkaAPRIoR} we have
\begin{equation}
\label{nakonetsotsenili_phi}
\|\phi\|_{L^2(\Omega)}^2\leq C_6 \left(\|\rho\|_{H^{3/2}(\partial \Omega)}^2+\|m\|_{L^2(\Omega)}^2\right).
\end{equation}
Now choose $\ve:=\frac{1}{2C_1 C_6}$ in \eqref{Dva}, then for $K\geq K_1=\frac{C_1}{\theta}(\frac{1}{\ve}+1)$ the following bounds hold,
\begin{equation}
\label{nakonetsotsenili_phi_rho}
\|\rho\|_{H^{3/2}(\partial \Omega)}^2 \leq C_7  \left(\|\varrho\|_{L^2(\partial \Omega)}^2+\|m\|_{L^2(\Omega)}^2\right),
\quad 
\|\phi\|_{L^2(\Omega)}^2\leq C_8 \left(\|\varrho\|_{L^2(\partial \Omega)}^2+\|m\|_{L^2(\Omega)}^2\right).
\end{equation}

It remains to find a bound for $m$. To this end multiply \eqref{apriori_Hmyosin_lin_eq} by $m$ and integrate over $\Omega$.  Using \eqref{apriori_Hmyosin_lin_BC}, \eqref{apriori_Hact_flow_lin_op} and the fact that $\partial_\nu(\Phi-Vx) =0$ on $\partial \Omega$, we find
\begin{multline}
	\label{VychislenIYA}
	K \int_\Omega m^2 \, dx dy+\int_\Omega |\nabla m|^2 \, dx dy=-\int_\Omega fm \, dx dy+\Lambda\int_\Omega \phi \div(m \nabla   e^{\Phi-Vx})\, dx dy \\
+\int_\Omega\left(V\partial_x m-\nabla m\cdot \nabla \Phi -m\Delta\Phi + \Lambda e^{\Phi-Vx} (m-\zeta\phi)\right)m\, dxdy\\
	 -\int_{\partial \Omega}
\left\{\tilde\Lambda e^{\Phi-Vx}\left(\frac{\partial^2 \Phi}{\partial \nu^2}\rho -\Bigl(\frac{\partial \Phi}{\partial \tau}+V \nu_y\Bigr)\rho^\prime\right)\right\} m\, ds.
	\end{multline}
We can estimate the right hand side of \eqref{VychislenIYA} with the help of the Cauchy–Schwarz inequality, bounds \eqref{nakonetsotsenili_phi_rho}, 
and the inequality for traces $  \int_{\partial \Omega} |m|^2 ds\leq C \int_\Omega( |\nabla m|^2 +m^2) \, dx dy$, as the result we get
\begin{multline}
\label{nu_vse}
K \int_\Omega m^2 \, dx dy+\frac{1}{2}\int_\Omega |\nabla m|^2 \, dx dy \leq  C_9\left( \int_{\Omega} |m|^2 \, dxdy+\int_{\Omega} |f|^2 \, dxdy+\|\varrho\|_{L^2(\partial \Omega)}^2 \right).
\end{multline}
Thus for $K\geq C_9$ (and $K\geq K_1$) we have obtained a bound for $H^1$-norm of $m$ in terms of $L^2$-norms of $\varrho$ and $f$, this in turn yields bounds 
for $\|\rho\|_{H^{3/2}(\partial \Omega)}$ and $\|\phi\|_{L^2(\Omega)}$. Therefore we can obtain  bounds for the norm of $\phi$ in $H^{3/2}(\Omega)$ and for the norm of $m$ in 
$H^2(\Omega)$ via elliptic estimates applied to problems \eqref{apriori_Hact_flow_lin_op}, \eqref{apriori_Hact_flow_sec_BC} and 
\eqref{apriori_Hmyosin_lin_eq}-\eqref{apriori_Hmyosin_lin_BC}. Finally, since we have a bound for $\phi$ in $H^{1/2}(\partial \Omega)$ (which follows from the bound for $\phi$ in $H^{3/2}(\Omega)$) equation \eqref{apriori_Hact_flow_BC_dir_op} yields a bound for
 $\|\rho\|_{H^3(\partial\Omega)}$. Lemma \ref{VspomogLemma} is proved.

%\textcolor{blue}{After some manipulation and interpolatation inequality (Mironescu-Brezis), the proof of Lemma \ref{VspomogLemma} is complete ...} 
%\begin{equation}
%	\label{apriori_Hmyosin_lin_BC}
%	\partial_\nu m+\tilde\Lambda e^{\Phi-Vx}\left(\frac{\partial^2 \Phi}{\partial \nu^2}\rho -\Bigl(\frac{\partial \Phi}{\partial \tau}+V \nu_y\Bigr)\rho^\prime\right)=0
%	\quad \text{on}\ \partial \Omega_{\rm tw}.
%\end{equation}
\end{proof}

{\it Proof of  Theorem  \ref{GlavnayaTh} (continued).} Writing the equation $U_j=({\mathcal{A}-\lambda_j} (V_j))W_j$  as 
${\mathcal{A}} (V_j))W_j +KW_j =(\lambda_j+K)W_j +U_j$ and applying Lemma \ref{VspomogLemma} we see that norms $\|m_j\|_{H^2(\Omega(V))}$ and 
$\|\rho_j\|_{H^{3}(\partial\Omega(V))}$ are uniformly bounded. Therefore there exists $\lambda$ with $\delta/2\leq|\lambda|\leq 2\delta$, a function $\phi\in H^{3/2}(B_R)$ and nontrivial pair $(m,\rho)\in H^2(B_R) \times H^{5/2}(\partial B_R)$ such that, up to a subsequence, $\lambda_j\to \lambda$, $\left(\rho_j\left(Rs/{L(V_j)}\right),\phi_j\left(x\left({R}s/{L(V_j)}\right), y\left({R}s/{L(V_j)}\right)\right)\right)\to (\rho(s),\phi(x(s), y(s))$ weakly in $H^3(\partial B_R)\times H^{1}(\partial B_R)$ (where $L(V_j)$ denotes the length  of $\partial \Omega(V_j)$) and 
$m_j\to m$, $\phi_j\to \phi$ strongly in $H^1$ on every compact subset of $B_R$. Then passing  to the limit in (variational fomulation of)  problem \eqref{Hmyosin_lin_eq}-\eqref{Hmyosin_lin_BC}, with a smooth test function  $v(x,y)$, we find 
\begin{equation}
	\lambda \int m v dx dy =-\int \nabla m\cdot \nabla v dxdy+m_0 \int (m-\zeta \phi)v dx dy,
\end{equation}
where we have used \eqref{Hact_flow_lin_op} to eliminate $\Delta \phi_j$. Thus $m\in H^2(B_R)$ and $m$ satisfies $\lambda m= \Delta m +m_0(m-\zeta \phi)$ in $B_R$ along with the boundary condition
$\partial_\nu m=0$ on $\partial B_R$.  Passing to the limit in \eqref{Hact_flow_lin_op} with test functions from $C_0^\infty(B_R)$ we get $\Delta \phi=\zeta\phi-m$
in $B_R$, thus the equation for $m$ rewrites as $\lambda m= \Delta m -m_0\Delta \phi$ in $B_R$. Also, taking limit in \eqref{Hact_flow_BC_dir_op} yields 
$\zeta\phi=p_{\ast}^\prime(|B_R|)\int_{\partial B_R}
	\rho(s)ds
	+\gamma(\rho^{\prime\prime}+\frac{1}{R^2} \rho)$ on $\partial B_R$. Finally, using a  smooth test function $v(x,y)$ in variational formulation of equation
\eqref{Hact_flow_lin_op} with boundary condition 
\eqref{Hact_flow_sec_BC} we obtain 
\begin{multline}
0=-\int_{\Omega(V_j)} \nabla\phi_j\cdot \nabla v\, dxdy+ \int_{\Omega(V_j)}(m_j -\zeta \phi_j)v dxdy+\int_{\partial \Omega(V_j)} \lambda_j \rho_j v ds +o(1)\\
=-\int_{B_R} \nabla\phi\cdot \nabla v\, dxdy+ \int_{B_R} (m -\zeta \phi)v dxdy+\int_{\partial {B_R}} \lambda \rho v ds +o(1),
	\end{multline}
implying that $\lambda\rho=\partial_r \phi$ on $\partial B_R$. Thus $\lambda$ is an eigenvalue of the operator $\mathcal{A}_{ss}$, contradicting the assumption.
Repeating this reasoning for $\delta/4$ in place of $\delta$, $\delta/16$ etc. we conclude that all eigenvalues $\lambda$ of ${\mathcal{A}} (V))$ with $|\lambda|<2\delta$ necessarily  converge to zero as $V\to 0$. 

To establish convergence of eigenvalues with multiplicities, consider for sufficiently small $V$ spectral projectors on the invariant subspace
 spanned by eigenvectors and generalized eigenvectors corresponding to eigenvalues $\lambda$ with $|\lambda|<\delta$,
\begin{equation}
\Pi_\delta(V):=\frac{1}{2\pi {\rm i}}\oint_{|\lambda|=\delta}(\lambda-\mathcal{A}(V))^{-1}\, d\lambda,
\end{equation}
restricted to $\mathcal{I}(V)$. Show that $\Pi_\delta(V)\bigl|_{\mathcal{I}(V)}\bigr.$ converges (in the sence desribed below)  to 
$$
\Pi_\delta(0) W=\frac{W_1}{\langle W_1, W_3^\ast\rangle}\Bigl|_{V=0}\Bigr.\langle W,W_3^\ast \rangle
%\frac{(-\Lambda \partial_x e^{\Phi-Vx},\cos (s/R))}{\langle,\rangle}\Bigl\{\int m \partial_V dxdy+\int_{\partial B_R} \rho ds\Bigr\}
$$
as $V\to 0$, where $W_3^\ast$ is a generalized eigenvector of ${\mathcal{A}}^\ast_{ss}$ corresponding to $W_2^\ast$,  i.e.  
${\mathcal{A}}^\ast_{ss}W_3^\ast=W_2^\ast$.
Namely, we claim that for any sequence $V_j\to 0$ and $(m_j,\rho_j)\in \mathcal{I}(V_j)$ such that $\rho_j\left(Rs/{L(V_j)}\right)\to \rho$ in $L^2(\partial B_R)$, and
$m_j\to m$ in $L^2(\mathbb{R}^2)$ (where we assume $m_j$ and $m$ continued by zero in $\mathbb{R}^2\setminus \Omega(V_j)$ and $\mathbb{R}^2\setminus B_R$, correspondingly) the sequence of pairs $(f_j,\varrho_j):=\Pi_\delta(V_j))(m_j,\rho_j)$ converges to $\Pi_\delta(0)(m,\rho)$ weakly in $H^2(B_R)\times H^3(\partial B_R)$, more precisely this convergence holds for functions $f_j$ exstended to $B_R$ (if necessary) by standard reflection through the normal  and 
$\varrho_j\left(Rs/{L(V_j)}\right)$. The proof of this claim follows exactly the lines above: we use Lemma \ref{VspomogLemma} to get uniform a priori bounds for  
$(\lambda-\mathcal{A}(V_j))^{-1})(m_j,\rho_j)$ in $H^2(\Omega (V_j))\times H^3(\partial \Omega (V_j))$ and then pass to limit in variational formulations of correspoding problems with smooth test functions. It follows that for suffciently small $V$ the dimension of the space $\Pi_\delta(V) {\mathcal{I}(V)}$ is at most one. Indeed, otherwise there exists a sequence $V_j\to 0$ and elements $W_j,\tilde W_j$ of $\Pi_\delta(V_j)\mathcal{I}(V_j)$ that are orthogonal and normalized to one in $L^2(\Omega(V_j))\times L^2(\partial\Omega(V_j))$. Since $W_j=\Pi_\delta(V_j)W_j$ and $\tilde W_j=\Pi_\delta(V_j)\tilde W_j$, after extracting a subsequence, if necessary, both $W_j$ and $\tilde W_j$ converge strongly in $L^2$-topology to limits belonging to ${\rm span} \{W_1\bigl|_{V=0}\}$, a contradiction. Furthernore,  we construct below
\begin{equation}
\label{approximateEV}
W=\tilde W + \theta W_1 \in  {\mathcal{I}(V)}\quad \text{with}\  \theta =O(V),
\end{equation}
out of the vectors $\tilde W$ from Lemma \ref{TestFunctionLem}, then we have
\begin{equation}
\label{shoditsya}
\Pi_\delta (V)W\mathop{\longrightarrow}_{V\to 0} \Pi_\delta(0) W_1=W_1\not =0.
\end{equation}
Therefore $\mathcal A(V)\bigl|_{\mathcal{I}(V)}$ has for sufficiently small $V$ exactly one simple eigenvalue $\lambda(V)$ with $|\lambda(V)|\leq 2\delta$, and 
$\lambda(V)\to 0$ as $V\to 0$. Moreover, %since $ \theta =O(V)$  
by virtue of Lemma \ref{TestFunctionLem} we get 
\begin{equation}
\label{nevyazka}
\|\mathcal{A} W-V^2\hat\lambda (V)W\|_{L^2(\Omega(V))\times L^2(\partial\Omega(V))}=O(V^2 M^\prime(V)). 
\end{equation}
Then, since 
\begin{multline*}
0=\frac{1}{2\pi{\rm i}}\oint_{|\lambda|=\delta} (\lambda-\mathcal{A}(V))^{-1}(\hat\lambda(V) V^2-\mathcal{A}(V))W\,d\lambda+
\frac{1}{2\pi{\rm i}}\oint_{|\lambda|=\delta} (\lambda-\mathcal{A}(V))^{-1}(\lambda-\hat\lambda(V) V^2)W\,d\lambda\\
=\Pi_\delta(V) (\hat \lambda(V) V^2 -\mathcal{A}(V)) W
+(\lambda(V)-\hat \lambda(V) V^2)\Pi_\delta (V)W,
\end{multline*}
we have
$$
|\lambda(V)-\hat\lambda(V) V^2|\leq \frac{ \|\mathcal{A}(V) W-\hat\lambda(V) V^2 W\|_{L^2(\Omega(V))\times L^2(\partial\Omega(V))}}
{
\|\Pi_\delta (V)W\|_{L^2(\Omega(V))\times L^2(\partial\Omega(V))}
} =O(V^2 M^\prime(V)).
$$

It remains to find $\theta=\theta(V)$ such that $\tilde W + \theta W_1 \in  {\mathcal{I}(V)}$. According to Lemma \ref{TestFunctionLem} we have
 $\langle\mathcal{A}(V) W-\hat\lambda(V) V^2 W, W_1^\ast\rangle_{L^2} = 0$, i.e. $ \hat\lambda(V) V^2\langle  W, W_1^\ast\rangle_{L^2} = \langle\mathcal{A}(V) W, W_1^\ast\rangle_{L^2}=\langle W, \mathcal{A}^\ast (V) W_1^\ast\rangle_{L^2}=0$. Thus we only need to chose $\theta$ such that 
$\theta \langle W_1, W_2^\ast\rangle_{L^2}=-\langle\tilde W, W_2^\ast\rangle_{L^2}$. Since 
$\langle  \mathcal{A}(V)\tilde W, W_2^\ast\rangle_{L^2}=\langle \tilde W,  W_1^\ast\rangle_{L^2}=0$ we have 
$\hat\lambda(V) V^2\langle  \tilde W, W_2^\ast\rangle_{L^2} = \langle\hat\lambda(V) V^2 \tilde W-\mathcal{A}(V)\tilde  W, W_2^\ast\rangle_{L^2}$, while by Lemma 
\ref{TestFunctionLem} and Lemma \ref{generLemma}
\begin{multline*}
\left|\langle\hat\lambda(V) V^2 \tilde W-\mathcal{A}(V)\tilde  W, W_2^\ast\rangle_{L^2}\right|\leq 
\left\|\hat\lambda(V) V^2 \tilde W-\mathcal{A}(V)\tilde  W\right\|_{L^2(\Omega(V))\times L^2(\partial\Omega(V))}
\\
\times
\left\|W_2^\ast\right\|_{L^2(\Omega(V))\times L^2(\partial\Omega(V))}=O(|M^\prime(V)|^2 V^2).
\end{multline*}
This leads to the bound $\langle\tilde W, W_2^\ast\rangle_{L^2}=O(VM^\prime(V))$, and noticing that 
$ \langle W_1, W_2^\ast\rangle_{L^2}= \langle \mathcal{A}(V)W_2, W_2^\ast\rangle_{L^2}=\langle W_2, W_1^\ast\rangle_{L^2}=M^\prime(V)$ we find the required bound
$|\theta|\leq C|V|$. Theorem \ref{GlavnayaTh} is completely proved.
%Q.E.D.
\end{proof}

%To construct $W$ we use the asymptotic expansions developed in Section ???. \textcolor{red}{Note that terms to the order $V^4$ make sence, and both third and forth terms extend to a bigger disk, in particular to $\Omega(V)$}   
%In order to satisfy conditions $\langle W,W_1^\ast\rangle=\langle
%W,W_1^\ast\rangle=0$ it suffices to guarantee that
%$$
%???
%$$
%
%
%\textcolor{red}
%{ Introduce a smooth family of diffeomorphic mappings $U(V): B_R\mapsto \Omega(V)$ preserving the symmetry with respect to the $x$-axis and such that $U(V)-{\rm Id}$, $\nabla (U(V)-{\rm Id})$, $\nabla^2 (U(V)-{\rm Id})$ converge to zero uniformly on 
%$\overline{B}_R$ as $V\to 0$. It follows from \ref{VspomogLemma} that, up to a subsequence,  $W_{j}$}

\section{Linear stability analysis of traveling wave solutions under  perturbations without symmetry assumptions}
\label{SectionNoSymmetry}

So far we assumed  reflectional symmetry with respect to the $x$-axis of traveling waves  (that are solutions $\phi=\Phi (x,y, V)$, $\Omega=\Omega (V)$ of \eqref{tw_Liouvtypeeq}--\eqref{addit_cond_tw}) and their perturbations. In this Section we study general perturbations that is with no symmetry assumptions on the pairs $(m,\rho)$ from the domain of the linearized operator $\mathcal{A}(V)$. Consider first the case $V=0$. It  corresponds to the  stationary radial solution  with the radius $R=R_0$ that satisfies the bifurcation conditions \eqref{VseVterminahF}.
%bifurcate to the
%traveling wave solutions, i.e. with the radius $R=R_0$ meeting the bifurcation conditions \eqref{VseVterminahF}. 
 Then the linearized operator  
$\mathcal{A}(0)=\mathcal{A}_{\rm ss}$ has the same eigenvalues as under the above symmetry assumptions, but multiplicities of nonradial eigenvectors double since the odd Fourier modes $m=\hat m_n(r)\sin n\varphi$, $\rho=\hat \rho_n \sin n\varphi$ are also considered. 
In particular, we have the zero eigenvalue with two eigenvectors corresponding to infinitesimal shifts 
\begin{equation}
\label{SHIFTSxy}
(m,\rho)=(0,\nu_x)=(0,\cos\varphi),\quad (m,\rho)=(0,\nu_y)= (0,\sin\varphi),
\end{equation}
and two generalized eigenvectors
\begin{equation}
(m,\rho)=(m_0(\Phi_V^0(x,y)-x),0),\quad (m,\rho)=(m_0(\Phi_V^0(y,x)-y),0)
\end{equation}
(cf. \eqref{TwderivativeinV1}  with $V=0$), where  $\Phi_V^0$ is the unique solution of \eqref{def_of_Phi_V}--\eqref{def_of_Phi_Vadditional}. For $V\not =0$ the generalized eigenspace of the zero eigenvalue is described in
\begin{prop} 
\label{Proposytsiya}
The operator $\mathcal{A}(V)$ defined in \eqref{Hact_flow_lin_op}-\eqref{Hmyosin_lin_BC} has the zero eigenvalue with two eigenvectors $W_1=(m_1,\rho_1)$, $W_3=(m_3,\rho_3)$ corresponding to  infinitesimal shifts,
\begin{equation}
\label{xy_shiftsTW}
m_1:=-\Lambda(V)\partial_x e^{\Phi-Vx}, \rho_1:=\nu_x,\quad m_3:=-\Lambda(V)\partial_y e^{\Phi-Vx},\ \rho_3:=\nu_y,
\end{equation}
the generalized eigenvector $W_2$ given by \eqref{TwderivativeinV1} (which is obtained by taking derivative of the traveling wave solution in $V$),  and the  following generalized eigenvector $W_4=(m_4,\rho_4)$,  
\begin{equation}
\label{rotations}
m_4:=-\frac{\Lambda(V)}{V}\partial_\varphi e^{\Phi-Vx}=\frac{\Lambda(V)}{V}(y\partial_x e^{\Phi-Vx}-x\partial_y e^{\Phi-Vx}),\quad \rho_4:=\frac{1}{V}(-y\nu_x+x\nu_y),
\end{equation}
 which represents infinitesimal rotations of the traveling wave solution.
The generalized eigenvectors $W_2$, $W_4$ satisfy $\mathcal{A}(V)W_2=W_1$, $\mathcal{A}(V)W_4=W_3$. 
\end{prop}

\begin{rem}
The  eigenvectors   $W_1$, $W_3$  appear due to  translational invariance of the problem	\eqref{actflow_in_terms_of_phi}--\eqref{myosin_bc_I}  under  shifts   of  the frame in $x$ and $y$ respectively.%, that is  a shift of the  initial   data by a given vector results in the same shift of  the solution for  all $t>0$.   
  This problem is also  invariant under rotations.
% in the following sense: a rotation of the frame results in the same  rotation of  the solution for  all $t>0$. 
 However, the equations \eqref{tw_Liouvtypeeq}--\eqref{addit_cond_tw} for the traveling waves solutions and corresponding linearized operator are written in the frame that translates with velocity $V$. That is why rotational invariance  gives rise to the generalized eigenvector $W_4$ 
% describing solution that linearly grows in time 
rather than true eigenvector.
%  constant in time. 
%This can be interpreted as follows: a true non-symmetric cell ``chooses its own direction to move, not necessarily the $V$ direction."   $W_2$   
\end{rem}

%\label{Hact_flow_lin_op}
%	\Delta \phi + m=\zeta \phi \quad\text{in}\ \Omega(V),%\Omega_{\rm tw},
%	\end{equation}
%	\begin{equation}
%	\label{Hact_flow_BC_dir_op}
%	\zeta(\phi+V\nu_x \rho) =p_{\ast}^\prime(| \Omega(V)|)\int_{\partial \Omega(V)}
%	\rho(s)ds
%	+\gamma(\rho^{\prime\prime}+\kappa^2 \rho)\quad\text{on}\   \partial\Omega(V),% \partial\Omega_{\rm tw},
%	\end{equation}
%	\begin{equation}
%	\label{Hact_flow_sec_BC}
%	\partial_t \rho=(\mathcal{A}(V)(m,\rho))_\rho:=\frac{\partial \phi}{\partial \nu}+\frac{\partial^2 \Phi}{\partial \nu^2}\rho-\left(\frac{\partial \Phi}{\partial \tau}+V \nu_y\right)\rho^\prime\quad\text{on}\ \partial \Omega(V),%\Omega_{\rm tw},
%	\end{equation}
%\begin{equation}
%	\label{Hmyosin_lin_eq}
%	\partial_t m=(\mathcal{A}(V)(m,\rho)))_m:=\Delta {m}+V\partial_x m -\div( \Lambda e^{\Phi-Vx} \nabla {\phi})
%	-\div(m\nabla \Phi) \quad\text{in}\   \Omega(V),%\Omega_{\rm tw}.
%	\end{equation}
%\begin{equation}
%	\label{Hmyosin_lin_BC}
%	\partial_\nu m+\Lambda e^{\Phi-Vx}\left(\frac{\partial^2 \Phi}{\partial \nu^2}\rho -\Bigl(\frac{\partial \Phi}{\partial \tau}+V \nu_y\Bigr)\rho^\prime\right)=0
%	\quad \text{on}\ \partial  \Omega(V).%\Omega_{\rm tw}.
\begin{proof} First we show that $\mathcal{A}(V)W_3=0$. Clearly \eqref{Hact_flow_lin_op} is satisfied with $\phi =-\partial_y \Phi$, also $(\mathcal{A}(V)W_3)_m$, given by \eqref{Hmyosin_lin_eq} equals zero identically. To verify that  $(\mathcal{A}(V)W_3)_\rho=0$  take the  tangential derivative of  the boundary condition 
$\partial_\nu \Phi=V\nu_x$ (this amounts to differentiating with respect to the arc length $s$):
\begin{equation}
\label{proizvodnayabc}
-\partial^2_{xx}\Phi\nu_x\nu_y+\partial^2_{xy}\Phi\nu_x^2-\partial_{xy}\Phi\nu_y^2+\partial^2_{yy}\Phi\nu_x\nu_y-\partial_{x}\Phi\kappa \nu_y
+\partial_{y}\Phi\kappa \nu_x=-V\kappa \nu_y,
\end{equation}
where we have used the Frenet-Serret formulas $\nu_x^\prime=-\kappa \nu_y$, $\nu_y^\prime=\kappa \nu_x$. Multiply this relation by $\nu_x$ and add to its both 
sides $\partial_{\nu\nu}\Phi \nu_y=(\partial^2_{xx}\Phi\nu_x^2+2\partial^2_{xy}\Phi\nu_x\nu_y+\partial^2_{yy}\Phi\nu_y^2)\nu_y$ to find
$$
0=\partial_\nu\partial_y\Phi-\partial_{\nu\nu}\Phi\, \nu_y+\kappa\nu_x (\partial_\tau\Phi+V\nu_y)= \partial_\nu\partial_y\Phi-\partial_{\nu\nu}\Phi\, \nu_y+ (\partial_\tau\Phi+V\nu_y)\nu_y^\prime. 
$$
The verification of  \eqref{Hmyosin_lin_BC}  is analogous, while to show \eqref{Hact_flow_BC_dir_op} we differentiate the equality $\zeta\Phi=p_{\ast}(|\Omega|)-\gamma\kappa$ in $s$ and obtain $\zeta\partial_{\tau}\Phi=-\gamma \kappa^\prime$. Then recalling that $\partial_{\nu}\Phi=V\nu_x$ we derive
\begin{equation*}
%\label{proverkanavshivost'}
-\zeta \partial_y\Phi=-\zeta(\partial_{\tau}\Phi\, \tau_y+\partial_\nu \Phi\, \nu_y)=\gamma\kappa^\prime \tau_y-\zeta V\nu_x\nu_y=\gamma(\nu_y^{\prime\prime}+\kappa^2\nu_y)-\zeta V\nu_x\nu_y.
\end{equation*}
Clearly, all the above arguments apply % with minor changes
 to show that $W_1$ is also an eigenvector.

%The verification of the equality $(\mathcal{A}(V)W_1)_\rho=0$ is the same, \textcolor{red}{while checking of \eqref{Hact_flow_BC_dir_op} and \eqref{Hmyosin_lin_BC} is similar}. 

We proceed now with the vector $W_4$. Take the derivative in $\varphi$ of $\Delta \Phi+\Lambda(V)e^{\Phi-Vx}=\zeta\Phi$ to obtain that \eqref{Hact_flow_lin_op}
is satisfied with $\phi=-\partial_\varphi\Phi$. Also, taking  the derivative in $\varphi$ of the equation 
$-V\partial_x e^{\Phi-Vx}=\Delta e^{\Phi-Vx} -{\rm div}(e^{\Phi-Vx}\nabla \Phi)$ and using the identities
$\partial_\varphi\partial_x \,\cdot = \partial_x\partial_\varphi \,\cdot -\partial_y \,\cdot$, 
$\partial_\varphi\partial_y \,\cdot = \partial_y\partial_\varphi \,\cdot +\partial_x \,\cdot$ we get 
$m_3=\Delta m_4 +V m_4-{\rm div}(m_4\nabla \Phi)+\Lambda(V){\rm div}(e^{\Phi-Vx}\nabla \partial_\varphi\Phi)$.  Considering equations on the boundary $\partial\Omega(V)$ we provide detailes
only for the equation \eqref{Hact_flow_sec_BC}, the verification of  \eqref{Hact_flow_BC_dir_op} and \eqref{Hmyosin_lin_BC}
being similar. Multiply the equation $\partial_\nu\partial_x \Phi-\partial_{\nu\nu}\Phi\, \nu_x+ (\partial_\tau\Phi+V\nu_y)\nu_x^\prime=0$  by $y$ and subtract 
the equation $-\partial_\nu\partial_y \Phi+\partial_{\nu\nu}\Phi\, \nu_y-(\partial_\tau\Phi+V\nu_y)\nu_y^\prime=0$ multiplied by $x$. After simple manipulations we obtain
\begin{equation*}
\begin{aligned}
0&=-\partial_\nu \partial_\varphi \Phi-\nu_y\partial_x\Phi+\nu_x\partial_y \Phi +\partial_{\nu\nu}\Phi (x\nu_y-y\nu_x)+ (\partial_\tau\Phi+V\nu_y)(y \nu_x^\prime-x\nu_y^\prime)\\
&=-\partial_\nu \partial_\varphi \Phi+\partial_\tau \Phi+\partial_{\nu\nu}\Phi (x\nu_y-y\nu_x)+ (\partial_\tau\Phi+V\nu_y)\left((y \nu_x-x\nu_y)^\prime-y^\prime\nu_x+x^\prime\nu_y\right).
\end{aligned}
\end{equation*}
Since $x^\prime=\tau_x=-\nu_y$ and $y^\prime=\tau_y=\nu_x$ we finally get
$$
-\partial_\nu \partial_\varphi \Phi+\partial_{\nu\nu}\Phi (x\nu_y-y\nu_x)- (\partial_\tau\Phi+V\nu_y)(x\nu_y-y \nu_x)^\prime=V\nu_y.
$$
Proposition \ref{Proposytsiya} is proved.\end{proof}

While Theorem  \ref{GlavnayaTh} describes spectrum of the operator $\mathcal{A}(V)$ in the space of vectors possessing symmetry with respect to the $x$-axis, in view of 
Proposition \ref{Proposytsiya} the multiplicity of the zero eigenvalue of $\mathcal{A}(V)$ in the (invariant) space of vectors anti symmetric with respect to the $x$-axis remains equal two for small $V\not=0$. Thus we have the following theorem which summarizes  spectral analysis of the operator $\mathcal{A}(V)$.

\begin{thm} \label{ThmNoSymmetry} Under assumptions of Theorem \ref{biftwtheorem} the spectrum of the operator $\mathcal{A}(V)$ for small $V$  has the following structure. The operator $\mathcal{A}(V)$ has zero eigenvalue with multiplicity four and the structure of  the corresponding generalized eigenspace is  described in  Proposition \ref{Proposytsiya}. There exists another small  simple eigenvalue $\lambda(V)$, whose asymptotic representation is 
given by \eqref{sovsemfinal_formula_for_lambda} (under additional assumption that $M^\prime(V)\not =0$ for small $V\not=0$).  All other eigenvalues 
are separated from  zero. Moreover, each of them but, possibly one,  has negative real parts. In particular, if condition   \eqref{areashrinkingprevention} on $k_e$ is satisfied, 
then among all nonzero eigenvalues only $\lambda(V)$ can have non negative real part. 
\end{thm}

%\textcolor{brown}{
Together with general formula \eqref{sovsemfinal_formula_for_lambda}  for the key eigenvalue $\lambda(V)$ we mention its particular case \eqref{formula_for_lambda_second}, which sheds light on the role of main physical parameters in the stability/instability of cell motion. In short, this theorem shows that the sign of $\lambda(V)$ determines stability of emerging traveling waves for small velocities.
Specifically,  if the inverse compressibility  coefficient $k_{\rm e}$ defined in \eqref{k_e}   satisfies \eqref{areashrinkingprevention}, then the eigenvalue $\lambda(V)$ given by \eqref{formula_for_lambda_second} determines stability of traveling waves  via sign  of the product of the two key physical quantities. First, the derivative of the eigenvalue $\frac{d E}{d M}\bigl|_{M=M_0}\bigr.$ ($ M_0=m_0 \pi R_0 ^2, m_0= p_{\ast}(\pi R_0^2)-\gamma/ R_0$) that describes the change of movabiliy of stationary solutions. Indeed, in view of  Remark \ref{eigenvalue_stationary}   $E\bigl|_{M=M_0}=0\bigr.$ and therefore positive/negative values of $\frac{d E}{d M}\bigl|_{M=M_0}\bigr.$ lead to instability/stability of the stationary solutions.  
Second, the value $M^{\prime \prime}(0)$   that determines weather total myosin mass $M$ increases or decreases with $V$ (since $M^\prime(0)=0$, see \eqref{myosin_density} and combine \eqref{myosin_derivative_zero} with the asymptotic formula for $\Lambda(V)$ in Remark \ref{rem:boundary}). 
% }

%\textcolor{brown}{
Finally, we  recall that   $\frac{d E}{d M}\bigl|_{M=M_0}\bigr.$ is  calculated in Lemma \ref{poleznyeformul'ki}, formula \eqref{formula_for_E} (near the  bifurcation point  the value of $M_0$  uniquely determines  the value of $R_0$ and vice versa due  to Remark \ref{R_M}).  We  proceed with calculation of $M^{\prime \prime}(0)$.  To this end we construct the first terms in the expansions of traveling wave solutions to \eqref{tw_Liouvtypeeq}-\eqref{addit_cond_tw} as a  power series in $V$ similarly to the Appendix A in \cite{BerFuhRyb2018}. Consider \eqref{expansion_of_Phi} and the leading term in the expansion  of $\tilde\Phi(x,y,V)$ 
\begin{equation}\label{eq:phi_expansion}
\tilde\Phi(x,y,V)=\Phi_{10}(r)+\Phi_{11}(r)\cos\varphi+\Phi_{12}(r)\cos2\varphi+O(V), \quad V \to 0,
\end{equation}
as well as  the similar expansion for the shape of traveling waves and total myosin mass: 
\begin{equation}\label{eq:rho_expansion}
\rho_{tw}=V^2\left(\rho_{10}+\rho_{12}\cos2\varphi+O(V)\right),  
\end{equation}
\begin{equation}\label{eq:M_expanson}
M(V)=M_0+M_1V^2+O(V^3),
\end{equation}
where $\rho_{10}$, $\rho_{12}$,  $M_0$ and $M_1$  are constants.
Proceeding as in  \cite{BerFuhRyb2018} (Appendix A) one can derive  elliptic boundary value problems in the disk $B_R$, to determine the  unknowns $\Phi_{10}$, $\Phi_{11}$, $\Phi_{12}$, $\rho_{10}$, $\rho_{12}$, and $M_1$. 
 % The  unknowns $\Phi_{10}$, $\Phi_{11}$, $\Phi_{12}$, $\rho_{10}$, $\rho_{12}$, and $M_1$   are  determined via solutions of  elliptic boundary value problems in the disk $B_R$, which can be derived  analogously to the appendix in \cite{BerFuhRyb2018}.
 % in the appendix in \cite{BerFuhRyb2018} to derive elliptic boundary value problems in the disk $B_R$ to determine $\Phi_{10}$, $\Phi_{11}$, $\Phi_{12}$, $\rho_{10}$, $\rho_{12}$, and $M_1$.
% Solving these problems numerically
 %These numerical implementaion of results of this paper  
Numerical solution of  these elliptic problems  explains the nature of the onset of motion via passing from unstable statinary solutions to stable traveling waves.  
In  particular, the numerics shows that if conditions of Theorem \ref{ThmNoSymmetry} and   the condition \eqref{areashrinkingprevention} hold, then   the bifurcation is {\it always the supercritical pitchfork}, since the real part of the key eigenvalue $\lambda(V) $ is always negative for sufficiently small $V$. This result agrees with
1D results from %\cite{RecPutTru2013},
\cite{RecPutTru2015}, where the normal mode analysis revealing the structure of the bifurcation has been performed for the first time. Besides it was observed in 1D case an interesting fact that there may be a re-entry behavior when the symmetry is restored for large enough amount of motors. The qestion if this is also true for the 2D model introduced in this work is open.

\bibliographystyle{plain}
\bibliography{references}

\end{document}